\documentclass[letterpaper, 11pt,  reqno]{amsart}
\usepackage[margin=1.2in,marginparwidth=1.5cm, marginparsep=0.5cm]{geometry}

\usepackage{amsmath, mathrsfs, amssymb,amscd,amsthm,amsxtra, esint}
\usepackage{mathtools}
\usepackage{tikz} 
\usepackage{color}
\usepackage[implicit=true]{hyperref}
\usepackage{cases}%%%
\mathtoolsset{showonlyrefs}
\usepackage{upgreek}
\usepackage{soul}

\allowdisplaybreaks[2]

\definecolor{gr}{rgb}   {0.,   0.69,   0.23 }
\definecolor{bl}{rgb}   {0.,   0.5,   1. }
\definecolor{mg}{rgb}   {0.85,  0.,    0.85}
\definecolor{yl}{rgb}   {0.8,  0.7,   0.}
\definecolor{or}{rgb}  {0.7,0.2,0.2}

\newcommand{\lvec}[1]{\vec{\mkern4mu #1}}

\usetikzlibrary{shapes.misc}
\usetikzlibrary{shapes.symbols}
\usetikzlibrary{shapes.geometric}

\tikzset{
	ddot/.style={circle,fill=white,draw=black,inner sep=0pt,minimum size=0.8mm},
	>=stealth,
	}

\tikzset{
	ddot2/.style={circle,fill=black,draw=black,inner sep=0pt,minimum size=0.8mm},
	>=stealth,
	}

\newtheorem{theorem}{Theorem} [section]

\newtheorem{lemma}[theorem]{Lemma}
\newtheorem{proposition}[theorem]{Proposition}
\newtheorem{remark}[theorem]{Remark}

\newtheorem{definition}[theorem]{Definition}
\newtheorem{corollary}[theorem]{Corollary}

\DeclareMathOperator{\com}{com}

\DeclareMathOperator{\Id}{Id}
\DeclareMathOperator{\sgn}{sgn}

\DeclareMathOperator{\Ker}{Ker}

\newcommand{\I}{\mathcal{I}}

\newcommand{\noi}{\noindent}
\newcommand{\Z}{\mathbb{Z}}
\newcommand{\R}{\mathbb{R}}

\newcommand{\T}{\mathbb{T}}
\newcommand{\q}{\mathbb{Q}}
\newcommand{\bul}{\bullet}

\let\P= \undefined
\newcommand{\P}{\mathbf{P}}

\renewcommand{\H}{\mathcal{H}}

\newcommand{\CC}{\mathcal{C}}

\renewcommand{\L}{\mathcal{L}}

\newcommand{\F}{\mathcal{F}}

\newcommand{\al}{\alpha}

\newcommand{\dl}{\delta}
\newcommand{\updl}{\updelta}

\newcommand{\Dl}{\Delta}
\newcommand{\eps}{\varepsilon}

\newcommand{\g}{\gamma}
\newcommand{\G}{\Gamma}
\newcommand{\ld}{\lambda}

\newcommand{\s}{\sigma}

\newcommand{\ft}{\widehat}

\newcommand{\dx}{\partial_x}
\newcommand{\dt}{\partial_t}

\newcommand{\ta}{\theta}

\newcommand{\les}{\lesssim}
\newcommand{\ges}{\gtrsim}

\newcommand{\jb}[1]
{\langle #1 \rangle}

\newcommand{\ind}{\mathbf 1}

\newcommand{\pa}{\partial}

\newcommand{\M}{\mathcal{M}}

\newcommand{\N}{\mathbb{N}}

\newcommand{\NN}{\mathcal{N}}
\newcommand{\cL}{\mathcal{L}}
\newcommand{\cC}{\mathcal{C}}

\newcommand{\cX}{\mathcal{X}}
\newcommand{\cY}{\mathcal{Y}}

\newcommand{\W}{\mathcal{W}}

\newcommand{\Lip}{\mathrm{Lip}}
\newcommand{\uw}{U^w}
\newcommand{\uwa}{U^{w_{\alpha}}}
\newcommand{\uu}{\mathbf{u}}
\newcommand{\vv}{\mathbf{v}}
\newcommand{\z}{\zeta}

\newcommand{\Ta}{\Theta}

\newcommand{\sub}{\substack}

\newcommand{\KDV}{\text{\rm KdV} }

\newcommand{\too}{\longrightarrow}

\newtheorem*{ackno}{Acknowledgements}

\numberwithin{equation}{section}
\numberwithin{theorem}{section}

\makeatletter
\@namedef{subjclassname@2020}{\textup{2020} Mathematics Subject Classification}
\makeatother

\begin{document}
\baselineskip = 14pt

\title[Well-posedness theory for a coupled modulated KdV system]{Diophantine conditions in well-posedness theory for a coupled modulated Korteweg-de Vries system}

%\date{\today}
\author[T~. Arthur, D.~Greco, and K.~Tsugawa]
{Thomas Arthur, Damiano Greco, and Kotaro Tsugawa}

\address{Thomas Arthur, School of Mathematics\\
The University of Edinburgh\\
and The Maxwell Institute for the Mathematical Sciences\\
James Clerk Maxwell Building\\
The King's Buildings\\
Peter Guthrie Tait Road\\
Edinburgh\\ 
EH9 3FD\\
United Kingdom}

\email{t.arthur@ed.ac.uk}

\address{Damiano Greco, School of Mathematics\\
The University of Edinburgh\\
and The Maxwell Institute for the Mathematical Sciences\\
James Clerk Maxwell Building\\
The King's Buildings\\
Peter Guthrie Tait Road\\
Edinburgh\\ 
EH9 3FD\\
United Kingdom}

\email{dgreco@ed.ac.uk}

\address{Kotaro Tsugawa, Department of Mathematics, Faculty of Science and Engineering, Chuo University, 
1-13-27 Kasuga, Bunkyo-ku, Tokyo, 112-8551, Japan}

\email{tsugawa@math.chuo-u.ac.jp}


\begin{abstract}
We study the well-posedness theory  of a coupled modulated Korteweg-de Vries (KdV) system on the circle with
a time non-homogeneous modulation acting on the linear dispersion term.
When the coupling parameter is equal to one, it has been recently proved that given any $s\in \mathbb{R}$, the resulting  modulated KdV system is
globally well-posed in $H^s(\mathbb{T})\times H^s(\mathbb{T})$, with a sufficiently irregular modulation. For couplings different from one, we use Diophantine conditions to characterize the resonances and prove that (under further  restrictions on the coupling constant) for any $s\in \mathbb{R}$ the coupled modulated KdV system is globally well-posed  in $H^s(\mathbb{T})\times H^s(\mathbb{T})$. This result differs from its unmodulated counterpart where it is known that global well-posedness holds for $s\ge s_*\in (5/7,1]$. 
\end{abstract}
\maketitle

\tableofcontents

\section{Introduction}
In this paper we consider a coupled modulated  Korteweg-de Vries (KdV)  system of the following form on the circle $\T=\R/(2\pi\Z)$:

\begin{equation}
\label{cmkdv}
\begin{cases}
\dt u+  \dx^3 u \cdot \dt w=\frac{1}{2}\dx v^2,\\
\dt v+  \al\dx^3 v \cdot \dt w=\dx (uv),\\
(u,v)_{|t=0}=(u_0,v_0),
\end{cases}
\end{equation}
 where 
 %$\T=\R/(2\pi\Z)$\footnote{By convention, $\T$ is endowed with the normalized Lebesgue measure $dx_{\T}=(2\pi)^{-1}dx$ such that we do not
%need to consider factors involving $2\pi$.}, 
$\alpha\in \textcolor{black}{(0,4]}$,
$u,v \colon \R_+\times \T \to \R$  are the unknown, and 
 $w \colon \R_+\to\R$ is an arbitrary continuous function of time, 
called a {\it modulation}.
In the unmodulated case (i.e. $w(t)= t$), the system \eqref{cmkdv} was introduced by Majda and Biello \cite{MB} as a reduced asymptotic model for nonlinear resonant interactions between long-wavelength equatorial Rossby waves and barotropic Rossby waves with significant mid-latitude projection, in the presence of suitably sheared zonal mean flows. In the future we may refer to \eqref{cmkdv} as to the modulated Majda–Biello system.
In the literature, existence results for the system \eqref{cmkdv} are limited to particular regimes. The case 
$\alpha=1$ and generic continuous modulation $w$  has been analyzed in \cite{CGLLO, GLO},  whereas the attention on general values of 
$\alpha$ has so far been restricted to the case of  homogeneous modulation; see \cite{LWP_diophantine, GWP_dioph}. The main novelty of this paper is to extend these results by developing a unified framework that applies to general $\alpha$ and general $w$.

\subsection{Background: well-posedness results for \texorpdfstring{$\alpha=1$}{alpha equal 1}}
\noi To begin with, let us recall what is known in the regime $\alpha=1$ and $w$ a generic  continuous modulation.
%When  studying of the above system, two \textit{resonance equations} will play a crucial role:
%Throughout the paper, we fix $\alpha \in (0,4)$.
%Hereafter, for any constants that depend on $\alpha$, we will not explicitly indicate this dependence.
%Recall the definition
%When $n=n_1+n_2$, it follows that
%\begin{align}
%&\Xi^{(1)}(n, n_1, n_2)=n(3\alpha n_1^2-3\alpha nn_1+(\alpha-1)n^2)=3\alpha n(n_1-cn)(n_1-c_*n),\label{eq_factori1}\\
%&\Xi^{(2)}(n, n_1, n_2)=\Xi^{(1)}(-n_1, -n, n_2)=-3\alpha n_1(n-cn_1)(n-c_* n_1),\label{eq_factori2}
%\end{align}
In this direction, the second author in collaboration with Liu and Oh \cite{GLO} continued the analysis started in \cite{CGLLO} of the following modulated Korteweg-de Vries equation on the circle:  
\begin{align} 
\label{maineq}
\begin{cases}
    \dt u + \dx^3 u \cdot \dt w = \dx u^2, \\
    u|_{t=0} = u_0.
\end{cases} 
\end{align}
Such an equation arises naturally as a model for weakly nonlinear long waves in an
inhomogeneous waveguide; see \cite{CMG,HZ}. Originally, de Bouard and Debussche and, Debussche and Tsutsumi \cite{BD, DT} started to study nonlinear dispersive equations with Brownian modulation; see also \cite{Dub}. In their cases, the equations could be understood as stochastic partial differential equations (SPDEs) via Ito’s calculus.
However, in the case of a generic modulation, such as \eqref{maineq} and \eqref{cmkdv},  the interpretation as an Ito or Stratonovich SPDE is not possible. To overcome such an issue, Chouk and Gubinelli \cite{CG1} first, followed by subsequent works in \cite{CGLLO, GLLO}, developed a different approach based on nonlinear Young integration; see \cite{FV, G,Y}. The starting point is the following notion  of {\it irregularity} introduced in \cite{R.CG16}.
 \begin{definition}
\label{DEF:ir}
\rm

Let $\rho>0$ and  $0 < \g < 1$.
Given $T > 0$, we say that a function $w\in C([0,T];\R)$ is $(\rho,\g)$-irregular 
on the time interval $[0, T]$  if we have
\begin{align}
\|\Phi^w\|_{  \W^{\rho,\g}_T} 
:= \sup_{a\in \R} \sup_{0\leq r < t\leq T} \langle a \rangle^\rho \frac{|\Phi^w_{t,r}(a)|}{|t-r|^\g} 
< \infty, 
\label{rho1}
\end{align}

\noi
where 
\begin{align}
\Phi^w_{t,r}(a)
=\int_r^t e^{i  a w(t') } d t'.
\label{rho2}
\end{align}

\noi
We say that $w$ is 
$(\rho,\g)$-irregular 
on $\R_+$  if it is 
$(\rho,\g)$-irregular 
on $[0, T]$ for each finite $T > 0$.
%Given an interval $I = [a, b] \subset \R_+$, 
%we say that $w$ is 
%$(\rho,\g)$-irregular 
%on the interval $I$ if $w(\cdot + a)$ is 
%$(\rho,\g)$-irregular 
%on $[0, b-a]$.
\end{definition}
Next, we briefly recall the approach used in \cite{CG1,CGLLO} to study \eqref{maineq}. The same approach will be indeed used for \eqref{cmkdv}. The first observation is that, since $w$ is only assumed to be continuous, the modulated KdV equation \eqref{maineq} is merely formal.
On the other hand, the corresponding
%As a result, the modulated dispersive equation~\eqref{maineq} (as well as  the modulated KdV system \eqref{cmkdv}) is then understood via the following 
 Duhamel formulation (= mild formulation) involves $w$ only and not its time derivative:
\begin{align}
\label{mild32} 
    {u} (t) = U^{w}(t)u_0 + U^{w}(t)\int_{0}^{t} U^{w} (t')^{-1} \pa_{x} \big( {u} (t'))^2 \big) d t',
\end{align} 
where $U^{w} (t) = e^{- w (t) \pa_x^3}$ denotes the modulated linear propagator such that $U^{w} (0) = \rm{Id}$\footnote{Note that this is not an additional restriction since throughout the paper we will assume $w(0) = 0$. This condition does not affect our analysis since only the time derivative $\dt w$
 appears in \eqref{cmkdv}-\eqref{maineq}.}.  If we set $\uu$ to be the modulated interaction representation of the unknown $u$ defined as
 \begin{align}\label{represeintro}
 \mathbf{u} (t) = U^{w} (t)^{-1} u(t),
 \end{align}
 the Duhamel formulation \eqref{mild32} becomes
\begin{align}
\label{mild3} 
    \mathbf{u} (t) = u_0 + \int_{0}^{t} U^{w} (t')^{-1} \pa_{x} \big( ( U^{w} (t') \mathbf{u} (t'))^2 \big) d t'.
\end{align} 
As already underlined in \cite{CGLLO}, the key reason for passing from \eqref{mild32} to \eqref{mild3}, is the absence of $U^w(t)$ in front of the integral, which allows us to exploit temporal regularity of the interaction representation $\uu$. The main idea of \cite{CG1,CGLLO} is to apply the sewing lemma (see Lemma \ref{LEM:sew}) in order to give meaning to the integral in the right hand side of \eqref{mild3} 
%\noi  By applying the sewing lemma (see Lemma \ref{LEM:sew})
%the integral in the right hand-side of \eqref{mild3} 
%is understood
as a nonlinear Young integral denoted by $\I^{X}(\uu)$ for a suitable bilinear (two parameter) driver $X(=X_{t,r})$  (see \cite[eq. (3.3)]{CGLLO}),
%$\mathcal{I}^{X}(\uu)$
% for a suitable driver $X$; see \cite[eq. (2.29)]{GLO}. 
and to recast  \eqref{mild3} as a nonlinear Young differential equation (YDE):
\begin{align}\label{YDEintro}
\uu(t)=u_0+\I^{X}(\uu)(t).
\end{align}
See Subsection \ref{SUBSEC:Y2} for a brief review on the the construction of nonlinear Young integrals.
%Similar considerations apply to the KdV system in \eqref{cmkdv}.
 In particular, since the YDE \eqref{YDEintro} is solved by a contraction argument, thus 
exhibiting a semilinear nature of the equation, we refer to local well-posedness
as {\it semilinear local well-posedness}.

%In particular, we will prove existence (and uniqueness)  in a suitable subspace of $C([0,T]; H^s(\T))$, cf., \cite[Definition 1.2]{CGLLO}.  Such a uniqueness statement is then called \textit{conditional}. We refer the reader to \cite{GLLO} for
%different solution theory based on adapting the normal form method to the modulated setting. 
% An analogous  comment
%applies to the other modulated equations considered in this paper such as \eqref{scaleKdV}. %In view of Lemma \ref{LEM:OBS1}, Lemma \ref{LEM:sew} and Proposition \ref{PROP:young1},   

\remark \label{examples}\rm
Non trivial examples of $(\rho,\g)$-irregular functions can be found for example in~\cite{R.CG16,GG}. In particular, given  any $\delta \in (0,1)$, a generic $\dl$-H\"older continuous function $w \in C^\delta([0,1];\R^d)$ is $(\rho,\g)$-irregular for any $\rho < \frac1{2\delta}$ with some $\gamma = \gamma(\rho) \in (\frac 12,1)$. Here, ``generic'' is to be understood according to the notion of {\it prevalence}; see~\cite[p.\,2418 and Theorem 3.1]{GG} for details and the references therein. 
 %is $(\rho,\g)$-irregular for any $\rho < \frac1{2\delta}$ with some $\gamma = \gamma(\rho) \in (\frac 12,1)$
%In particular, a fractional Brownian motion $\{W_t\}_{t\in \R_+}$ 
%of Hurst index $H\in(0,1)$, is $(\rho,\gamma)$-irregular 
%if $\rho < \frac{1}{2H}$,  
%there exists $\frac 12 < \g < 1$
%such that,  with probability one,  the sample paths of 
%$W$ are $(\rho,\g)$-irregular on $\R_+$.

%\begin{oldtheorem}
%\label{THM:B}
%Let  $d \ge 1$.
%Given  any $\delta \in (0,1)$, a generic $\dl$-H\"older continuous function $w \in C^\delta([0,1];\R^d)$ 
 %is $(\rho,\g)$-irregular for any $\rho < \frac1{2\delta}$ with some $\gamma = \gamma(\rho) \in (\frac 12,1)$.
%\end{oldtheorem}

%Notice that Theorem \ref{THM:A}
%shows that fractional Brownian motions can be $(\rho, \g)$-irregular for arbitrarily large $\rho$, 
%while Theorem \ref{THM:B} actually shows that $(\rho, \g)$-irregularity is a generic property of H\"older functions of sufficiently low regularity,  
%where ``generic'' is to be understood according to the notion of {\it prevalence}; see~\cite[p.\,2418 and Theorem 3.1]{GG} for details and the references therein. 

 We can summarize the well-posedness results for \eqref{maineq} as follows: 
 \begin{theorem}[\textnormal{\cite[Theorem 1.3-(i)]{CGLLO}, \cite[Theorem 1.4]{GLO}}]
 \label{THM:1}
Given $\rho \ge\frac 12$,  $\frac12< \g < 1$, and $T> 0$, 
let  $w$ be $(\rho,\g)$-irregular on $[0, T]$ in the sense of Definition~\ref{DEF:ir}.

\smallskip

\noi
\textup{(i)}
Suppose that $(u_0,v_0)$ have mean zero,  $\rho \ge \frac 12$ and $s\in \R$
satisfy one of the following conditions\textup{:}
\begin{align}
\begin{split}
\textup{(i.a)} &\ \  \tfrac 12 \le \rho \le \tfrac 34 \quad \text{and} \quad 
s > \tfrac 32 - 3 \rho, \\
\textup{(i.b)} &\ \  \rho > \tfrac 34\quad  \text{and} \quad  s \ge - \rho.
\end{split}
\label{reg1}
\end{align}

\noi
Then, the modulated KdV equation \eqref{maineq}
on  $\T$
is semilinearly locally well-posed in $H^s(\T)$.

\smallskip

\noi
\textup{(ii)}
In addition, suppose that  $s\in \R$ satisfy one of the following conditions\textup{:}
\begin{align}
\begin{split}
\textup{(ii.a)} & \  \tfrac{1}{2} < \rho \le \tfrac{3\g}{6\g-2} \quad \text{and} \quad s > \big(\tfrac{3}{2}-3\rho \big) \g, \\ 
\textup{(ii.b)} & \  \tfrac{3\g}{6\g-2}<\rho \leq \tfrac {3}{2}\quad \text{and} \quad s \ge -\rho,  \\ 
\textup{(ii.c)} & \ \tfrac{1}{2}<\g<\tfrac{\sqrt{5}-1}{2},\\
&\ \textup{(ii.c.1)} \  \tfrac{3}{2}\le\rho<\tfrac{3(1-\g)^2}{2(-\g^2-\g+1)}, \quad \text{and} \quad s \ge -\rho,  \\  
& \text{or} \\ 
&\ \textup{(ii.c.2)} \ \rho \geq \tfrac{3(1-\g)^2}{2(-\g^2-\g+1)}, \quad \text{and} \quad s > -\tfrac {4\rho(2\g-1)+3(1-\g)^2}{2(-\g^2+3\g-1)}, \\
\textup{(ii.d)} & \ \ \tfrac {\sqrt{5} - 1}{2}\le  \gamma < 1,\quad \rho\ge \tfrac{3}{2},\quad \text{and}\quad s\ge -\rho. 
\end{split} 
\label{mcon1}
\end{align}

\noi
Then, 
 the modulated KdV equation \eqref{maineq}
on $\T$
is globally  well-posed in $H^s(\T)$.
\end{theorem}

\remark \rm \label{sol} To be precise, in Theorem \ref{THM:1},
%for all the well-posedness results for modulated dispersive equations/systems, 
we say that $u\in C([0,\tau];H^s(\T))$ solves \eqref{maineq}
%(respectively $(u,v)\in   C([0,\tau];H^s(\T)\times H^s(\T))$) is a solution
  on $[0,\tau]\times \T$ if its  associated modulated interaction representation $\uu$ defined in \eqref{represeintro} solves the associated nonlinear Young differential equation 
  \eqref{YDEintro} on $[0,\tau]\times \T$, where  $\tau>0$ denotes the existence time; see also \cite[Definition 1.2, Remark 1.4]{CGLLO}. A similar comment applies to the statement of Theorem \ref{THM:11} below. 
\medskip

Note that,
%, Theorem \ref{THM:1} exhibits a strong regularization-by-noise phenomenon which differs from its unmodulated counterpart. Indeed, 
unlike  the (unmodulated) KdV equation on $\T$
\begin{align*}
 \dt u + \dx^3 u = \dx u^2,
\end{align*}
which is known to be ill-posed in $H^{s}(\T)$ for $s<-1$, see  \cite{KGL,CKSTT04,M, KP, KV19}, given any $s\in \R$, the modulated KdV \eqref{maineq} is globally well-posed  in $H^{s}(\T)$
provided that $w$ is sufficiently irregular (i.e., sufficiently large $\rho=\rho(s)$), Hence,  
in connection with many studies such as \cite{FHLN,GG0,FGP,GP,RT,T}, Theorem \ref{THM:1} established a regularization-by-noise phenomenon
for the modulated KdV \eqref{maineq}.
%That is, they proved that given any $s\in \R$, the modulated
%KdV \eqref{maineq} with a sufficiently irregular modulation $w$ (i.e., sufficiently large $\rho$) is globally well-posed in $H^s(\T)$. 
%In particular, their results, as well as our main result in Theorem \ref{THM:11}, apply to the case where the modulation is given by a fractional Brownian motion or a generic $\dl$-H\"{o}lder continuous function.
We also stress that, prior to the work \cite{CGLLO}, regularization by noise in the context of dispersive PDEs was known only for probabilistic well- posedness with random initial data and / or additive noises of super-critical regularity (such as \cite{B096, BT, CO, BOP, P, OP, GKO24, DNY, DNY2, OKT, B, OKT2, BDNY}).

Since the abstract setting for studying \eqref{cmkdv}  is similar to that of \eqref{maineq}, for the convenience of the reader, we briefly go over the main steps for the proof of Theorem \ref{THM:1}. Regarding the local result (Theorem \ref{THM:1}-(i)),  one of the key ingredients  is to obtain  suitable bilinear estimates for the driver $X$ whose Fourier transform satisfies
\begin{align}\label{XkdVintro}
\F ( X_{t,r}(f_1,f_2)) (n)= in\sum_{ \substack{n_1, n_2 \in \Z\\n = n_1+n_2}}\Phi^{w}_{t,r}  (\Xi_{\textup{KdV}}(\bar n))  \ft f_1(n_1)  \ft f_2 (n_2),
\end{align}
where $f_1,f_2$ are smooth functions on $\T$, and $\Xi_\KDV$ is 
 given by 
\begin{align}
\Xi_\KDV (\bar{n}) 
:= - n^3+ n_1^3 +n_2^3=-3n n_1n_2,\quad n=n_1+n_2.
\label{K3}
\end{align} 
See \cite[eq. (4.4), Proposition 4.1]{CGLLO} for the precise statements. Hence, in view of \eqref{XkdVintro}, studying the regularity of the driver $X$ essentially reduces to studying the decay of the Fourier coefficient $\Phi^{w}_{t,r}  (\Xi_{\textup{KdV}}(\bar n))$. 
%In order to so, as in the unmodulated case $w(t)\equiv t$,  a central role is played by the { resonance function} 
While in the unmodulated framework the resonance function in \eqref{K3} allows one to gain essentially $3/2$ derivatives \footnote{Under mean zero assumptions.}; see e.g., \cite{KGL}, if $w$ is $(\rho,\g)$-irregular on $[0,T]$ in the sense of Definition \ref{DEF:ir}, the bound
\begin{align}\label{eq115intro}
|\Phi^{w}_{t,r}  (\Xi_{\textup{KdV}}(\bar n))\les |t-r|^{\gamma}\jb{\Xi_{\textup{KdV}}(\bar n)}^{-\rho},\quad 0\le r< t\le T,
\end{align}
yields smoothing of arbitrarily high order\footnote{The smoothing of arbitrary order is understood in the sense that we gain an arbitrary amount of derivatives by taking $\rho$ sufficiently large.
%In particular, all the lower bounds on $s$ provided in Theorem \ref{THM:1} tends to $-\infty$ as $\rho$ diverges.
}. This is why, in contrast with the classic theory of the KdV equation initiated by Bourgain ~\cite{BO93} and followed by \cite{KGL, Kish}, where global well-posedness is known to be true for $s\ge -\frac{1}{2}$ (actually for $s\ge -1$ by further employing complete integrability \cite{KV19, KP}), Theorem \ref{THM:1} exhibits well-posedness even in the regimes where the standard KdV is ill-posed. 
 
Next, we move our attention to the global well-posedness result in Theorem \ref{THM:1}-(ii). For $s\ge 0$ the conclusion follows by combining the conservation of the $L^2$-norm with persistency of regularity \footnote{Note that the proof of the $L^2$-conservation must be adapted to the modulated setting as well, see Proposition \ref{PROP:L2}.}; see \cite[Proposition 6.1]{CGLLO}.  For $s<0$, the strategy is to adapt the so-called $I$-method (= the method of almost conservation laws), introduced by
Colliander, Keel, Staffilani, Takaoka, and Tao \cite{CKSTT03} to the modulated setting. Initially, in \cite{CGLLO} the authors applied the usual KdV scaling to the unknown $u$. This led to restricting to the subcritical regime $s>-\frac{3}{2}$, independently of the size of $\rho$. Recently,      in \cite{GLO}, by employing a different scaling, the authors removed the lower barrier $-\frac{3}{2}$ of \cite[Theorem 1.3-(ii)]{CGLLO},  leading to conditions for local and global well-posedness that share the same structural form\footnote{Here, by sharing the same structural form we understand that, in analogy with \eqref{reg1}, the lower bounds for $s$ in \eqref{mcon1} go to $-\infty$ as $\rho$ diverges.}; see Subsection \ref{Subsec4} for the details  of the scaling argument.  %The same strategy will be applied to our modulated KdV system \eqref{cmkdv}.

Since the same outcome of Theorem \ref{THM:1} applies to  the coupled modulated system \eqref{cmkdv} with $\alpha=1$, our main purpose is study \eqref{cmkdv} with $\alpha\neq 1$. To be precise, we will focus on the range  $\alpha\in (0,4)\setminus\{1\}$; see Remarks \ref{introrestrict}-\ref{large_a_intro}.
 
\subsection{Main results: well-posedness for \texorpdfstring{$\alpha\in(0,4)\setminus\{1\}$}{alpha in (0,4) except 1}}
In this subsection, we focus on well-posedness theory for \eqref{cmkdv} when $\alpha\neq 1$ by extending the results of \cite{GWP_dioph} to the modulated setting.  In particular, our main goal is to show a regularization-by-noise phenomenon as in Theorem \ref{THM:1}. Although the role of a different coefficient might appear a trivial modification, it in fact produces some structural differences we will highlight in the sequel. 
%First of all, $\alpha\neq 1$, 
%$the lack of symmetry between the  dispersion terms %of \eqref{cmkdv}, 
%produces  a family of linear semigroups:

In analogy with \eqref{maineq} and \eqref{mild3},  we begin by recasting \eqref{cmkdv} via the mild formulation:
 \begin{equation}
\label{mild2bis}
\left\{
\begin{aligned}
\uu(t) &= u_0 + \frac{1}{2}\int_0^t \uw(t')^{-1}\big(\partial_x(\uwa(t') \vv(t'))^2\big)\,dt', \\
\vv(t) &= v_0 + \int_0^t \uwa(t')^{-1}\big(\partial_x\big((\uw(t') \uu(t'))(\uwa(t') \vv(t'))\big)\big)\,dt',
\end{aligned}
\right.
\end{equation}
where  \begin{align}
(\uu, \vv)(t)=(\uw(t)^{-1}u(t), \uwa(t)^{-1}v(t)),% \quad U^{w_\al}(t) = e^{- \al w(t)\partial_x^3 },  \quad U^{w_\al}(0) = \Id.
\label{int1} 
\end{align}
and \begin{align}\label{semigroup_a}
U^{w_\alpha}(t):=e^{-\alpha w(t)\dx^3}. 
\end{align}
For simplicity we set $w_\alpha|_{\alpha=1}=w.$
%To better explain the role of the coupling parameter $\alpha$, we start by focusing on the unmodulated version of \eqref{mild2bis} (i.e., with $w(t)=t$) that was studied in \cite{LWP_diophantine, GWP_dioph}. 
Then, as we have already explained above, the first step to give a meaning to the integrals in the right hand side of \eqref{mild2bis} as nonlinear Young integrals, is to suitably introduce two drivers $X^j, j=1,2$, see \eqref{driver1}, and prove appropriate bilinear estimates (Lemma \ref{lem_nonlinsmooth}).  In order to do so, similarly to \eqref{eq115intro}, the main tool is the bound 
\begin{align}\label{eq116intro}
|\Phi^{w}_{t,r}  (\Xi^{(j)}(\bar n))|\les |t-r|^{\gamma}\jb {\Xi^{(j)}(\bar n)}^{-\rho}
\end{align}
where $\Xi^{(j)},\, j=1,2$,  are defined by  
\begin{equation}\label{res2intro}
(\Xi^{(1)}(\bar n), \Xi^{(2)}(\bar n)):=(\alpha n_1^3+\alpha n_2^3-n^3,n_1^3+\alpha n_2^3-\alpha n^3),  \quad {n} = n_1+n_2.
\end{equation}
% In particular, when considering the case $\alpha=1$, as in the KdV case,  the  algebraic identity
% \begin{align}
% \Xi_\KDV (\bar n)=3nn_1n_2,\quad  n=n_1+n_2, 
% \end{align}   
%  essentially allows us to gain $3/2$ derivatives (for $n,n_1,n_2\neq 0$). In the classic local well-posedness theory for the KdV equation this aspect was crucial in order to obtain the bilinear estimate \footnote{By assuming $u,v$ are mean zero.}:
%  \begin{align}
%  \|\dx (uv)\|_{Z^{s}(\T\times \R)}\les \|u\|_{X^{s,\frac{1}{2}}}\|v\|_{X^{s,\frac{1}{2}}},\quad s\ge -\frac{1}{2}. 
%\end{align} 
\noi Let us focus first on $\Xi^{(1)}$. Under the condition $n=n_1+n_2$, it follows that
\begin{align}\label{eq_factori1}
&\Xi^{(1)}(\bar n)=n(3\alpha n_1^2-3\alpha nn_1+(\alpha-1)n^2)=3\alpha n(n_1-cn)(n_1-c_*n),
\end{align}
where
\begin{equation}
\label{roots1}
    c := \frac{1}{2} - \frac{\sqrt{-3+12\al^{-1}}}{6}, \qquad c_* := \frac{1}{2} + \frac{\sqrt{-3+12\al^{-1}}}{6}.
\end{equation}
Note that $c,c_*\in \R$ (and $c+c_*=1$) if and only if $\alpha\in (0,4]$. While in the case of \eqref{maineq} mean zero assumptions are enough to ensure that the resonance function $\Xi_{\textup{KdV}}$ in \eqref{K3} is far from zero (and so \eqref{eq115intro} provides smoothing possibly by taking $\rho$ large), the resonance function $\Xi^{(1)}$ might vanish infinitely many times,  not allowing a gain of smoothness from \eqref{eq116intro} (independently of the size of $\rho)$. %Due to the derivative nonlinearity, we only consider $n \neq 0$.
 %In particular, there exist $c_1,c_2\in \mathbb{R}$ (with $c_1+c_2=1$) such that the triple $(n_1,n_2,n)=(c_1n,c_2 n,n)$ is a root for $\Phi_\alpha$, causing \textit{resonance}. 
In particular, as already explained in \cite{LWP_diophantine, GWP_dioph}, two cases can occur. 
 
 \noi\textup(i) If $c\in \mathbb{Q}$ (and so also $c_*\in \mathbb{Q})$\footnote{Note that this happens if and only if $\alpha=\frac{12}{k^2+3}$ for some $k\in \mathbb{Q}$.}, $\Xi^{(1)}$ vanishes infinitely many times, i.e., there are infinitely many $n\in \mathbb{Z}\setminus\{0\}$ causing resonance. As a result, we do not expect any gain of smoothness in this regime.
 
\noi \textup(ii)  If $c\notin \mathbb{Q}$ (and so also $c_*\notin \mathbb{Q})$ and $n\in \Z\setminus\{0\}$, then  $\Xi^{(1)}$ is never zero. Nevertheless, it can be arbitrarily close to zero, according to ``how close $c$ is to a rational number". 
 %We refer the reader to \cite{LWP_diophantine, GWP_dioph} for a deep introduction to this topic.
Such a closeness is quantified by an index associated to $c$($=:\nu_c$) introduced by Arnold \cite{Arnold}, see Definition \ref{inf_index}. In particular, if such an index is small enough (less than one), $\jb{\Xi^{(1)}(\bar{n})}$ exhibits nontrivial lower bounds; see Lemma \ref{lem_res}. On the other hand, if such  an index is large (bigger than  one), the lack of a nontrivial  lower bound suppresses the regularization by noise mechanism, and the resulting behavior coincides with that of the unmodulated one; see Remark \ref{ind_large}.   %as well as suitable smallness of the index $\nu_c$ (which in particular prevents $\Xi^{(1)}$ and $\Xi^{(2)}$ to vanish)
For this reason, in the sequel we focus on the case, $\nu_{c} < 1$, where the irregularity of the modulation $w$ plays a role. 

\remark \rm \label{introrestrict} In particular, since $c\in \mathbb{Q}$ when $\alpha=4$, this case is not considered in Theorems \ref{THM:11}-\ref{THM:G1}.
\begin{definition}\label{inf_index}
A real number $r$ is called of type $(K,\nu) $ (or simply of type $\nu$) if there exists positive $K$ and $\nu$ such that for all pairs of integers $(m,n)$, we have 
\begin{align}\label{index}
\left| r - \frac{m}{n}\right| \geq \frac{K}{|n|^{2+\nu}}.
\end{align}
Moreover, given a real number $r$ define the minimal type index $\nu_{r}$ of $r$ by 
\begin{align*}
\nu_{r }: = \begin{cases}
\infty, ~ \text{if} ~ r \in \q \\
\inf \{\nu > 0 : r ~ \text{is of type } \nu \}, ~ \text{if } ~ r \notin\q.
\end{cases}
\end{align*}
\end{definition} 

\remark\rm\label{v_zero} Note that  $\nu_{r}\ge 0$ and, 
as observed in \cite{LWP_diophantine}, $\nu_{r}=0$ for almost every $r\in \R$.
\smallskip

Next, in view of the relation 
\begin{align}\label{relphase}
\Xi^{(1)}(n,n_1,n_2)=\Xi^{(2)}(-n_1,-n,n_2),
\end{align}
we can apply similar considerations  to the second resonance function $\Xi^{(2)}$ (which vanishes when $n_1=0$).  As a result,  we study the well-posedness theory for \eqref{mild2bis} by assuming $\ft{u}_0(0)=0$. 
Under this restriction, as well as the above mentioned smallness of the index associated to $c$, we prove that the regularization phenomenon occur; see Theorem \ref{THM:11}. Namely, for any fixed $s\in \R$ we  obtain well-posedness in $H^s(\T)\times H^s(\T)$ provided $w$ is sufficiently irregular; see Proposition \ref{PROP:main}, Corollary \ref{GWP_2} and Lemma \ref{lem_region2} for the precise statements. 

%Before properly stating our main contributions, we recall the definition of the above mentioned index. 

%we present their unmodulated counterparts proved in  \cite{LWP_diophantine, GWP_dioph} by first 

\smallskip

\begin{theorem}\label{THM:11}
Given $\alpha\in (0,4)\setminus\{1\}$,  $\rho \ge\frac 12$,  $\frac12< \g < 1$, $c$ as in \eqref{roots1}, $\nu_c$ as in Definition \ref{inf_index}, $0<\eps<1-\nu_{c}$, and $T > 0$, let
$w$ be $(\rho,\g)$-irregular on $[0, T]$ in the sense of Definition~\ref{DEF:ir}, and $u_0$ be with mean zero. 
\smallskip

\noi
\textup{(i) (local well-posedness).}
Suppose that $s\in \R$ satisfies 
\begin{align}\label{condsintro}
s\ge 1-\rho(1-\nu_{c}-\eps).
\end{align} 
Then, the modulated KdV system \eqref{cmkdv}
is semilinearly locally well-posed in $H^s(\T)\times H^s(\T)$.
\smallskip

\smallskip

\noi
\textup{(ii) (nonlinear smoothing).}
In addition to the hypotheses of Part \textup{(i)}, 
suppose that $r:=s_0-s>0$ satisfies 
 one of the following conditions\textup{:}
\begin{equation}\label{nsmoothintro}
\begin{split}
&\textup{(ii.a)}\ \ r\le s+\rho(1-\nu_c-\eps)-1\quad \text{and}\quad   s\le \rho(1+\nu_c+\eps),\\
&\textup{(ii.b)}\ \ r\le 2\rho-1\quad \text{and}\quad s> \rho(1+\nu_c+\eps).
\end{split}
\end{equation}

\noi
Given $(u_0,v_0) \in H^s(\T)\times H^s(\T)$, 
let $(u,v)
\in C([0, \tau]; H^s(\T)\times H^s(\T))$ be the  solution 
to the coupled modulated KdV system~\eqref{cmkdv}
on $\T$
with $(u,v)|_{t = 0} = (u_0,v_0)$, 
where 
$\tau \in (0,  T]$ denotes the local existence time.
Then, we have 
\begin{align*}
 (u - e^{-w(t) \dx^3}u_0,v- e^{-\alpha w(t) \dx^3}v_0) \in C([0, \tau]; H^{s_0}(\T)\times H^{s_0}(\T)).
\end{align*}

\noi
\textup{(iii) (global well-posedness).}
In addition to the hypotheses of Part \textup{(i)}, suppose that 
  $w$ is $(\rho,\g)$-irregular on $\R_+$
and  that either $s\ge 0$ or $\rho, s<0$
satisfy one of the following conditions\textup{:}
\begin{equation}
\label{condG}
\begin{split}
\textup{(iii.a)} &\ \ \tfrac {1}{2} \le \rho < \tfrac{5-2\g}{2(1-\g)(1-\nu_{c} -\eps)} \quad \text{and} \quad s > \g \left( 1-\rho(1-\nu_{c}-\eps)\right), \\ 
\textup{(iii.b)} &\ \  \rho \geq \tfrac{5-2\g}{2(1-\g)(1-\nu_{c} -\eps)} \quad \text{and} \quad s >  \frac{
2(2\gamma-1)\bigl(1 - \rho(1-\nu_c-\varepsilon)\bigr)
- 3(1-\gamma)^2
}{
2(-\gamma^2+3\gamma-1)
}.
\end{split}
\end{equation}
Then, the modulated KdV system \eqref{cmkdv}
is globally well-posed in $H^s(\T)\times H^s(\T)$.
\end{theorem}

\remark \rm Note that, in view of Remark \ref{v_zero}, Theorem \eqref{THM:11} above (combined with \cite[Theorem 1.5]{GLO}) guaranties that for {\it almost every $\alpha\in (0,4)$}, the Majda-Biello system \eqref{cmkdv} is globally well-posed in $H^s(\T)\times H^s(\T)$, provided $\rho=\rho(s)>0$ is sufficiently large. 

\medskip

%Although it is not in the author's intention to compare the conditions of Theorem \ref{THM:11} with the ones of Theorem \ref{THM:1} on each case, it is interesting to perform an ``asymptotic analysis". 

%\noi We define the following drivers associated to these equations.
%\begin{align}
%\begin{aligned}
%X^{1}_{t,r}(f_1,f_2)
%=\frac 12\displaystyle\int_r ^t \uw(t')^{-1}(\partial_x((\uwa(t') f_1(t'))(\partial_x((\uwa(t') f_2(t')))dt' \\
%X^{2}_{t,r}(f_1,f_2)
%=\displaystyle\int_r ^t \uwa(t')^{-1}(\partial_x((\uw(t') f_1(t'))(\partial_x((\uwa(t') f_2(t')))dt'
%\end{aligned}
%\label{K1}
%\end{align}

%\noi We now recall some useful definitions.

%\begin{definition}
% Given $s \in \R$, $0 < \g < 1$,  $k \in \N$, and $T > 0$, 
%we define the space 
%$\cX^{s, \g}_k([0, T]\times\M)$ of drivers 
%as follows; for $k \ge 2$, we set
%\begin{align}
%\begin{split}
%& \cX^{s, \g}_k ([0, T]\times \M) \\
%& \quad = 
%%X(0) = DX[0]= 0 \text{ and }\updl X = 0\big\}, 
%\end{split}
%\label{X1}
%\end{align}

%\noi
%endowed with the norm
%\begin{align}
%\|X\|_{\cX^{s, \g}_k([0, T]\times \M)} 
%=  \|X\|_{C^\g_{2, T}\Lip_k(H^s(\M))} + \|DX\|_{C^\g_{2, T}\Lip_{k-1}(H^s(\M);\cL_1(H^s(\M)))}.
%\label{X2}
%\end{align}

%\end{definition}

In addition to the well-posedness theory, we also obtain
convergence of the Galerkin approximation. That is, let us fix $N \in \N$ and consider following modulated KdV system  on the circle:
\begin{equation}
\label{kdv1x}
\begin{cases}
\dt u^{N}+  \dx^3 u^{N}
 \cdot \dt w (t)=\frac{1}{2}\dx \P_N\big((\P_Nv^N)^2\big),\\
\dt v^N+  \al\dx^3 v^N \cdot \dt w(t)=\dx \P_N\big(\P_Nu^N\cdot  \P_Nv^N \big),\\
(u^N,v^N)_{|t=0}=(u_0,v_0),
\end{cases}
\end{equation}

\noi
where   $\P_N$ denotes  the usual frequency projector onto the (spatial) frequencies 
$\{|n| \leq N\}$ defined in \eqref{Diri1}. 
%Then, a slight modification of the proof of Theorem \ref{THM:1}\,(i)
%yields  the following results
%for the modulated KdV \eqref{kdv1} on the circle.
Then, the following result holds true:
\begin{theorem}\label{THM:G1}
Given $\alpha\in (0,4)\setminus\{1\}$,  $\rho \ge\frac 12$,  $\frac12< \g < 1$, $c$ as in \eqref{roots1}, $\nu_c$ as in Definition \ref{inf_index}, $0<\eps<1-\nu_{c}$, and $T> 0$, let
$w$ be $(\rho,\g)$-irregular on $[0, T]$ in the sense of Definition~\ref{DEF:ir}, and $u_0$ be with mean zero. 
Suppose that $s\in \R$ satisfies:
\begin{align}\label{condsintrostrict}
s> 1-\rho(1-\nu_{c}-\eps).
\end{align} 

\noi
Then, given $(u_0,v_0) \in H^s(\T)\times H^s(\T)$, 
the solution $(u^N,v^N)$ to the truncated modulated KdV equation~\eqref{kdv1x}
on $\T$
converges to the solution $(u,v)$
to the modulated KdV system \eqref{cmkdv}
on $\T$
in $C([0, \tau]; H^s(\T)\times H^s(\T))$, 
where $\tau \in (0, T]$ denotes the local existence time of the modulated KdV system
\eqref{cmkdv}.
Moreover, 
given any $r > 0$, 
the rate of convergence of $(u^N,v^N)$ to $(u,v)$ is
uniform in $(u_0,v_0) \in B_r$ 
where
$B_r$ denotes the ball 
in $H^s(\T)\times H^s(\T)$
of radius $r > 0$ centered at the origin.
\end{theorem}

\remark \rm We recall that, for $\alpha\in (0,4)\setminus\{1\}$, $c$ as in \eqref{roots1},  $\nu_c$ as in Definition \ref{inf_index}, $0<\eps<1-\nu_c$, and $\frac{1}{2}<s^*_1\le s^*_2$ defined by
 \begin{align}\label{ss7}
s^*_1:= \frac{1+\nu_c+\eps}{2}, \quad s^*_2:=\text{max}\Big(\frac{6s^{*}_1-2(s^{*}_1)^2}{5-s^{*}_1}, \frac{2 s^{*}_1+9}{14} \Big),
\end{align} 
in \cite{LWP_diophantine,GWP_dioph} the author proved that the unmodulated Majda-Biello system
\begin{equation}
\label{ckdv}
\begin{cases}
\dt u+  \dx^3 u =\frac{1}{2}\dx v^2,\\
\dt v+  \al\dx^3 v =\dx (uv),\\
(u,v)_{|t=0}=(u_0,v_0),
\end{cases}
\end{equation}
%\noi For $c$ as in \eqref{roots1},  $\nu_c$ as in Definition \ref{inf_index} and $0<\eps<1-\nu_c$,  we define
%\begin{align}\label{ss7}
%s^*_1:= \frac{1+\nu_c+\eps}{2}, \quad s^*_2:=\text{max}\Big(\frac{6s^{*}_1-2(s^{*}_1)^2}{5-s^{*}_1}, \frac{2 s^{*}_1+9}{14} \Big).
%\end{align}
%Note that $s^*_2\ge s^*_1$.
is (for $u_0$ with mean zero) semilinearly locally well-posed in $H^{s}(\T)\times H^s(\T)$ provided $s\ge s^*_1$, and globally well-posed if $s\ge s^*_2$.
%\begin{theorem}\label{THM:H}{\textnormal{(\cite[Theorem 1]{LWP_diophantine}, \cite[Theorem 1.8]{GWP_dioph})}}
%Given $\alpha\in (0,4]\setminus\{1\}$, $s^*_1, s^*_2$ as in \eqref{ss7} and $u_0$ with mean zero. If $s\ge s^{*}_1$ then the system \eqref{ckdv} is semilinearly locally well-posed in $H^{s}(\T)\times H^s(\T)$. In addition, if $s\ge s^{*}_2$, the system is globally well-posed.
%\end{theorem}
To be precise, in the original statements, see  \cite[Theorem 1.8]{GWP_dioph},  the author takes into account two additional indexes $\nu_{d}, \nu_{d_*}$ associated with $\Xi^{(2)}$ in \eqref{res2intro}. Namely,
%\begin{align*}
%\Xi^{(2)}(n,n_1,n_2)=(1-\al)n_1 (n_1 -d n)(n_1 -d_*n), 
%\end{align*}
%with 
\begin{align}\label{dd}
d:=\frac{-3\alpha+\sqrt{3\alpha(4-\alpha)}}{2(1-\alpha)},\quad d_*:=\frac{-3\alpha-\sqrt{3\alpha(4-\alpha)}}{2(1-\alpha)}.
\end{align}
On the other hand, in view of \eqref{relphase}, \eqref{roots1} and \eqref{dd}, it is easy to verify that 
$$c=\frac{1}{d},\quad c_*=\frac{1}{d_*},$$ which, combined with Definition \ref{inf_index} and $c+c_*=1$, implies
\begin{align}\label{eq_index}
\nu_{d}=\nu_{c}=\nu_{c_*}=\nu_{d_*}.
\end{align}
Thus, we can state the results of \cite{LWP_diophantine, GWP_dioph} by replacing $\text{max}(\nu_{c},\nu_{d}, \nu_{d_*})$ with $\nu_c$.  We also remark that, in \cite{GWP_dioph}, the main challenges for obtaining  global well-posedness precisely rely on the fact that $\alpha\neq 1$. In fact, the lack of the relation \eqref{K3} prevents suitable pointwise cancellations in the multilinear estimates used in the $I$-method energy increments. To overcome such an issue, the author initially considers a first modified energy $E^{(1)}$ (see \cite[page 24]{GWP_dioph}) and then carefully chooses a second modified energy $E^{(2)}$ (see \cite[eq. (5.6)]{GWP_dioph}) designed to cancel the worst resonant terms. See \cite[Section 5]{GWP_dioph} for further details. In particular, since $\nu_c=0$ for almost every $\alpha\in (0,4)\setminus\{1\}$, \cite[Theorem 1.8]{GWP_dioph} essentially provides global well-posedness for \eqref{ckdv} in $H^s(\T)\times H^s(\T)$ for $s>\frac{5}{7}$, exhibiting a big discrepancy  with the case $\alpha=1$.
%On the other hand, our main result shows that such a discrepancy essentially disappears in the modulated framework. As a matter of fac
\remark \rm \label{large_a_intro}Note that, if $\alpha\in (-\infty,0)\cup(4,\infty)$, then $c,c_*$ are not real numbers, and $\Xi^{(1)}$ satisfies the inequality
$$|\Xi^{(1)}(\bar{n})| \ges |n| \big(\max (|n|, |n_1|)\big)^2,$$ 
resulting in a sharper estimate compared to \eqref{eq_res_1}, \eqref{eq_res_2}. Hence, a gain of smoothness is expected. However, since the focus of this paper is the interplay between Diophantine conditions on $c$ and the $(\rho,\gamma)$-irregularity of the modulation $w$, we do not consider this case.

\section{Preliminaries}\label{prel}
In this section, we recall some notations and preliminary results. 
%Finally, we recall a useful lemma on continuous and discrete convolutions.Then, we introduce the general theory of the nonlinear Young integrals and the Young differential equation, followed by giving a specific driver associated with the modulated KdV \eqref{maineq}.
Let $A\les B$ denote an estimate of the form $A\leq CB$ for some constant $C>0$. We write $A\sim B$ if $A\les B$ and $B\les A$, while $A\ll B$ denotes $A\leq c B$ for some small constant $c> 0$. 
We may write  $\les_{\al}$ and $\sim_{\al}$ to 
emphasize the dependence on an external parameter $\al$.
We use $C>0$ to denote various constants, which may vary line by line.

In expressing the dependence of a function $u$
on the time variable, we often use the short-hand notation
$u_t = u(t)$,  which is standard in probability theory and stochastic analysis.

%where $dx_\T$ denotes
% the normalized Lebesgue measure $ dx_\T =  (2\pi)^{-1}dx$.

\subsection{Function spaces} Given $\ld\ge 1$,  we set
$\Z_{\ld} = \Z/\ld$
,
$\Z^*_{\ld} = \Z_\ld\setminus\{0\}$ 
and 
$\T_\ld = \R / (2\pi \ld \Z)$.
We define 
the Fourier transform of a function $f$ on $\T_\ld$
by 
$$
\mathcal{F}_{\T_\ld}(f)(n)=\ft  f(n)=\int_{\T_\ld}f(x)e^{-i n x}\frac{dx}{2\pi}
$$

\noi
for $n\in\Z_{\ld}$.
Then, we have 
\begin{align}
\begin{split}
	f(x)
		& =
		\frac{1}{\ld} \sum_{ n \in\Z_{\ld}}\ft  f(n)e^{i  nx}, \\
		\frac{1}{2\pi}\int_{\T_\ld}|f(x)|^2 dx
		& =\frac{1}{\ld}\sum_{n\in\Z_{\ld}}|\ft  f(n)|^2,\\
		\ft{fg}(n)
		& =\frac{1}{\ld}
		\sum_{ \substack{n_1, n_2 \in \Z_\ld\\n = n_1+n_2}}
		\ft  f(n_1)\ft  g(n_2).
	\end{split}
	\label{FT3}
\end{align}
 
 Given $N \in \N$, 
we also denote by  $\P_N$
 the Dirichlet projector onto the frequencies 
$\{|n| \leq N\}$
defined by 
\begin{align}
 \ft{\P_N f}(n)= \ind_{\{|n|\leq N\}} \ft f(n).
 \label{Diri1}
\end{align}
%In particular, 
%$\S'(\R)$ denotes the space of tempered distributions on $\R$, 
%while 
%$\S'(\T) = \mathcal D'(\T)$
%denotes the space of distributions on $\T$.
Then, given $s \in \R$, we define the non-homogeneous and homogeneous Sobolev spaces $H^{s}(\T_\ld)$, $\dot H^{s}(\T_\ld)$  via
\begin{align*}
\|  f  \|^2_{H^{s}(\T_\ld)} :=
 \sum_{n\in \Z_\ld} \jb{n}^{2s} |\ft f(n)|^2,\\
 \|  f  \|^2_{\dot H^{s}(\T_\ld)} :=
 \sum_{n\in \Z^*_\ld} |n|^{2s} |\ft f(n)|^2,
\end{align*}
where $\jb{\cdot}=(1+|\cdot|^2)^{\frac{1}{2}}$. To simplify the notation, we define $\H^s(\T_\ld):=~H^{s}(\T_\ld)\times H^s(\T_\ld)$
endowed with the norm
\begin{align*}
\|(f_1,f_2)\|^2_{\H^s(\T_\ld)}:=\|f_1\|^2_{\H^s(\T_\ld)}+\|f_2\|^2_{\H^s(\T_\ld)},
\end{align*}
and, when there is no confusion on the spatial domain, we may simply write $\H^s_x$, or $\H^s$ etc.
 \medskip
 
Given $k \in \N$, 
let $V, V_1, \dots, V_k$  be separable Hilbert spaces.
We use 
\begin{align}\label{tensor}
\cL_k\Big(\bigotimes_{j = 1}^k V_j; V\Big)
\end{align}

\noi
 to denote 
the Banach space of bounded $k$-linear operators 
on $\bigotimes_{j = 1}^k V_j$ 
(equipped with the Hilbert tensor norm) %product) 
with values in $V$.
When $V_j =  V$ for $j =1, \dots,k$,
we simply set  
 $\cL_k(V)=\cL_k(V^{\otimes k}; V)$. 

Let $V$ be a Banach space and $T>0$.
For $n\in\N$, we denote 
\begin{align*}
\Delta_{n, T} = 
\big\{ (t_1, \ldots, t_n) \in [0,T]^n: \ t_i > t_j 
\text{ for } i < j\big\}.
\end{align*}
We denote by $C_{n,T}V$ 
the space of continuous functions 
from $\Delta_{n,T}$ to $V$. When $n=1$,
we may write $C_T V$ for simplicity, 
and equip this space with the supremum norm: 
\begin{align*}
\|f\|_{C_T V} = \|f\|_{L^\infty_T V} = \sup_{0\leq t \leq T} \|f(t)\|_V.
\end{align*}

We define the coboundary operator 
$\updl: C_{n,T} V  \to C_{n+1,T} V$ 
as follows; 
given  $f\in C_{n,T} V$ 
and $(t_1,\ldots, t_{n+1}) \in \Dl_{{n+1},T}$, 
we set
\begin{align*}
(\updl f)_{t_1,\ldots , t_{n+1}} = 
\sum_{k=1}^{n+1} (-1)^{n-k} f_{t_1, \ldots, t_{k-1}, t_{k+1}, \ldots, t_{n+1}}.
\end{align*}

\noi
For example, for $f\in C_T V$
and 
$g\in C_{2,T} V$, 
we have 
\begin{align}
\begin{split}
(\updl f )_{t,r} &= f_t - f_r, \\
(\updl g)_{t_1,t_2,t_3} &= 
g_{t_1,t_3} - g_{t_1,t_2} - g_{t_2,t_3}
\end{split}
\label{dl1}
\end{align}

\noi
for $(t,r)\in\Dl_{2,T}$
and $(t_1,t_2,t_3) \in \Dl_{3,T}$.
As noted in \cite{GT10}, 
the sequence 
\begin{align*}
0 \too \R \too C_{1, T}V \stackrel{\updl}{\too} C_{2, T}V
\stackrel{\updl}{\too} C_{3, T}V
\stackrel{\updl}{\too} \cdots
\end{align*}

\noi
is exact. 
In particular, we have 
 $\updl\circ\updl =0$ and 
if $f \in C_{n,T} V$ with $\updl f =0$, 
then there exists a $g\in C_{n-1,T}V$ 
such that $f= \updl g$; 
see, for example,  \cite[Lemma 2.1]{GT10}.

Given $0 < \g < 1$, we denote by $C^\g_T V = C^\g([0, T]; V)$ the space of $\g$-H\"older continuous functions taking values in $V$, endowed  with the seminorm:
\begin{align*}
\|f\|_{C^\g_T V} = \sup_{(t,r)\in \Dl_{2,T}} 
\frac{\|(\updl f)_{t,r}\|_V}{|t-r|^\g}.
\end{align*}

\noi
We also define 
 $\CC^\g_T V = \CC^\g([0, T]; V)$ via the norm:
\begin{align}
\| f \|_{\CC^\g_TV} = \| f \|_{L^\infty_TV} + \|u\|_{C^\g_T V}.
\label{Ho2a}
\end{align}

\noi 
We also introduce the spaces 
$C^\g_{n,T}V$, $n=2,3$, 
equipped with the following H\"older-type norms;
 for $g\in C_{2,T}V$ and $h\in C_{3,T}V$, we set
\begin{align}
\begin{split}
\| g\|_{C^\g_{2,T} V} & 
= \sup_{(t,r) \in \Dl_{2,T}} \frac{\|g_{t,r} \|_{V} }{|t-r|^{\g}}, \\
\| h\|_{C^\g_{3,T} V} & 
= \inf_{0<\al<\g} 
\sup_{(t_1,t_2,t_3) \in \Dl_{3,T}} 
\frac{ \|h_{t_1,t_2,t_3}\|_V}
{|t_1-t_2|^{\al} |t_2-t_3|^{\g-\al}}.
\end{split}
\label{Ho2}
\end{align}

\noi For $n=2,3$, we also set 
$C^{1+}_{n,T}V = \bigcup_{\g>1}C^\g_{n,T}V$.

\medskip

We essentially follow the notation of \cite[Section 2]{CGLLO, GLO}, suitably adapted to the present Sobolev space framework. 

Let us fix $s,s_0\in \R$, with  $s_0\ge s$. 
We  denote by $\Lip_2(\H^s; H^{s_0})$ 
the Banach space of 
locally Lipschitz maps $f:\H^s\to H^{s_0}$
with polynomial growth of order $2$
such that 
\begin{align}
\|f\|_{\Lip_2(\H^s;\, H^{s_0})} := 
\sup_{x,y\in \H^s} 
\frac{\|f(x)-f(y)\|_{H^{s_0}}}{\|x-y\|_{\H^s} \big(1+\|x\|_{\H^s}+\|y\|_{\H^s}\big)}
<\infty .
\label{Lip1}
\end{align}
%When $s = s_0$, 
%we simply set $\Lip_2(H^s)= \Lip_2(\H^s;H^{s})$. 
%When $V = W$, 
%we simply set $\Lip_2(V)= \Lip_2(V;V)$.
Also, we say that $f\in~\Lip^2_2(\H^s; H^{s_0})$ 
if 
\smallskip

\begin{itemize}
\item[(i)] $f\in\Lip_2(\H^s; H^{s_0})$,

\smallskip
\item[(ii)]
 $f$ is Fr\'echet differentiable 
with 
$Df \in \Lip(\H^s;\cL_1(\H^s; H^{s_0}))$,

\end{itemize}

where, for Banach spaces $V,W$,
$\Lip(V;W)$ denotes the usual space of Lipschitz continuous functions from $V$ to $W$.

\smallskip

\noi
From \eqref{Ho2} and \eqref{Lip1}, we have 
\begin{align}
\|f \|_{C^\g_{2, T}\Lip_2(\H^s;H^{s_0})}
= \sup_{(t,r) \in \Dl_{2,T}}
\frac 1{|t-r|^{\g}}
\sup_{x,y\in \H^s} 
\frac{\| f_{t, r}(x)- f_{t, r}(y)\|_{H^{s_0}}}{\|x-y\|_{\H^s}\big(1+\|x\|_{\H^s}+\|y\|_{\H^s}\big)}.
\label{Ho3}
\end{align}

\smallskip

\noi
In addition, 
we  use $\Lip_2(\H^s, \H^{s_0}; H^{s_0})$
%$\subset  \Lip_2(\H^s;H^s)$ 
to denote the Banach space of 
locally Lipschitz maps $f:\H^{s}\to H^{s_0}$
such that 
\begin{align}
\|f\|_{\Lip_2(\H^s, \H^{s_0}; H^{s_0})}
= 
\sup_{x,y\in \H^{s_0}} 
\frac{\|f(x)-f(y)\|_{H^{s_0}}}
{G_{s,{s_0}}(x, y)}
<\infty , 
\label{Lip2}
\end{align}
\noi
where $G_{s, s_0}(x, y)$ is given by 
\begin{align*}
G_{s, s_0}(x, y)
& = \|x-y\|_{\H^{s_0}} \big(1+\|x\|_{\H^{s}}+\|y\|_{\H^{s}}\big)\\
&\quad+ \|x-y\|_{\H^{s}}
 \big(1+\|x\|_{\H^{s_0}}+\|y\|_{\H^{s_0}}\big).
\end{align*}

\noi
Then, 
from \eqref{Ho2} and \eqref{Lip2}, we have 
\begin{align}
\|f \|_{C^\g_{2, T}\Lip_2(\H^s, \H^{s_0}; H^{s_0})}
= \sup_{(t,r) \in \Dl_{2,T}}
\frac 1{|t-r|^{\g}}
\sup_{x,y\in \H^{s_0}} 
\frac{\| f_{t, r}(x)- f_{t, r}(y)\|_{H^{s_0}}}{
G_{s, s_0}(x, y)}.
\label{Ho4}
\end{align}

\noi
%In addition, suppose that $V_0 \hookrightarrow V$
%is a Banach subspace of $V$.
%Given an integer $k \ge 2$, %$k \in \N$, 
%we  use $\Lip_k(V, V_0; V_0)\subset  \Lip_k(V)$ 
%to denote the Banach space of 
%locally Lipschitz maps $f:V_0\to V_0$
%such that 
%\begin{align}
%\|f\|_{\Lip_k(V, V_0; V_0)}
%= \sup_{x,y\in V} \frac{\|f(x)-f(y)\|_{V_0}}{G_{V, V_0}(x, y)}<\infty , 
%\label{Lip2}
%\end{align}

%\noi
%where $G_{V, V_0}(x, y)$ is given by 
%\begin{align*}G_{V, V_0}(x, y)& = \|x-y\|_{V_0} \big(1+\|x\|_{V}+\|y\|_{V}\big)^{k-1}\\& \quad + \|x-y\|_{V} \big(1+\|x\|_{V}+\|y\|_{V}\big)^{k-2}\big(1+\|x\|_{V_0}+\|y\|_{V_0}\big)\end{align*}\noi
%Then, from \eqref{Ho2} and \eqref{Lip2}, we have \begin{align}\|f \|_{C^\g_{2, T}\Lip_k(V, V_0; V_0)}
%= \sup_{(t,r) \in \Dl_{2,T}}
%\frac 1{|t-r|^{\g}}\sup_{x,y\in V} \frac{\| f_{t, r}(x)- f_{t, r}(y)\|_{V_0}}{G_{V, V_0}(x, y)}.\label{Ho4}\end{align}

\medskip

Given $s \in \R$, $0 < \g < 1$, $T > 0$, and $\ld\ge 1$, 
we define the space 
$\cX^{s, \g}_2([0, T]\times\T_\ld)$ of drivers 
as follows;
\begin{align}
\begin{split}
& \cX^{s, \g}_2 ([0, T]\times \T_\ld):=
\big\{X \in C^\g_{2, T} \Lip_2^2(\H^s;H^s):
X(0) = DX[0]= 0 \text{ and }\updl X = 0\big\}, 
\end{split}
\label{X1}
\end{align}

\noi
endowed with the norm
\begin{align}
\|X\|_{\cX^{s, \g}_2([0, T]\times \T_\ld)} 
:=  \|X\|_{C^\g_{2, T}\Lip_2(\H^s;H^s)} + \|DX\|_{C^\g_{2, T}\Lip(\H^s;\,\cL_1(\H^s;H^s))}.
\label{X2}
\end{align}

\noi
%When $k = 1$, we simply set
%\begin{align}
% \cX^{s, \g}_1 ([0, T]\times \T) 
%=  C^\g_{2, T} \Lip_1(H^s(\T))
%\text{ ($=  C^\g_{2, T} \L_1(H^s(\T))$)}.
%\label{X1z}
%\end{align}

\noi
In \eqref{X1}, $X(0) = 0$ (and $DX[0]= 0$) means $X_{t, r}(0) = 0$ 
(and $DX_{t, r}[0]= 0$, respectively)
for any $(t, r) \in \Dl_{2, T}$, 
where $DX_{t, r}[0]$
denotes the Fr\'echet derivative of $X_{t, r}$ at $u = 0 \in \H^s$.
%Similarly, we define 
%$\dot \cX^{s, \g}_k([0, T]\times\T)$ by setting\begin{align}\begin{split}& \dot \cX^{s, \g}_k ([0, T]\times \T) \\& \quad = \big\{X \in C^\g_{2, T} \Lip_k^2(\dot H^s(\T)):X(0) = DX[0]= 0 \text{ and }\updl X = 0\big\}, \end{split}\label{X1a}\end{align}

%\noiendowed with the norm\begin{align}\|X\|_{\dot \cX^{s, \g}_k([0, T]\times \T)} 
%=  \|X\|_{C^\g_{2, T}\Lip_k(\dot H^s(\T))} + \|DX\|_{C^\g_{2, T}
%\Lip_{k-1}(\dot H^s(\T);\cL_1(\dot H^s(\T)))}
%\label{X2a}
%\end{align}\noifor  $k \ge 2$.
%For  $k = 1$, 
% we set 
%$ \dot \cX^{s, \g}_1 ([0, T]\times \T) =  C^\g_{2, T} \Lip_1(\dot H^s(\T))$.
We may simply write $\cX^{s, \g}_2 ([0, T])$ or $\cX^{s, \g}_2(T)$, 
where there is no confusion about 
an underlying spatial domain.
We write 
$X \in \cX^{s, \g}_2 (\R_+)
= \cX^{s, \g}_2 (\R_+\times \T)$ 
if  $X \in \cX^{s, \g}_2(T)$ for any $T > 0$
(but $\sup_{T > 0 } \| X\|_{\cX^{s, \g}_2(T)}$ may be infinite).

In the sequel,  we also set the space
$\cX^{s, s_0, \g}_{2}([0, T]\times\T_\ld)$ 
of drivers (for establishing nonlinear smoothing) by 
setting 
\begin{align}
& \cX^{s, s_0, \g}_2 ([0, T]\times \T_\ld) 
:= C^\g_{2, T} \Lip_2(\H^s; H^{s_0}).
%& \quad = 
%\big\{X \in C^\g_{2, T} \Lip_k(H^s(\M); H^{s_0}(\M)):
%X(0) = DX[0]= 0 \text{ and }\updl X = 0\big\}, 
%& \quad = 
%\big\{X \in C^\g_{2, T} \Lip_k^2(H^s(\M); H^{s_0}(\M)):
%X(0) = DX[0]= 0 \text{ and }\updl X = 0\big\}, 
\label{X1x}
\end{align}

%\noi
%Similarly, 
%we set 
%\begin{align*}
%& \dot \cX^{s, s_0, \g}_2 ([0, T]\times \T) 
%= C^\g_{2, T} \Lip_2(\dot H^s(\T); \dot H^{s_0}(\T)).
%\end{align*}
 while 
$\cY^{s, s_0, \g}_{2}([0, T]\times\T_\ld)$ (for establishing persistency of regularity) by 
\begin{align}
\begin{split}
& \cY^{s, s_0, \g}_2 ([0, T]\times \T_\ld) := 
  C^\g_{2, T} \Lip_2(\H^s, \H^{s_0};H^{s_0}).
\end{split}
\label{X2c}
\end{align}

For $0<\s<1,$ we often use short-hand notations such as
$\CC^\s_T H^s_x$ to denote  $\CC^\s\big([0, T]; H^s(\T_\ld))$,  etc, 
 when there is no ambiguity. Moreover, when considering  $\ld=1$, in all the spaces we have defined we simply drop the dependence of $\ld$. 

%\noi
%Similarly, we set 
%\begin{align}
%\begin{split}
%& \dot \cY^{s, s_0, \g}_2 ([0, T]\times \T) = 
%  C^\g_{2, T} \Lip_2(\dot H^s(\T)\times \dot H^s(\T), \dot H^{s_0}(\T);\dot H^{s_0}(\T)).
%\end{split}
%\label{X2d}
%\end{align}

%\medskip

%We often use short-hand notations such as
%$C^\g_T H^s_x  = C^\g\big([0, T]; H^s(\T))$, 
% when there is no ambiguity.

%Let $f_1, f_2  \in H^{s}(\T)$. 
\subsection{Sewing lemma and Young integrals}
\label{SUBSEC:Y2}
\noi The main goal of this subsection is to provide criteria for defining the integrals in the right hand side of \eqref{mild2bis}, following the approach of \cite{CG1, CGLLO, GLO}. 
To this aim, for $f_1,f_2$ functions on $\T$, we define the drivers 
\begin{equation}
\label{driver1}
\begin{split}
X^1_{t,r}(f_1,f_2)
&=\frac 12\displaystyle\int_r ^t \uw(t')^{-1}\mathcal{N}(\uwa(t') f_1, \uwa(t') f_2)dt',\\
X^2_{t,r}(f_1,f_2)
&=\displaystyle\int_r ^t \uwa(t')^{-1}\mathcal{N}(U^{w}(t') f_1, \uwa(t') f_2)dt',
\end{split}
\end{equation}
%where 
%\begin{align}
%\label{f_moda}
 %\F\left( \uwa (t) f \right)(n)  = e^{i\al n^3 w(t)} \ft f(n) \\
%\F \left( \uw (t) f \right)(n)  = e^{i n^3 w(t)} \ft f(n)
%\label{fmod}
%\end{align}
where  $\mathcal{N}(u,v) = \dx (uv)$ and $U^{w_\alpha}$ is defined as in \eqref{semigroup_a}.
By taking the Fourier transform in \eqref{driver1} we have
\begin{align}
\label{F_X}
2\F ( X^{1}_{t,r}(f_1,f_2)) (n)= in\sum_{ \substack{n_1, n_2 \in \Z\\n = n_1+n_2}}\Phi^{w}_{t,r}  (\Xi^{(1)}(\bar n))  \ft f_1(n_1)  \ft f_2 (n_2),
\end{align}
\begin{align}
\label{F_X2}
\F ( X^{2}_{t,r}(f_1,f_2)) (n)= in\sum_{ \substack{n_1, n_2 \in \Z\\n = n_1+n_2}}\Phi^{w}_{t,r}  (\Xi^{(2)}(\bar n))  \ft f_1(n_1)  \ft f_2 (n_2),
\end{align}
where $\Phi^{w}_{t,r}$ and $\Xi^{(j)}$ are defined as in \eqref{rho2}, \eqref{res2intro}.
%while 
%\begin{align}
%\label{mod1}
%\Xi^{(1)}_{\al - \KDV}(\bar n)=\Xi^{(1)}_{\al - \KDV}(n,n_1,n_2)= \al n_1^3 + \al n_2^3 - n^3
%\end{align}
%\begin{align}
%\label{mod2}
%\Xi^{(2)}_{\al - \KDV}(\bar n)=\Xi^{(2)}_{\al - \KDV}(n,n_1,n_2)= \al n^3 + \al n_2^3 - n_1^3
%\end{align}
\smallskip

In view of Lemma \ref{LEM:sew} (and in particular  Proposition \ref{PROP:young1}), we first recall  the following result (see e.g.,  \cite[Lemma 3.1]{CGLLO}) which gives sufficient criteria to check the regularity of the drivers $X^j$ defined in \eqref{driver1}. 
\begin{lemma}\label{LEM:OBS1}

Let $X^{j}$, $j=1,2,$ be a driver as in \eqref{driver1}.
Then, given  $s, s_0 \in \R$, $\frac 12 < \g < 1$,  and $T \ge 1$
such that $s_0 > s$, 
the following statements hold\textup{:}

\smallskip
\begin{itemize}
\item[(i)] 
If we have 
\begin{align*}
\|X^{j}_{t,r}\|_{\cL_2(H^s(\T))}
\les 
\|\Phi^w\|_{\W^{\rho,\g}_T}
|t-r|^{\g}
\end{align*}

\noi
for any $0 \le r < t \le T$, 
then
 $X^j$ belongs to $\cX^{s, \g}_2([0, T]\times \T)$
 defined in 
 \eqref{X1}
with the bound 
\begin{align*}
\|X^j\|_{\cX^{s, \g}_2([0, T]\times \T)} \les \|\Phi^w\|_{\W^{\rho,\g}_T}.
\end{align*}

\smallskip
\item[(ii)] \textup{(persistence of regularity).}
Suppose that $X$ can be written as  $X^{j} = X^{j,1}+X^{j,2}$
such that  for $j,k= 1,2$, we have 
\begin{align}
\begin{split}
X^{j,k} &\in \cL_{2, k}^{s, s_0}\\ :\!&=
\cL_2\Big(\big(\bigotimes_{i  = 1}^{k-1}H^{s}(\T)\big) \otimes
H^{s_0}(\T) 
\otimes \big(\bigotimes_{i  = k+1}^{2}H^{s}(\T)\big); H^{s_0}(\T)\Big)
\end{split}
\label{X5}
\end{align}

\noi
with the bound
\begin{align*}
\|X^{j,k}_{t,r}\|_{\cL_{2, j}^{s, s_0}}
\les 
\|\Phi^w\|_{\W^{\rho,\g}_T}
|t-r|^{\g}.
\end{align*}

\noi
Then, 
 $X^j$ belongs to $ \cY^{s, s_0, \g}_2([0, T]\times \T)$
 defined  in \eqref{X2c}
 with the bound 
\begin{align*}
\|X^j\|_{\cY^{s, s_0, \g}_2([0, T]\times \T)} \les \|\Phi^w\|_{\W^{\rho,\g}_T}.
\end{align*}

\smallskip
\item[(iii)] \textup{(nonlinear smoothing).}
If we have 
\begin{align*}
\|X^j_{t,r}\|_{\cL_2(\H^{s}(\T); H^{s_0}(\T))}
\les 
\|\Phi^w\|_{\W^{\rho,\g}_T}
|t-r|^{\g}
\end{align*}

\noi
for any $0 \le r < t \le T$, 
then
 $X^j$ belongs to $\cX^{s, s_0, \g}_2([0, T]\times \T)$
 defined in 
 \eqref{X1x}
with the bound 
\begin{align*}
\|X^j\|_{\cX^{s, s_0, \g}_2([0, T]\times \T)} \les \|\Phi^w\|_{\W^{\rho,\g}_T}.
\end{align*}

\smallskip
\item[(iv)] \textup{(convergence).}
Given $N \in \N$ and $j=1,2$, 
define
 the truncated drivers $X^{j,N}$ by 
\begin{equation}\label{trunc}
\begin{split}
X^{1,N}_{t,r}(f_1,f_2): = \frac{1}{2}\int_r^t \uw(t')^{-1} \P_N \NN(\P_N \uwa(t') f_1, \P_N \uwa(t') f_2) dt',\\
X^{2,N}_{t,r}(f_1,f_2) := \int_r^t \uwa(t')^{-1} \P_N \NN(\P_N \uw(t') f_1, \P_N \uwa(t') f_2) dt'.
\end{split}
\end{equation}
\noi
If we have
\begin{align*}
\|X^{j,N}_{t, r} - X^{j}_{t,r}\|_{\cL_2(H^s(\T))}
\les 
o(1)|t-r|^{\g},
\end{align*}

\noi
as $N\to \infty$, 
uniformly in  $0 \le r < t \le T$, 
then, 
 $X^{j,N}$
 converges to $X^{j}$ in $\cX^{s, \g}_2([0, T]\times \T)$.

\end{itemize}

\end{lemma}

\smallskip

\remark \label{samedriver} \rm Although Lemma \ref{LEM:OBS1} is stated for the drivers $X^{j}$ in \eqref{driver1}, an analogous result holds for the drivers $X^{j,\ld}$ and $Y^{j,\ld}$ defined respectively in \eqref{Xld1} and \eqref{X1d2}. Furthermore, we refer to \cite[Subsections 2.2, 3.1]{CGLLO} for more general criteria addressing the regularity of $k$-linear drivers.
\smallskip

Next, we carry on by  recalling the following version of the sewing lemma (cf.  
\cite{GT10} and \cite[Lemma 3.3]{CGLLO}).  
See also 
\cite[Proposition 2.3, Corollary 2.4, Corollary 2.5]{GT10} and \cite[Lemma 4.2]{FH20}.

\begin{lemma}[sewing lemma]
\label{LEM:sew}

Let $V$ be a Banach space and 
fix $T>0$. 
Then,  there exists a unique linear map \textup{(}called the sewing map\textup{)}
$\Lambda:C^{1+}_{3,T} V \cap 
\Ker \updl|_{C_{3,T}V}
\to C^{1+}_{2,T}V$ such that 

\smallskip
\begin{enumerate}
\item[(i)] 
We have 
$\updl \Lambda h = h$
for each  $h\in C_{3,T}V\cap \Ker  \updl|_{C_{3,T} V}$.

\smallskip

\item[(ii)]
 For each $\z >1$,
the sewing map $\Lambda$ is continuous
from $C^\z _{3,T}V\cap 
\Ker  \updl|_{C_{3,T} V}$ to 
$C^\z _{2,T}V$ such that 
\begin{align}
\|\Lambda h \|_{C^\z _{2,T}V} 
\le \frac{1}{2^\z - 2}  \| h \|_{C^\z _{3,T}V} 
\label{sew1}
\end{align}

\noi
for any $h\in C^\z _{3,T}V$.

\smallskip
\item[(iii)] 
Given any  $g\in C_{2,T}V$ 
with 
$\updl g\in C^\z _{3,T}V$, 
 there exists  unique
$f\in C([0, T];V)$  \textup{(}modulo an additive  constant\textup{)} 
such that 
$\updl f = (\Id - \Lambda \updl)g$. 
In addition, 
we have 
\begin{align}
(\updl f)_{t,r} = \lim_{|\Pi([r,t])|\to 0} 
\sum_{j=0}^n g_{t_j,t_{j+1}}
\label{sew2}
\end{align}

\noi
 for any $(t,r)\in \Dl_{2,T}$,
 where 
 the limit is over any partition
 $\Pi([r,t])$  
 of  $[r,t]$\textup{:}
\[\Pi ([r,t]) = \{r = t_n < \dots < t_1 <  t_0 = t\}\]
whose mesh size 
$|\Pi([r,t])| = \sup_{j} |t_j-t_{j+1}|$ 
tends to $0$.
\end{enumerate}

\end{lemma}

Next, let us describe briefly the construction of the nonlinear Young integral associated to a driver  $X$. See also \cite[Section 2]{GLO}, and \cite{G1,Hu} for a more general setting. Given $T>0$ and $X\in \cX^{s, \g}_2([0, T]\times \T_\ld)$, where $\cX^{s, \g}_2([0, T]\times \T_\ld)$ is defined in \eqref{X1}.
%Assume in addition that the driver $X_{t,r}$ is  a (bilinear) integral operator over the integral $[r,t]$, as \eqref{driver1} or \eqref{Xld1}.  
Given $\vec{\bf z}:=(\uu,\vv)\in \CC^{\s}([0,T];\H^s(\T_\ld))$, for some $0<\s<1$,  our task is to define the nonlinear Young integral $\I^{X}(\vec{\bf z})$ as the unique function on $[0,T]$ whose increment satisfies
\begin{align}\label{YI}
(\updl  \I^{X}(\vec{\bf z}))_{t,r}=X_{t,r}(\vec{\bf z}(r))+R_{t,r},\quad 0\le r< t\le T,
\end{align}
for some sufficiently regular two-parameter remainder $R=R^{X,\vec{\bf z}}$. To be more precise, we look for $R$ sufficiently regular allowing us to define $\I^{X}(\vec{\bf z})$ as the unique limit of Riemann-Stieltjes type sums.  If we set $\Theta$ on $\Delta_{2,T}$ as 
\begin{align}\label{13T}
    \Theta_{t,r}:=X_{t,r}(\vec{\bf z}(r)),
\end{align}
by applying the co-boundary operator $\updl$ to \eqref{YI} with \eqref{13T}, we infer that any error $R$ (if it exists) must satisfy
\begin{align}\label{R13}
(\updl R)_{t_1,t_2,t_3}=-(\updl \Theta)_{t_1,t_2,t_3}=X_{t_1,t_2}(\vec{\bf z}(t_2))-X_{t_1,t_2}(\vec{\bf z}(t_3)).
\end{align}
Moreover, by combining \eqref{R13} with the regularity assumptions on $X$ and $\vec{\bf z}$, we obtain that $\updl R\in C^{\s+\g}_{3,T}H^s(\T_\ld)$; see e.g., \cite[Lemma 3.4 (i)]{CGLLO}. As a result, if $\s+\g>1$, the sewing lemma (Lemma \ref{LEM:sew}) allows us to define $R$ by the relation
\begin{align}\label{Rr}
R:=-\Lambda \updl \Theta\in C^{\s+\g}_{2,T}H^s(\T_\ld)
\end{align}
and, the nonlinear Young integral $ \I^{X}(\vec{\bf z})$ of $\vec{\bf z}$ as the unique function in $C^\g([0,T];H^s(\T_\ld))$ with $\I^{X}(\vec{\bf z})(0)=0$ whose increment is given by 
$$(\updl  \I^{X}(\vec{\bf z}))=(\text{Id}-\Lambda\updl)\Theta. $$
In particular, in view of \eqref{Rr}, \eqref{YI}, \eqref{13T}, \eqref{Ho2} and $\s+\g>1$,
\begin{align*}
\I^{X}(\vec{\bf z})(t)&=\lim_{|\Pi([0,t])|\to 0} \sum_{j=0}^n \Theta_{t_j,t_{j+1}}\\
&=\lim_{|\Pi([0,t])|\to 0} \sum_{j=0}^n X_{t_j,t_{j+1}}(\vec{\bf z}(t_{j+1}))+\lim_{|\Pi([0,t])|\to 0} \sum_{j=0}^n R_{t_j,t_{j+1}}\\
&= \lim_{|\Pi([0,t])|\to 0} \sum_{j=0}^n X_{t_j,t_{j+1}}(\vec{\bf z}(t_{j+1})),\quad t\in [0,T],
\end{align*}
where the above limits are takes in the sense of Lemma \ref{LEM:sew}-(iii).
See Proposition \ref{PROP:young1} below for the precise statement.
We also refer to \cite[Subsection 1.6]{GLLO} for a deeper introduction to the topic.
% and $\vec{X}=(X^1,X^2)$, we define 
%\begin{align}\label{vec_I}
%\I^{\vec X}(\vec{\bf w})(t)=(\I^{X^{1}}(\vec{\bf w})(t),\,\I^{X^{2}}(\vec{\bf w})(t)).
%\end{align}
%We also refer the reader to  \cite[Lemma 3.4, Proposition 3.7]{CGLLO} for the proof of the following result.
\begin{proposition}\label{PROP:young1}\textnormal{(\cite[Lemma 3.4, Proposition 3.7]{CGLLO})}
Given   $s \in \R$, $0 < \g < 1$, $T >0$, and $\ld\ge 1$. 
Let   $X \in \cX^{s, \g}_2([0, T]\times \T_\ld)$, 
where $ \cX^{s, \g}_2([0, T]\times \T_\ld)$ is  defined  in \eqref{X1}, and 
$0 < \s < 1$
such that 
 $\z  = \s + \g > 1$.

\smallskip

\noi
\textup{(i)}
Let  
  $\lvec{\bf z} \in \CC^{\s}([0,T];\H^s(\T_\ld))$.
 Then, the nonlinear Young integral $\I^{X}(\vec{\bf z})$
 with the driver $X$ exists
 as the unique function
$\I^{X}(\vec{\bf z})\in C^\g([0,T]; H^s(\T_\ld))$
 with $\I^{X}(\vec{\bf z})(0) = 0$
 such that 
 \begin{align}
\updl \I^{X}(\vec{\bf z}) = (\Id - \Lambda \updl)\Theta, 
\label{Y1}
\end{align}

\noi
where $\Ta$ is given by   
\begin{align*}
\Theta_{t,r} = X_{t,r}(\vec{\bf z}(r)), \quad (t,r)\in \Dl_{2,T}.
\end{align*}

\noi
Moreover, we have
\begin{align}
\label{Y2}
\|\updl \I^{X}(\vec{\bf z}) - \Theta\|_{C^\z_{2,T} H^s_x}
\le
C_1(\vec{\bf z})\|X\|_{\cX^{s, \g}_2(T)},\\
\label{Y3}
\|\I^{X}(\vec{\bf z})\|_{C^\g_T  H^s_x} 
\le
C_2(\vec{\bf z})\|X\|_{\cX^{s, \g}_2(T)}, 
\end{align}

\noi
where $C_1(\uu)$ and $C_2(\uu)$ are given by
\begin{align}
\begin{split}
C_1(\vec{\bf z})
%& = 
%C_1(\|f(0)\|_{ V}, \|f\|_{C^\al_T  V}
%, T , k, \al, \z) \\
&= 
\frac{1}{2^\z-2}
\big(1+ 2\|\vec{\bf z}\|_{ L^\infty_T \H^s_x}\big) \|\vec{\bf z}\|_{C^\s_T  \H^s_x}, \\
C_2 (\vec{\bf z})
%& = 
%C_2(\|f(0)\|_{ V}, \|f\|_{C^\al_T  V}
%, T , k, \al, \z) \\
&= (1 \vee T)^\s C_1(\vec{\bf z}) + 
\big(1 + \|\vec{\bf z}\|_{ L^\infty_T \H^s_x} \big)
\|\vec{\bf z}\|_{L^\infty_T \H^s_x}.
\end{split}
\label{Ja3x}
\end{align}

Here, $a \vee  b \coloneqq \max(a,b)$. Finally, for $(t,r) \in \Dl_{2,T} $, we have,
\begin{align*}
(\updl \I^{X}(\vec{\bf z}))_{t,r} = \lim_{|\Pi([r,t])|\to 0} 
\sum^n_{j=0} \Theta_{t_j,t_{j+1}}.
\end{align*}

\smallskip

\smallskip

\item[(ii)]
\textup{(persistence of regularity).}
Suppose that $s_0>s$ and
 $X\in \cY^{s, s_0, \g}_2 ([0, T]\times \T_\ld)$, 
 where $\cY^{s, s_0, \g}_2 ([0, T]\times \T_\ld)$ is as in \eqref{X2c}
 and that 
   $\vec{\bf z}\in C^{\s}([0,T];\H^{s_0}(\T_\ld))$.
Then, we have 
\begin{align}
\Theta\in C^\g_{2,T}H^{s_0}_x
\qquad \text{and} 
\qquad 
\updl \Theta \in C^\z_{3,T}H^{s_0}_x.
\label{Jaa1}
\end{align}

\noi
Furthermore, 
we have 
$\I^{X}(\vec{\bf z})\in C^\g([0,T];H^{s_0}(\T))$, 
satisfying 
\begin{align}
\|\updl \I^{X}(\vec{\bf z}) - \Theta\|_{C^\z_{2,T}  H^{s_0}_x}
& \leq 
C_3(\vec{\bf z})
\|X\|_{\cY^{s, s_0, \g}_2 (T)},
\label{Jaa2}
\\
\|\I\|_{C^\g_T  H^{s_0}_x} 
& \leq
C_4(\vec{\bf z})
\|X\|_{\cY^{s, s_0, \g}_2 (T)}, 
\label{Jaa3}
\end{align}

\noi
where $C_3(\vec{\bf z})$ and $C_4(\vec{\bf z})$ are given by 
\begin{align}
\begin{split}
C_3(\vec{\bf z})
&= 
\frac{1}{2^\z-2}\Big(
\big(1+ 2\|\vec{\bf z}\|_{L^\infty_T \H^s_x} \big) \|\vec{\bf z}\|_{C^\s_T  \H^{s_0}_x} +  \|\vec{\bf z}\|_{C^\s_T  \H^s_x}
\big(1+ 2\|\vec{\bf z}\|_{L^\infty_T \H^{s_0}_x} \big)\Big),\\
C_4 (\vec{\bf z})
&= (1 \vee T)^\s C_3 + 
2\big(1 + \|\vec{\bf z}\|_{ L^\infty_T \H^s_x}\big) 
\big(1 + \|\vec{\bf z}\|_{ L^\infty_T \H^{s_0}_x} \big).
\end{split}
\label{Jaa4}
\end{align}

\noi
\textup{(iii)}
Suppose that a sequence $\{X^N\}_{N \in \N}\subset 
 \cX^{s, \g}_2([0, T]\times \T_\ld)$
 converges to $X$ in 
 $ \cX^{s, \g}_2([0, T]\times \T_\ld)$ as $N \to \infty$.
Then, given
any  $\vec{\bf z} \in \CC^{\s}([0,T];\H^s(\T_\ld))$, 
 the nonlinear Young integral $\I^{X^N}(\vec{\bf z} )$
 of~$\vec{\bf z} $ 
 with the driver $X^N$ 
 converges to 
 the nonlinear Young integral $\I^X(\vec{\bf z} )$
 of~$\vec{\bf z} $ 
 with the driver $X$, 
 in $C^{\g}([0,T]; H^{s}(\T_\ld))$
as $N \to \infty$.
Moreover, 
given any $r > 0$, 
the rate of convergence is
uniform in $\vec{\bf z} \in B_r$ 
where
$B_r$ denotes the ball 
in $\CC^{\s}([0,T];\H^s(\T_\ld))$
of radius $r > 0$ centered at the origin.
\end{proposition}

%%%%%%%%%%%%%%%%%%%%%%%%%%%%%%%%%%%%%%%%%%%%%%%%%%%%%%%%%%%%%%%%%%%%%%%%%%%%%%%%%%%%%%%%%%
\subsection{Auxiliary estimates}
%We begin the appendix by recalling the following technical result whose proof can be found in \cite[Lemma 4.2]{GTV} and \cite[Lemma 4.1]{MWX}. 
 %\begin{lemma}\label{LEM:SUM}
%\textup{(i)}
%Let  $\al, \be \in \R$ satisfy
%\begin{align}
 %\al \ge \be \ge 0 \qquad \text{and}\qquad  \quad \al+ \be > 1.
 %\label{SUM1}
%\end{align}

%\noi
%\Then, we have
%\begin{align*}
%\int_{\xi = \xi_1 + \xi_2}\frac{d\xi_1}{\jb{\xi_1}^\al \jb{\xi_2}^\be}
%& \les \frac 1{\jb{\xi}^{ \be - \ld}}, \\
 %\sum_{n = n_1 + n_2} \frac{1}{\jb{n_1}^\al \jb{n_2}^\be}
%& \les \frac 1{\jb{n}^{ \be - \ld}}
%\end{align*}

%\noi
%for any $\xi \in \R$ and $n \in \Z$, 
%where $\ld = 
%\max( 1- \al, 0)$ when $\al\ne 1$ and $\ld = \eps$ when $\al = 1$ for any $\eps > 0$.
%In particular, if \eqref{SUM1} holds, then we have 
%\[
%\sup_{\xi  \in \R} \int_{\xi = \xi_1 + \xi_2}\frac{d\xi_1}{\jb{\xi_1}^\al \jb{\xi_2}^\be}
%+ 
%\sup_{n \in \Z} \sum_{n = n_1 + n_2} \frac{1}{\jb{n_1}^\al \jb{n_2}^\be}
%< \infty. \]
%\end{lemma}

In this subsection, we state some technical results that are widely used in the proof of the main bilinear estimates of Section \ref{sec4}, and of the commutator estimates of Section \ref{sec6}. 

 We recall that for $c$ as in \eqref{roots1}, we denote by $\nu_c$ its corresponding minimal type index; see Definition \ref{inf_index}. In the next lemma we provide lower bounds for  $|\Xi^{(1)}(\bar{n})|$ (respectively, $|\Xi^{(2)}(\bar{n})|$) assuming $n\neq 0$ (respectively, $n,n_1\neq 0$) and  $\nu_c\neq \infty$. 

\begin{lemma}\label{lem_res}
Let $\nu_c\neq \infty$,  $\lambda \ge 1$, $j=1,2$, and $n\in \Z^{*}_\ld$.
Assume that $(n_1,n_2)\in \Z_\lambda^2$ are such that  $n=n_1+n_2$. The following possibilities hold:
\begin{itemize}
\item[(i)]
If $|n| \sim |n_1| \sim |n_2|$ and $\min (|n_1-cn|,|n_1-c_*n|) \ges \lambda^{-1}$, then
\begin{equation}\label{eq_res_1}
|\Xi^{(j)}(\bar n)|  \ges \lambda^{-1}|n|^2.
\end{equation}
%If $|n| \sim |n_1| \sim |n_2|$ and $\min (|n-cn_1|,|n-c_*n_1|) \ges \lambda^{-1}$, then
%\begin{equation*}
%|\Xi^{(2)}(\bar n)|  \ges \lambda^{-1}|n|^2.
%\end{equation*}
\item[(ii)]
If $|n| \sim |n_1| \sim |n_2|$ % and \textcolor{red}{$\min (|n_1-cn|,|n_1-c_*n|) \ll \lambda^{-1}$} ,
 then
\begin{equation}\label{eq_res_2}
|\Xi^{(j)}(\bar n)| \ges_\varepsilon  \lambda^{-2-\nu_c-\varepsilon}|n|^{1-\nu_c-\varepsilon}.
\end{equation}

\item[(iii)]
If $\max(|n|,|n_1|,|n_2|) \gg \min(|n|,|n_1|,|n_2|)$, then
\begin{equation}\label{eq_res3}
|\Xi^{(1)}(\bar n)|  \ges |n| (\max(|n|,|n_1|,|n_2|))^2, \ \ \ \ |\Xi^{(2)}(\bar n)|  \ges |n_1| (\max(|n|,|n_1|,|n_2|))^2.
\end{equation}
\end{itemize}
\end{lemma}
\begin{proof}
From \eqref{relphase}, we obtain the estimates for $|\Xi^{(2)}(\bar n)|$ from those for $|\Xi^{(1)}(\bar n)|$ by interchanging the roles of $n_1$ and $n$.
Thus, we focus only on those for $|\Xi^{(1)}(\bar n)|$.

By the triangle inequality, we have
\begin{equation}\label{t16}
\max( |n_1 -cn|, |n_1 - c_*n| ) \ge \frac{1}{2}(|n_1 -cn|+|n_1 - c_*n|) \ge \frac{|c-c_*|}{2}|n|.
\end{equation}
Then, \eqref{eq_factori1} and \eqref{t16} imply
\begin{align}
  |\Xi^{(1)}(\bar n)|& \sim |n|\max(|n_1 -c  n|, |n_1 - c_* n| )\min(|n_1 -c  n|, |n_1 - c_* n| )\notag\\
& \ges |n|^2\min(|n_1 -c  n|, |n_1 - c_* n| ).\label{eq_res_4}
\end{align}
Therefore, we obtain \eqref{eq_res_1} by the assumption  (i).

Next, by Definition \eqref{index} and \eqref{eq_index},  it follows that for any ${m},{m_1} \in \N$
 \begin{align*}
 |{m}|^2|{m_1}-c{m}| \ges_\varepsilon |{m}|^{1-\nu_c-\varepsilon}\quad \text{and}\quad  |{m}|^2|{m_1} - c_*{m}| \ges_\varepsilon |{m}|^{1-\nu_c-\varepsilon}, 
 \end{align*} 
which implies that for any $n,n_1\in \Z_\ld$
\begin{align*}
|n|^2|n_1-cn| \ges_\varepsilon \lambda^{-2-\nu_c-\varepsilon}|n|^{1-\nu_c-\varepsilon}\quad \text{and}\quad |n|^2|n_1 - c_*n| \ges_\varepsilon \lambda^{-2-\nu_c-\varepsilon}|n|^{1-\nu_c-\varepsilon}.
\end{align*} 
Therefore, we obtain \eqref{eq_res_2} from \eqref{eq_res_4}.
Finally, \eqref{eq_res3} easily follows from \eqref{eq_factori1}.
\end{proof}

%\begin{remark}\rm Note that, if $\nu_c\ge 1$, \eqref{eq_res_2} does not provide any useful lower bound on $\Xi^{(j)}$. In particular, 
%\end{remark}

A corollary of the above result is of a particular interest for its application to the proof of the commutator estimates in Proposition \ref{PROP:com1}. 

\begin{corollary}\label{cor_res}
Let $\nu_c\neq \infty$ and  $\lambda \ge 1$. Let $n\in \Z^{*}_\ld$,
 $(n_1,n_2)\in \Z_\lambda^2$ such that  $n=n_1+n_2$, and $\Xi^{(j)}$ as in \eqref{res2intro}. The following hold:
\begin{itemize}
\item[(i)]
If $|n| \ges \max (|n_1|, |n_2|)$, $n_1\neq 0$ and $\min (|n_1-cn|,|n_1-c_*n|) \ges \lambda^{-1}$, then
\begin{equation}\label{cor_res_1}
|\Xi^{(j)}(\bar n)|  \ges \lambda^{-1}|n|^2.
\end{equation}
%If $|n| \ges \max (|n_1|, |n_2|)$, $\min (|n-cn_1|,|n-c_*n_1|) \ges \lambda^{-1}$ and $n_1 \neq 0$, then
%\begin{equation}\label{cor_res_2}
%|\Xi^{(2)}(\bar n)|  \ges \lambda^{-1}|n|^2.
%\end{equation}
\item[(ii)]
If $|n| \ges \max (|n_1|, |n_2|)$ and $n_1\neq 0$ % and \textcolor{red}{$\min (|n_1-cn|,|n_1-c_*n|) \ll \lambda^{-1}$}, 
then
\begin{equation}\label{cor_res_3}
|\Xi^{(j)}(\bar n)| \ges_\varepsilon  \lambda^{-2-\nu_c-\varepsilon}|n|^{1-\nu_c-\varepsilon}.
\end{equation}
%If $|n| \ges \max (|n_1|, |n_2|)$,  $n_1 \neq 0$ %and \textcolor{red}{$\min (|n_1-cn|,|n_1-c_*n|) \ll \lambda^{-1}$}, 
%then
%\begin{equation}\label{cor_res_4}
%|\Xi^{(2)}(\bar n)| \ges_\varepsilon  \lambda^{-2-\nu_c-\varepsilon}|n|^{1-\nu_c-\varepsilon}.
%\end{equation}
\item[(iii)]
If $|n| \ll \max (|n_1|, |n_2|)$, then
\begin{equation}\label{cor_res_5}
|\Xi^{(j)}(\bar n)|  \ges |n||n_1|^2 \sim |n||n_2|^2.
\end{equation}
\end{itemize}
Moreover, when $j=1$ the assertions \textup{(i)-(ii)} remain valid also for $n_1=0$. 
\end{corollary}
\begin{proof}
\textbf{Cases (i)-(ii)}. If $|n| \ges \max (|n_1|, |n_2|)$, then either $|n|\sim\min(|n_1|, |n_2|)$ or $|n|\gg\min (|n_1|, |n_2|)$.
When $|n| \sim \min (|n_1|, |n_2|)$ then $|n| \sim |n_1| \sim |n_2|$. Thus, \eqref{cor_res_1}-\eqref{cor_res_3} follow from (i) and (ii) of Lemma \ref{lem_res}.
On the other hand, when $|n| \gg \min (|n_1|, |n_2|)$, it follows that $|n| \sim \max (|n_1|, |n_2|) \gg \min (|n_1|, |n_2|)$.
Thus, in this case, \eqref{cor_res_1}-\eqref{cor_res_3} follow from Lemma \ref{lem_res}-(iii) and the bounds
%\[
%|n| (\max(|n|,|n_1|,|n_2|))^2 \ge \lambda^{-1}|n|^2 \ge \lambda^{-2-\nu_c-\varepsilon}|n|^{1-\nu_c-\varepsilon}.
%\]
%Similarly, by further assuming $n_1 \neq 0$,   \eqref{cor_res_2} and \eqref{cor_res_4} follow from (iii) of Lemma \ref{lem_res} and
\[
\min(|n|,|n_1|) ( \max(|n|,|n_1|,|n_2|))^2 \ge \lambda^{-1}|n|^2 \ge \lambda^{-2-\nu_c-\varepsilon}|n|^{1-\nu_c-\varepsilon}.
\]

\noi\textbf{Case (iii)}. If $|n| \ll \max (|n_1|, |n_2|)$ then $|n_1| \sim |n_2|$. Thus, \eqref{cor_res_5} follows from (iii) of Lemma \ref{lem_res}.
\end{proof}

{{Next, for the convenience of the reader, we explicitly state the following result aiming to control the $(\rho,\g)$-irregularity of a suitably scaled modulation. Such a scaling will play a key role in Section \ref{sec6} to prove global well-posedness; see e.g., \eqref{scaleKdV}.}

Given $b\in \R$ and $\ld\ge 1$, we define $w^\ld$ as:
\begin{align}\label{scalw}
w^\ld(t)=\ld^3w(\ld^{-b}t).
\end{align}
Then, the following holds:
\begin{lemma}\label{lem_Phi}
Let $(a,b)\in \R^*\times \R, \lambda \ge 1$. Given $ \rho,\gamma > 0$, $T>0$, let  $w$ be $(\rho,\gamma)$-irregular on $[0,T]$, then
\[
|\Phi_{t,r}^{w^\lambda}(a)| \le \lambda^{b(1-\gamma)-3\rho}\|\Phi^{w}\|_{\W^{\rho,\g}_{T}}|t-r|^{\gamma}|a|^{-\rho},
\]
\end{lemma}
\noi for any $0\le r <t \le \lambda^bT$,  where  $w^\ld$ is as in \eqref{scalw}.
\begin{proof}
From \eqref{rho2},  it follows that
\begin{align*}
\lambda^{3\rho}|a|^{\rho}\frac{|\Phi_{t,r}^{w^\lambda}(a)|}{|t-r|^\gamma}
&\le \sup_{a\in \R}\sup_{0\le r<t \le \lambda^{b}T}\jb{\lambda^3 a}^\rho \frac{|\int_r^t e^{ia\lambda^3 w(\lambda^{-b}t')}\, dt'|}{|t-r|^\gamma}\\
&=\lambda^{b(1-\gamma)}\sup_{a \in \R}\sup_{0\le r<t \le \lambda^b T}\jb{\lambda^3 a}^\rho \frac{|\int_{\lambda^{-b}r}^{\lambda^{-b}t} e^{ia\lambda^3 w(t')}\, dt'|}{|\lambda^{-b}t-\lambda^{-b}r|^\gamma}\\
&=\lambda^{b(1-\gamma)}\|\Phi^{w}\|_{\W^{\rho,\g}_{T}},
\end{align*}
concluding the proof. 
\end{proof}

In order to conclude our auxiliary subsection, we formulate two results that will be employed in most of the bilinear estimates of this paper; see Lemma \ref{lem_nonlinsmooth}, Lemma \ref{LEM:tri2} and Proposition \ref{PROP:com1}.
\begin{lemma}\label{lem_bi}
Let $\lambda \ge 1$ and $b\in \R$, Given  $\rho,\gamma > 0$, let 
$w$ be  $(\rho,\gamma)$-irregular on $[0,T]$,
$a(\bar n)$ (never zero) and $F(\bar n)$ be fuctions on $\Z_\lambda^3$. 
{{Assume that there exists $K>0$ such that either}}
\begin{equation}\label{eq_condition_1}
\quad \sup_{n\in \Z_\lambda}
\sum_{\substack{n_1, n_2 \in \Z_\lambda\\n = n_1 + n_2}}
|F(\bar n)|^2|a(\bar n)|^{-2\rho}  \les {{K}}
\end{equation}
or
\begin{equation}\label{eq_condition_2}
\quad \sup_{n_1 \in \Z_\lambda}
\sum_{\substack{n, n_2 \in \Z_\lambda\\n = n_1 + n_2}}
|F(\bar n)|^2|a(\bar n)|^{-2\rho}  \les {{K}}.
\end{equation}
Then, for any  $0\le r <t \le \lambda^bT$ we have
\begin{equation}
\begin{split}
&\sum_{n \in \Z_\lambda}  \bigg (\sum_{ \substack{n_1,n_2 \in \Z_\lambda \\n = n_1+n_2}} |F(\bar n)| |\Phi^{w^{\lambda}}_{t,r}  (a(\bar n))|
|\ft{f}_1 (n_1)|  |\ft{ f }_2 (n_2)|    \bigg)^2\\
&\phantom{XXX}\les  {{K}}\lambda^{2b(1-\gamma)-6\rho+2} \|\Phi^w_{t,r}\|^2_{\W^{\rho,\g}_T} |t-r|^{2\g}\|f_1\|^2_{L^2(\T_\lambda)}\|f_2\|^2_{L^2(\T_\lambda)},
\end{split}\label{eq_lem_bi}
\end{equation}
where $w^\ld$ is as in \eqref{scalw}.
\end{lemma}
\begin{proof}
When \eqref{eq_condition_1} holds, by Lemma \ref{lem_Phi}, the H\"older inequality and \eqref{FT3}, 
\begin{equation*}
\begin{split}
 \text{L.H.S}\ \text{of}\ \eqref{eq_lem_bi}&\les \lambda^{{2b(1-\gamma)-6\rho}}\|\Phi^{w}\|^2_{\W^{\rho,\g}_{T}}|t-r|^{2\gamma}\bigg(\sup_{n\in \Z_\lambda}
\sum_{\substack{n_1, n_2 \in \Z_\lambda\\n = n_1 + n_2}}
|F(\bar n)|^2|a(\bar n)|^{-2\rho}\bigg)\\
&\phantom{XX}\times\sum_{n \in \Z_\lambda}
{
\sum_{ \substack{n_1, n_2 \in \Z_\lambda\\n = n_1+n_2}}| \ft f_1(n_1) |^2 |\ft f_2(n_2)|^2}\\
&\les {{K}}\lambda^{2b(1-\gamma)-6\rho+2}\|\Phi^{w}\|^2_{\W^{\rho,\g}_{T}}|t-r|^{2\gamma}
\|{{f_1}}\|_{{{L^2}}(\T_\lambda)}^2
\|{{f_2}}\|_{{{L^2}}(\T_\lambda)}^2.
\end{split}
\end{equation*}
{{Similarly, when \eqref{eq_condition_2} holds, by Lemma \ref{lem_Phi}, the H\"older inequality and \eqref{FT3},
\begin{equation*}
\begin{split}
\text{L.H.S}\ &\text{of}\ \eqref{eq_lem_bi}
\les \lambda^{{2b(1-\gamma)-6\rho}}\|\Phi^{w}\|^2_{\W^{\rho,\g}_{T}}|t-r|^{2\gamma}\\
&\phantom{XXXXXX}\times\sum_{n \in \Z_\lambda}
\Big(\sum_{ \substack{n_1, n_2 \in \Z_\lambda\\n = n_1+n_2}} |F(\bar n)|^2 |a(\bar n) |^{-2\rho} | \ft f_1(n_1) |^2 \Big) \| f_2 \|^2_{{{L^2}}(\T_\lambda)}\\
&\phantom{XXXX} \les \lambda^{{2b(1-\gamma)-6\rho}}\|\Phi^{w}\|^2_{\W^{\rho,\g}_{T}}|t-r|^{2\gamma}\\
&\phantom{XXXXXX}\times\sum_{n_1 \in \Z_\lambda} | \ft f_1(n_1) |^2
\Big(\sum_{ \substack{n, n_2 \in \Z_\lambda\\n = n_1+n_2}} |F(\bar n)|^2 |a(\bar n) |^{-2\rho}  \Big)\| f_2 \|^2_{{{L^2}}(\T_\lambda)}\\
&\phantom{XXXX}\les {{K}}\lambda^{2b(1-\gamma)-6\rho+2}\|\Phi^{w}\|^2_{\W^{\rho,\g}_{T}}|t-r|^{2\gamma}
\|{{f_1}}\|_{{{L^2}}(\T_\lambda)}^2
\|{{f_2}}\|_{{{L^2}}(\T_\lambda)}^2.
\end{split}
\end{equation*} }}
\end{proof}
%\begin{remark}\label{remark_ld}
%\rm The same claim holds if any of $n,n_1,n_2\in \Z_\lambda$ in \eqref{eq_condition_1}, \eqref{eq_condition_2}, and \eqref{eq_lem_bi} is replaced by $n,n_1,n_2\in \Z^*_\lambda$. 
%\end{remark}

%%%%%%%%%%%%%%%%%%%%%%%% Prop for summations %%%%%%%%%%%%%%%%%%%%%%%%

\begin{proposition}\label{sum28}
 Let $\beta\in \R$,  $\ld\ge 1$ and $N\in \N$.  Then, 
\begin{align*}
\sum_{\substack{n\in \Z^{*}_\ld\\ |n|\le N }} |n|^{\beta}\sim \begin{cases} \ld N^{\beta+1}  & \mbox{if }\beta>-1, \\ \ld(\log(N)+\log(\ld))  \ & \mbox{if }\beta=-1,
\end{cases}
\end{align*}
Similarly, 
\begin{align*}\sum_{n\in \Z^{*}_\ld} |n|^{\beta}\sim \ld^{-\beta}  \ \  \mbox{if }\beta < -1.
\end{align*}
\end{proposition}
\begin{proof}
The proof follows from \cite[Lemma 1]{ET1} and a direct computation.
%\begin{align*}
%&\sum_{\substack{n\in \Z^{*}_\ld\\ |n|\le N}} |n|^{\beta}=\ld^{-\beta}\sum_{\substack{m\in \Z \\ 1\le |m|\le \ld N}}|m|^{\beta}\sim  \begin{cases} \ld N^{\beta+1}, & \mbox{if }\beta>-1,\\
%\ld(\log(N)+\log(\ld))& \mbox{if  }\beta=-1,
%\end{cases}\\
%&\sum_{n\in \Z^{*}_\ld} |n|^{\beta}=\ld^{-\beta}\sum_{\substack{m\in \Z \\ 1\le |m| }}|m|^{\beta}\sim \ld^{-\beta} \ \ \mbox{if  }\beta <-1.
%\end{align*}
\end{proof}

%%%%%%%%%%%%%%%%%%%%%%%%%%%%%%%%%%%%%%%%%%%%%%%%%%%%%%%%%%%%%%%%%%%%%%%%%%%
%%%%%%%%%%%%%%% Section 3 Local well-posedness
%%%%%%%%%%%%%%%%%%%%%%%%%%%%%%%%%%%%%%%%%%%%%%%%%%%%%%%%%%%%%%%%%%%%%%%%%%%

\section{Local well-posedness}\label{sec4}
 The goal of this section is to prove local well-posedness for the modulated Majda-Biello system \eqref{cmkdv}. The first step is, in view of Lemma \ref{LEM:OBS1} and Proposition \ref{PROP:young1},  to control the  $\L_2(H^s(\T))$- type norms of $X^j$, where we recall that $X^j$ and $\L_2(H^s(\T))$ are defined respectively in \eqref{driver1} and \eqref{tensor}. This allows us to define the integrals in the right hand side of \eqref{mild2bis} as nonlinear Young integrals. After that, local well-posedness will follow from \cite[Proposition 3.8]{CGLLO}; see Proposition \ref{PROP:main} below for the details. 

\begin{remark}\label{rem_0}\rm 
Strictly speaking, for the driver $X^2$,  the bilinear estimates of  Lemmas \ref{lem_nonlinsmooth},\ref{pers},\ref{lem_galapprox} below,  are not on 
$H^{s}(\T) \times H^{s}(\T)$, but on $H^{s}_{0}(\T) \times H^{s}(\T)$, where
$H^{s}_{0}(\T)$ denotes the closed subspace of $H^{s}(\T)$ consisting of mean-zero
functions. Consequently, in the case $j=2$,  the spaces $\cX^{s,s_0}_2([0,T]\times\T), \cY^{s, s_0, \g}_2([0, T]\times \T)$ should be
understood as being defined with $H^s_0(\T)$ in place of $H^s(\T)$ in the first
factor. For simplicity,  since Lemma \ref{LEM:OBS1} and Proposition \ref{PROP:young1} are essentially insensitive to this modification, we keep the same notation for the entire paper. 
\end{remark}

\subsection{Bilinear estimates}

\begin{lemma}\label{lem_nonlinsmooth}
 Given $\rho\geq \frac{1}{2}$,  $\frac{1}{2}<\gamma<1$, $c$ as in \eqref{roots1}, $\nu_c$ as in Definition \ref{inf_index}, $0<\eps<1-\nu_c$, and $T>0$, let  $w$ be $(\rho,\g)$-irregular on $[0, T]$ in the sense of Definition~\ref{DEF:ir}.  Assume that $s, s_{0}\in \R$ satisfy
\begin{equation}\label{condabc}
\begin{split}
&\textup{(a)} \, s_0-2s-\rho(1-\nu_c-\varepsilon)+1 \le 0;\\
&\textup{(b)} \, s+\rho \ge 0\ or\  2s_0-2\rho+3\ge 0;  \\
&\textup{(c)} \, s-s_0+2\rho-1 \ge 0\ or\ -2s-2\rho+1 \ge 0.
\end{split}
\end{equation}
Then,
\begin{equation}\label{bi_nonlinsmooth}
\big\|X^{j}_{t,r}(f_1,f_2)\big\|_{H^{s_0}}\les \|f_1\|_{H^s}\|f_2\|_{H^s}\|\Phi^w\|_{\W^{\rho,\g}_T}|t-r|^{\g},
\end{equation}
%In particular, $X^j$ belongs to $\cX^{s,s_0 \g}_2([0, T]\times \T)$. 
for any $0\le r<t\le T$, where $X^{j}$ is defined in \eqref{driver1} and  we  further assume $\ft{f}_1(0)=0$ for $j=2$.

\end{lemma}

\begin{remark}\rm
\label{b11}
\textup{(i)}. If $s=s_0$, the statement of Lemma \ref{lem_nonlinsmooth} is equivalent to 
\begin{align}\label{condX1}
\rho\ge \frac{1}{2}\quad \text{and}\ s \ge 1-\rho(1-\nu_c-\varepsilon).
\end{align}
In particular, from Lemma \ref{LEM:OBS1}-\textup{(i)}, the condition \eqref{condX1} yields that, for $\g, T$ as in Lemma \ref{lem_nonlinsmooth},  $X^{j}\in \cX^{s, \g}_2([0, T]\times \T)$ with the bound 
\begin{align*}
\|X^{j}\|_{\cX^{s, \g}_2([0, T]\times \T)}\les \|\Phi^w\|_{\W^{\rho,\g}_T}.
\end{align*}
\smallskip

\noi\textup{(ii)}. If we set $r:=s_0-s>0$, the statement of Lemma \ref{lem_nonlinsmooth} is equivalent to 
\begin{align*}
r\le\textup{min}(2\rho-1, s+\rho(1-\nu_c-\eps)-1), 
\end{align*}
which is in turn equivalent to the following:
\begin{equation}\label{nsmooth}
\begin{split}
&\textup{(ii.a)}\ \ r\le s+\rho(1-\nu_c-\eps)-1\quad \text{and}\quad   s\le \rho(1+\nu_c+\eps),\\
&\textup{(ii.b)}\ \ r\le 2\rho-1\quad \text{and}\quad s> \rho(1+\nu_c+\eps).
\end{split}
\end{equation}
As a result, by further  assuming \eqref{nsmooth}, Lemma \ref{LEM:OBS1}-\textup{(iii)} yields that  $X^j$ belongs to $\cX^{s, s+r, \g}_2([0, T]\times \T)$ with the bound
\begin{align*}
\|X^j\|_{\cX^{s, s+r, \g}_2([0, T]\times \T)}\les \|\Phi^w\|_{\W^{\rho,\g}_T}.
\end{align*}
\end{remark}

\noi{\it Proof of Lemma \ref{lem_nonlinsmooth}.}
For $x\in \R$, and $\delta>0$ sufficiently small,  we define $[x]_{\delta}$ as follows:
\begin{equation}\label{posa}
[x]_{\delta}:=
\begin{cases}
0\quad  \text{if}\quad x<0,\\
\delta\quad  \text{if}\quad x=0,\\
x\quad \text{if}\quad x>0.
\end{cases}
\end{equation} 
From \eqref{posa}, note that if $x+y<0$ and, $x\le 0$ or $y\ge 0$, then $x+[y]_{\delta} \le 0$ provided $\delta$ is small enough. Moreover, \eqref{F_X} and \eqref{F_X2} imply 
\begin{align}
\hspace{-0.5cm} \|  X^{j}_{t,r} (f_1, f_2) \|^2_{H^{s_0}} \les
  \sum_{ n\in\Z^*}\bigg( \sum_{\substack{n_1, n_2 \in\Z\\n = n_1+n_2}}\frac{|n|^{s_0+1}|\Phi^{w}_{t,r}(\Xi^{(j)}(\bar n))|}{\jb{n_1}^s\jb{n_2}^s}|\jb{n_1}^s\ft{f}_1(n_1)| |\jb{n_2}^s\ft{f}_2 (n_2)|\bigg)^2.\label{eq_bi_smooth}
\end{align}
Assume first that $f_1,f_2$ are two functions on $\T$ such that $\ft f_1(0)={\ft{f}}_2(0)=0$.
Thus, we split $\sum:=\{(n, n_1, n_2) \in (\Z^{\ast})^3: n=n_1+n_2 \} = B_1^{(j)} \cup B_2^{(j)} \cup B_3 \cup B_4 \cup B_5$, where 
%{\Bl{might be helpful to see the explicit factorization of resonance function $\Xi^{(2)}(\bar{n})$, possibly in introduction.}}
\begin{equation}\label{BB15}
	\begin{split} 
		B^{(1)}_{1}  &:= \Big\{ (n, n_1, n_2) \in \Sigma : |n| \sim  |n_1| \sim |n_3|, \, \min(|n_1-cn|, |n_1-c_*n|) \ges 1)  \Big\}, \\ 
		B^{(2)}_{1}  &:= \Big\{ (n, n_1, n_2) \in \Sigma : |n| \sim  |n_1| \sim |n_3|, \, \min(|n-cn_1|, |n-c_*n_1|) \ges 1)  \Big\}, \\
		B^{(1)}_{2}  &:= \Big\{ (n, n_1, n_2) \in \Sigma : |n| \sim  |n_1| \sim |n_3|, \, \min(|n_1-cn|, |n_1-c_*n|) \ll 1)  \Big\},\\
		B^{(2)}_{2}  &:= \Big\{ (n, n_1, n_2) \in \Sigma : |n| \sim  |n_1| \sim |n_3|, \, \min(|n-cn_1|, |n-c_*n_1|) \ll 1)  \Big\},\\
		B_{3} &:= \Big\{ (n, n_1, n_2) \in \Sigma  : |n_1| \sim |n_2| \gg |n| \Big\}, \\
		B_{4} &:= \Big\{ (n, n_1, n_2) \in \Sigma  : |n| \sim |n_2| \gg |n_1| \Big\}, \\
		B_{5} &:= \Big\{ (n, n_1, n_2) \in \Sigma  : |n| \sim |n_1| \gg |n_2| \Big\}.
	\end{split}
\end{equation}
 For $j=1,2$ and $k=1,2,3$, by Lemma \ref{lem_bi} we obtain 
\begin{align}
\text{R.H.S. of \eqref{eq_bi_smooth}} \les  \Big(\sum_{k=1}^{5}A^{(j,k)}_{s,s_0}\Big)
\| \Phi^w \|^2_{\mathcal{W}^{\rho, \gamma}_{T}} |t - r|^{2\gamma} \| f_1 \|^2_{H^s (\T)} \| f_2 \|^2_{H^s (\T)},
\label{a10ineq}
\end{align}
where 
\begin{align}
\hspace{-1.5cm}A^{(j,k)}_{s,s_0} &:= \!\sup_{n \in \Z^{\ast}}\!\! \sum_{\substack{n_1, n_2 \in\Z^*\\n = n_1+n_2}}\!\!\ind_{B^{(j)}_k} \frac {|n|^{2s_0+2}|n_1|^{-2s}|n_2|^{-2s}}{|\Xi^{(j)} (\bar{n})|^{2\rho}}\quad\!\!\! \text{if}\ (j,k)\!\in\!\{(1,1),(2,1),(1,2),(2,2)\};
\label{A19}\\
\hspace{-0.3cm} A^{(j,k)}_{s,s_0}&:=\! \sup_{n_1 \in \Z^{\ast}}\!\! \sum_{\substack{n, n_2 \in\Z^*\\n = n_1+n_2}}\!\!\ind_{B_k} \frac {|n|^{2s_0+2}|n_1|^{-2s}|n_2|^{-2s}}{|\Xi^{(j)} (\bar{n})|^{2\rho}}\quad\! \!\text{if}\ (j,k)\in \{(1,3),(2,3)\};
\label{A219}\\
\hspace{-0.3cm} A^{(j,k)}_{s,s_0} &:=\! \sup_{n \in \Z^{\ast}}\!\!\sum_{\substack{n_1, n_2 \in\Z^*\\n = n_1+n_2}}\!\!\ind_{B_k} \frac {|n|^{2s_0+2}|n_1|^{-2s}|n_2|^{-2s}}{|\Xi^{(j)} (\bar{n})|^{2\rho}}\quad\!\! \text{if}\ (j,k)\!\in\!\{(1,4),(2,4),(1,5),(2,5)\},
\label{A319}
\end{align} 
provided
 \begin{align}
 \label{Asmall}
 A^{(j,k)}_{s,s_0} \les 1.
 \end{align}
 In what follows, we apply the bounds of Lemma \ref{lem_res} and Proposition \ref{sum28}  with $\ld=1$.
 
By (i) of Lemma \ref{lem_res} and \eqref{A19},  for $j=1,2$ we have
\begin{align*}
A^{(j,1)}_{s,s_0} \les \sup_{n \in \Z^{\ast}}  |n|^{2s_0-4s-4\rho+2} \sum_{|n_1| \sim |n|} 1 \les 1,
\end{align*}
 if $2s_0-4s-4\rho+3 \le 0$, which follows from (a), $\rho \ge 1/2$ and $\nu_c > 0$.

By (ii) of Lemma \ref{lem_res} and \eqref{A19}, for $j=1,2$ we have
\begin{align*}
A^{(j,2)}_{s,s_0} \les \sup_{n \in \Z^{\ast}}  |n|^{2s_0-4s-2\rho(1-\nu_c-\varepsilon)+2} \les 1
,
\end{align*}
if $2s_0-4s-2\rho(1-\nu_c-\varepsilon)+2 \le 0$, which is equivalent to (a).

By (iii) of Lemma \ref{lem_res}, and \eqref{A219}, 
\begin{align*}
A^{(1,3)}_{s,s_0} \les \sup_{n_1 \in \Z^{\ast}}  |n_1|^{-4s-4\rho} \sum_{|n| \ll |n_1|} |n|^{2s_0-2\rho+2} \les 1
,
\end{align*}
if $-4s-4\rho+[2s_0-2\rho+3]_{\delta} \le 0$, which follows from (a) and (b).

By (iii) of Lemma \ref{lem_res} and \eqref{A219},
\begin{align*}
A^{(2,3)}_{s,s_0} \les   \sup_{n_1 \in \Z^{\ast}}  |n_1|^{-4s-6\rho} \sum_{|n| \ll |n_1|} |n|^{2s_0+2} \les 1
,
\end{align*}
if $-4s-6\rho+[2s_0+3]_{\delta} \le 0$,
which follows from (a) and (b).

By (iii) of Lemma \ref{lem_res} and \eqref{A319}, 
\begin{align*}
A^{(1,4)}_{s,s_0} \les \sup_{n \in \Z^{\ast}}   |n|^{2s_0-2s-6\rho+2}\sum_{|n_1| \ll |n|} |n_1|^{-2s} \les 1
,
\end{align*}
if $2s_0-2s-6\rho+2+[-2s+1]_{\delta} \le 0$,
which follows from (a) and (c).

By (iii) of Lemma \ref{lem_res} and \eqref{A319},
\begin{align*}
A^{(2,4)} \les \sup_{n \in \Z^{\ast}}  |n|^{2s_0-2s-4\rho+2} \sum_{|n_1| \ll |n|}  |n_1|^{-2s-2\rho}  \les 1
,
\end{align*}
if $2s_0-2s-4\rho+2+[-2s-2\rho+1]_{\delta}\le 0$, which follows from (a) and (c).

By (iii) of Lemma \ref{lem_res} and \eqref{A319}, 
\begin{align*}
A^{(j,5)}_{s,s_0} \les  \sup_{n\in \Z^{\ast}}   |n|^{2s_0-2s-6\rho+2} \sum_{|n_2| \ll |n|} |n_2|^{-2s} \les 1
,
\end{align*}
 if $2s_0-2s-6\rho+2+[-2s+1]_{\delta} \le 0$,
which follows from (a) and (c).
We have therefore concluded the proof by assuming $\ft{f}_1(0)=\ft{f}_2(0)=0$. In the general case,  for $j=1,2$, in view of \eqref{eq_bi_smooth} we have 
\begin{equation}\label{SSum}
\begin{split}
\|X^{j}_{t,r} (&f_1, f_2) \|^2_{H^{s_0}}\\
& \les \ind_{\{j=1\}}\cdot \sum_{n\in\Z^*}\bigg( \sum_{\substack{n_1\in \Z,n_2 \in\Z^*\\n =n_1+n_2}}\ind_{\{n_1=0\}}\cdot \frac{|n|^{s_0+1}|\Phi^{w}_{t,r}(\Xi^{(j)}(\bar n))|}{\jb{n_1}^s\jb{n_2}^s}|\jb{n_1}^s\ft{f}_1(n_1)| |\jb{n_2}^s\ft{f}_2 (n_2)|\bigg)^2\\
&\phantom{X}+ \sum_{ n\in\Z^*}\bigg( \sum_{\substack{n_1 \in\Z^*, n_2\in \Z\\n = n_1+n_2}}\ind_{\{n_2=0\}}\cdot \frac{|n|^{s_0+1}|\Phi^{w}_{t,r}(\Xi^{(j)}(\bar n))|}{\jb{n_1}^s\jb{n_2}^s}|\jb{n_1}^s\ft{f}_1(n_1)| |\jb{n_2}^s\ft{f}_2 (n_2)|\bigg)^2\\
&\phantom{X}+ \sum_{ n\in\Z^*}\bigg( \sum_{\substack{n_1, n_2 \in\Z^*\\n = n_1+n_2}}\frac{|n|^{s_0+1}|\Phi^{w}_{t,r}(\Xi^{(j)}(\bar n))|}{\jb{n_1}^s\jb{n_2}^s}|\jb{n_1}^s\ft{f}_1(n_1)| |\jb{n_2}^s\ft{f}_2 (n_2)|\bigg)^2\\
&=:S^{(j)}_1+S^{(j)}_2+S^{(j)}_3. 
\end{split}
\end{equation}
Note that $S^{(j)}_3$ is controlled in \eqref{a10ineq}. 
For $S^{(j)}_2$, by \eqref{res2intro} and arguing similarly to \eqref{a10ineq},  we get the bound
\begin{equation}
\begin{split}
S^{(j)}_2& \les \sum_{ n\in\Z^*}\bigg( \sum_{\substack{n_1\in \Z^*, n_2 \in\Z\\n =n_1+n_2}}\ind_{\{n_2=0\}}\cdot\frac{|n|^{s_0+1}|\Phi^{w}_{t,r}(\Xi^{(j)}(\bar n))|}{\jb{n_1}^s\jb{n_2}^s}|\jb{n_1}^s\ft{f}_1(n_1)| |\jb{n_2}^s\ft{f}_2 (n_2)|\bigg)^2\\
& \les \| \Phi^w \|^2_{\mathcal{W}^{\rho, \gamma}_{T}} |t - r|^{2\gamma} \| f_1 \|^2_{H^s (\T)} \| f_2 \|^2_{H^s (\T)} \\
&\phantom{XXX} \times \bigg(\sup_{n \in \Z^{\ast}} \sum_{\substack{n_1\in \Z^*, n_2 \in\Z\\n = n_1+n_2}}\ind_{\{n_2=0\}}\cdot  \frac {|n|^{2s_0+2}|n_1|^{-2s}\jb{n_2}^{-2s}}{|\Xi^{(j)} (\bar{n})|^{2\rho}}\bigg)\\
& \sim  \|\Phi^w \|^2_{\mathcal{W}^{\rho, \gamma}_{T}} |t - r|^{2\gamma} \| f_1 \|^2_{H^s (\T)} \| f_2 \|^2_{H^s (\T)} \bigg(\sup_{n \in \Z^{\ast}} |n|^{2(s_0-s)+2-6\rho}\bigg)\\
&\les \| \Phi^w \|^2_{\mathcal{W}^{\rho, \gamma}_{T}} |t - r|^{2\gamma} \| f_1 \|^2_{H^s (\T)} \| f_2 \|^2_{H^s (\T)},
\end{split}
\label{S1}
\end{equation}
provided 
\begin{align}\label{cond_mean}
s_0-s+1-3\rho\le 0.
\end{align}
If $j=1$ the same estimate holds for $S^{(j)}_1$, while if  $j=2$ then $S^{(j)}_1=0$ by definition so we don not have additional conditions. As a result, we obtain \eqref{bi_nonlinsmooth} if \eqref{condabc} and \eqref{cond_mean} hold.

We now claim that the condition \eqref{cond_mean} is already included in \eqref{condabc}.  In order to see this, we consider two cases. If $s_0\le s$, then \eqref{cond_mean} holds since $\rho\ge \frac{1}{2}$. If $s_0>s$, then \eqref{condabc} is equivalent to \eqref{nsmooth}. Then, the conclusion follows by noting that 
\begin{align*}
s-1+3\rho\ge 
\begin{cases} 2s+\rho(1-\nu_c-\eps)-1\quad \text{if}\ s\le \rho(2+\nu_c+\eps),\\
s-1+2\rho.
\end{cases} 
\end{align*}
Hence, by combining \eqref{a10ineq}, \eqref{S1}, \eqref{SSum} we obtain \eqref{bi_nonlinsmooth} under the assumptions in \eqref{condabc}.
\qed

\begin{remark}\label{ind_large}\rm
 Note that, if $s=s_0$ and $1\le \nu_c<\infty$,  Lemma \ref{lem_res}-(ii) does not provide a uniform nontrivial lower bound for $\jb{\Xi^{(j)}(\bar{n})}$. Indeed, since there are regimes where $|\Xi^{(j)}(\bar{n})|$ can be arbitrarily close to zero, we are only  reduced to employ the trivial lower bound $\jb{\Xi^{(j)}(\bar{n})}\ge 1$.   In this case,  it is easy to verify that the  bilinear estimates \eqref{bi_nonlinsmooth} hold by assuming $\rho\ge \frac{1}{2}$ and $s\ge 1$. This condition shows that in this case the size $\rho$ of the irregularity does not provide any smoothing and the well-posedness result in Proposition \ref{PROP:main} alines  with its unmodulated counterpart, cf., \cite[Theorem 1]{LWP_diophantine}.
  %by arguing as in the proof of Lemma \ref{lem_nonlinsmooth}, we would need to control from above terms of the form 
% \begin{align*}
% A^{(j)} &:= \sup_{n \in \Z^{\ast}} \sum_{\substack{n_1, n_2 \in\Z^*\\n = n_1+n_2}} \ind_{B^{(j)}_2} \frac{|n|^{2s+2}}{|n_1|^{2s}|n_2|^{2s}},\quad j=1,2, 
% \end{align*}
%where $B^{(j)}_2$ is defined in \eqref{BB15}. Then, it's easy to see that 
\end{remark}

 Remark \ref{b11} will allow us to obtain local well-posedness for the coupled modulated system, as well as nonlinear smoothing, see Proposition \ref{PROP:main} (i) and (iii). The next lemma is instead related to persistency of regularity. 
\begin{lemma}\label{pers}
    Given $\rho\geq \frac{1}{2}$,  $\frac{1}{2}<\gamma<1$, $c$ as in \eqref{roots1}, $\nu_c$ as in Definition \ref{inf_index}, $0<\eps<1-\nu_c$, and $T>0$, let  $w$ be $(\rho,\g)$-irregular on $[0, T]$ in the sense of Definition~\ref{DEF:ir}. If  $s, s_0\in \R$ satisfy
\begin{align}
\label{condX1 persistenceofreg}
 s\ge 1-\rho(1-\nu_{c}-\eps), \quad s_0 > s,
\end{align}
then  $X^j$ belongs to $\cY^{s, s_0, \g}_2([0, T]\times \T)$
with the bound 
\begin{align*}
\|X^j\|_{\cY^{s,s_0,  \g}_2(T)} \les \|\Phi^w\|_{\W^{\rho,\g}_T}.
\end{align*}
\end{lemma}
\begin{proof}
In analogy with the proof of Lemma \ref{lem_nonlinsmooth}, we consider the case of $f_1,f_2$ having zero mean,  the other case being similar. Furthermore, we focus on  $j=1$, the case $j=2$ being analogous. 

 Let us consider the driver $X^{1,1}$ defined by restricting the contribution of $X^{1}$ to $|n_1|\ge |n_2|$. Namely, 
 \begin{align}\label{X11}
 2\mathcal{F}(X^{1,1}_{t,r}(f_1,f_2)(n)=in\sum_{ \substack{n_1, n_2 \in \Z^*\\n = n_1+n_2\\ |n_1|\ge |n_2|}}\Phi^{w}_{t,r}  (\Xi^{(1)}(\bar n))  \ft f_1(n_1)  \ft f_2 (n_2).
 \end{align}
  Thus, in view of Lemma \ref{LEM:OBS1}-(ii), it is enough to prove that 
  \begin{align}
\|X^{1,1}_{t,r}\|_{\cL^{s, s_0}_{2,1}(H^s(\T))}
\les 
\|\Phi^w\|_{\W^{\rho,\g}_T}
|t-r|^{\g}, \label{bound:persistenceofreg 1}  \\
\|X^{1}_{t,r}-X^{1,1}_{t,r}\|_{\cL^{s, s_0}_{2,2}(H^s(\T))}
\les 
\|\Phi^w\|_{\W^{\rho,\g}_T}
|t-r|^{\g}. \label{bound:persistenceofreg 2} 
\end{align}
By symmetry, it is also sufficient to prove \eqref{bound:persistenceofreg 1}. By \eqref{X11} and arguing as in the proof of Lemma \ref{lem_nonlinsmooth}  we obtain
\begin{align}
\|  X^{1,1}_{t,r} (f_1, f_2) \|^2_{H^{s_0}} \les
\sum_{k=1}^{5}  \widetilde{A}^{(1,k)}_{s,s_0}
\| \Phi^w \|^2_{\mathcal{W}^{\rho, \gamma}_{T}} |t - r|^{2\gamma} \| f_1 \|^2_{H^{s_0} (\T)} \| f_2 \|^2_{H^s (\T)},
\label{a15ineq}
\end{align}
%In particular, by Lemma \ref{lem_bi}
%\begin{align}
%\text{R.H.S. of \eqref{eq_bi_smooth}} \les  \widetilde{A}^{(1,k)}
%\| \Phi^w \|^2_{\mathcal{W}^{\rho, \gamma}_{T}} |t - r|^{2\gamma} \| f_1 \|^2_{H^{s_0} (\T)} \| f_2 \|^2_{H^s (\T)},
%\label{a15ineq}
%\end{align}
where \begin{align}
	\widetilde{A}^{(1,k)}_{s,s_0} &:= \sup_{n \in \Z^{\ast}} \sum_{\substack{n_1, n_2 \in\Z^*\\n = n_1+n_2\\ |n_1|\ge |n_2|}} \ind_{B^{(1)}_k} \frac {|n|^{2s_0+2}|n_1|^{-2s_0}|n_2|^{-2s}}{|\Xi^{(1)} (\bar{n})|^{2\rho}}\quad \text{if}\ k=1,2;
\label{A159}\\
\widetilde{A}^{(1,k)}_{s,s_0}&:= \sup_{n_1 \in \Z^{\ast}} \sum_{\substack{n, n_2 \in\Z^*\\n = n_1+n_2\\ |n_1|\ge |n_2|}} \ind_{B_k} \frac {|n|^{2s_0+2}|n_1|^{-2s_0}|n_2|^{-2s}}{|\Xi^{(1)} (\bar{n})|^{2\rho}}\quad \text{if}\ k=3;
\label{A2519}\\
\widetilde{A}^{(1,k)}_{s,s_0} &:= \sup_{n \in \Z^{\ast}} \sum_{\substack{n_1, n_2 \in\Z^*\\n = n_1+n_2\\ |n_1|\ge |n_2|}} \ind_{B_k} \frac {|n|^{2s_0+2}|n_1|^{-2s_0}|n_2|^{-2s}}{|\Xi^{(1)} (\bar{n})|^{2\rho}}\quad \text{if}\ k=4,5,
\label{A3519}
\end{align} 
and $B_k$, $B^{(1)}_k$ are defined as in \eqref{BB15}. 
Thus, \eqref{bound:persistenceofreg 1} follows once we prove that 
\begin{align}\label{Atilde16}
\widetilde{A}^{(1,k)}_{s,s_0}\les 1.
\end{align} 
To do so, we note that for $s_0>s$ and $|n_1|\ge |n_2|$ we clearly have 
\begin{align*}
 \frac {|n|^{2s_0+2}|n_1|^{-2s_0}|n_2|^{-2s}}{|\Xi^{(1)} (\bar{n})|^{2\rho}}\les  \frac {|n|^{2s+2}|n_1|^{-2s}|n_2|^{-2s}}{|\Xi^{(1)} (\bar{n})|^{2\rho}},
\end{align*}
which implies that $\widetilde{A}^{(1,k)}_{s,s_0}\les {A}^{(1,k)}_{s,s}$, where  ${A}^{(1,k)}_{s,s}$ is defined as in \eqref{A19}, \eqref{A219} and \eqref{A319}. As a consequence, in view of Remark \ref{b11}, we derive \eqref{Atilde16} under the assumption \eqref{condX1}. 
We have therefore concluded the proof by assuming $\ft{f_1}(0) =  \ft{f_2}(0) = 0$. In the general case, we can argue as in the proof of Lemma \ref{lem_nonlinsmooth}.

%By (i) of Corollary \ref{cor_res} with $\lambda=1$
%begin{align*}
%A^{(1,1)} \les \sup_{n \in \Z^{\ast}}  |n|^{-2s-4\rho+2} \sum_{|n_1| \sim |n|} 1 \les 1
%\end{align*}
%if $-2s-4\rho+3 \le 0$, which follows from \eqref{condX1 persistenceofreg}.

%By (ii) of Corollary \ref{cor_res} with $\lambda=1$,
%\begin{align*}
%A^{(1,2)} \les \sup_{n \in \Z^{\ast}}  |n|^{-2s-2\rho(1-\nu_c-\varepsilon)+2} \les 1
%\end{align*}
%if $-2s-2\rho(1-\nu_c-\varepsilon)+2 \le 0$, which is equivalent to \eqref{condX1 persistenceofreg}.

%By (iii) of Corollary \ref{cor_res} with $\lambda=1$,
%\begin{align*}
%A^{(1,3)} \les \sup_{n_1 \in \Z^{\ast}}  |n_1|^{-4s-4\rho} \sum_{|n| \ll |n_1|} |n|^{2s_0-2\rho+2} \les 1
%\end{align*}
%if $-2s_0-2s-4\rho+[2s_0-2\rho+3]_{\delta} \le 0$, which again follows from \eqref{condX1 persistenceofreg}.

%Note that, under the restriction $|n_1|\ge |n_2|$, we have that $A^{(1,4)}=0$. Similarly, by (iii) of Lemma \ref{lem_res} with $\lambda=1$,
%\begin{align*}
%A^{(1,4)} \les \sup_{n \in \Z^{\ast}}   |n|^{2s_0-2s-6\rho+2}\sum_{|n_1| \ll |n|} |n_1|^{-2s} \les 1
%\end{align*}
%if $2s_0-2s-6\rho+2+[-2s+1]_{\delta} \le 0$,
%which follows from (a) and (c).

%By (iii) of Lemma \ref{lem_res} with $\lambda=1$,
%\begin{align*}
%A^{(1,5)} \les  \sup_{n\in \Z^{\ast}}   |n|^{-6\rho+2} \sum_{|n_2| \ll |n|} |n_2|^{-2s} \les 1
%\end{align*}
%for $j=1,2$,
%if $-6\rho+2+[-2s+1]_{\delta} \le 0$,
%which follows \eqref{condX1 persistenceofreg}.
\end{proof}

\begin{lemma}\label{lem_galapprox}
Let $j=1,2$, $N\in \N$,  $\frac{1}{2}<\gamma<1$ and $0<\eps<1-\nu_c$. Let $X^{j}$ as in \eqref{driver1} and $X^{j,N}$ be the drivers defined in \eqref{trunc}. Assume that $\rho, s \in \R$ satisfy
\begin{align}
\label{galapprox_cond}
\rho > \frac{1}{2}, \qquad 
s > 1 - \rho (1-\nu_{c} - \eps).
\end{align}
%Moreover, we assume $\ft{f}_1(0)=0$ for $j=2$.
Then $X^{j,N}$ converges to $X^{j}$ in 
 $  \cX^{s, \g}_2([0, T]\times \T)$, as $N \to \infty$.
\end{lemma}
\begin{proof}
Similarly to the proofs of Lemmas \ref{lem_nonlinsmooth} and \ref{pers}, we assume first that $\ft f_1(0)=\ft f_2(0)=0$. 
  From  \eqref{F_X}, \eqref{F_X2}, \eqref{trunc}, we have
    \begin{align*}
& \F\big( X^{j}_{t,r}(f_1,f_2)\big)(n) - \F\big( X^{j,N}_{t,r}(f_1,f_2)\big)(n)\\
& \quad =in \sum_{ \substack{n_1, n_2 \in \Z^*\\n = n_1+n_2}}
 \ind_{\{\max(|n|, |n_1|, |n_2| )>  N\}}
\Phi^w_{t,r}(\Xi^{(j)}(\bar n))
\ft f_1(n_1) \ft f_2(n_2).
    \end{align*}
Noting that
\begin{equation}
 \ind_{\{\max(|n|, |n_1|, |n_2|) >  N\}} 
 \les N^{-\eta} 
\max(|n|, |n_1|, |n_2| )^{\eta}=: N^{-\eta} |n_{\text{max}}|^{\eta}, \quad \eta>0,
\label{K8}
\end{equation}
and arguing as in \eqref{a10ineq} yields
\begin{equation*}
\begin{split}
\|  (X^{j,N}_{t, r} - &X^{j}_{t,r})(f_1, f_2) \|^2_{H^{s}}\\
 &\les N^{-2\eta}
  \sum_{ n\in\Z^*}\bigg( \sum_{\substack{n_1, n_2 \in\Z^*\\n = n_1+n_2}}\frac{|n|^{s+1}|\Phi^{w}_{t,r}(\Xi^{(j)}(\bar n))| |n_{\textup{max}}|^{\eta}}{\jb{n_1}^s\jb{n_2}^s}|\jb{n_1}^s\ft{f}_1(n_1)| |\jb{n_2}^s\ft{f}_2 (n_2)|\bigg)^2\\
 &\les  N^{-2\eta}\sum_{k=1}^{5}\bar{A}^{(j,k)}_{s}
\| \Phi^w \|^2_{\mathcal{W}^{\rho, \gamma}_{T}} |t - r|^{2\gamma} \| f_1 \|^2_{H^s (\T)} \| f_2 \|^2_{H^s (\T)},
\end{split}
\end{equation*}
where 
\begin{align}
\bar{A}^{(j,k)}_{s} &:= \sup_{n \in \Z^{\ast}} \sum_{\substack{n_1, n_2 \in\Z^*\\n = n_1+n_2}} \ind_{B^{(j)}_k} \frac {|n|^{2s+2}|n_1|^{-2s}|n_2|^{-2s}  |n_{\text{max}}|^{2\eta}}{|\Xi^{(j)} (\bar{n})|^{2\rho}}\quad\! \!\text{if}\ (j,k)\in\{(1,1),(2,1),(1,2),(2,2)\};\notag\\
\bar{A}^{(j,k)}_{s}&:= \sup_{n_1 \in \Z^{\ast}} \sum_{\substack{n, n_2 \in\Z^*\\n = n_1+n_2}} \ind_{B_k} \frac {|n|^{2s+2}|n_1|^{-2s}|n_2|^{-2s}  |n_{\text{max}}|^{2\eta}}{|\Xi^{(j)} (\bar{n})|^{2\rho}}\quad\!\! \text{if}\ (j,k)\in \{(1,3),(2,3)\};
\label{AA219}\\
\bar{A}^{(j,k)}_{s} &:= \sup_{n \in \Z^{\ast}} \sum_{\substack{n_1, n_2 \in\Z^*\\n = n_1+n_2}} \ind_{B_k} \frac {|n|^{2s+2}|n_1|^{-2s}|n_2|^{-2s}  |n_{\text{max}}|^{2\eta}}{|\Xi^{(j)} (\bar{n})|^{2\rho}}\quad \text{if}\ (j,k)\in\{(1,4),(2,4),(1,5),(2,5)\}
\label{AA319},
\end{align} 
and $B_k$, $B^{(j)}_k$ are defined as in \eqref{BB15}. Then, by performing an analogous computation to the one in the proof of Lemma \ref{bi_nonlinsmooth}, one can see that $\bar{A}^{(j,k)}_{s}\les 1$ if 
\begin{align*}
    \rho  \geq \frac{1}{2}+\frac{\eta}{2}, \quad \text{and}\quad
s  \geq 1 - \rho (1-\nu_{c} - \eps)+\eta. 
\end{align*}
%In particular, by chooching $\eta$ sufficiently small we conclude in view of \eqref{galapprox_cond}. 
The general case where only $\ft f_1(0)=0$ (for $j=2$) is analogous. We have therefore proved that, by assuming \eqref{galapprox_cond}, there exists $\eta$ sufficiently small such that
\begin{equation}\label{bi_galapprox}
\|X^{j,N}_{t, r} - X^{j}_{t,r}\|_{\cL_2(H^s(\T))}
\les 
N^{-\eta} \|f_1\|_{H^s}\|f_2\|_{H^s}\|\Phi^w\|_{\W^{\rho,\g}_T}|t-r|^{\g}.
\end{equation}
In particular, by Lemma \ref{LEM:OBS1}-(iv), we conclude that 
$X^{j,N}$ converges to $X^{j}$ in 
 $  \cX^{s, \g}_2([0, T]\times \T)$.
\end{proof}

\subsection{Proof of Theorems \ref{THM:11} (i)-(ii)}

Finally, we have all the ingredients to prove local well-posedness, as well as nonlinear smoothing and persistency of regularity for \eqref{cmkdv}. To this aim, let us consider the following system of Young differential equations:
\begin{equation}
\label{YDE1}
\begin{cases}
\uu(t) = u_0 + \I^{X^{1}}(\vv,\vv)(t),\\
\vv(t)=v_0+\I^{X^{2}}(\uu,\vv)(t),
\end{cases}
\end{equation}
where $X^{j}$, $j=1,2$ are as in \eqref{driver1} and 
\begin{align}\label{z_vex_def}
\vec{z}_0 \coloneqq (u_0,v_0)\in H^s_0(\T)\times H^s(\T).\end{align}

\noi To simplify the notation, we denote  $\vec{X}:=(X^1,X^2)$ and 
\begin{align}\label{vec_I}
\I^{\vec X}(\vec{\bf z})(t)
:=(\I^{X^{1}}(\vv,\vv)(t),\,\I^{X^{2}}(\uu,\vv)(t)).
\end{align}
As a result, \eqref{YDE1} is now rewritten as follows:
\begin{align}\label{sYDE}
\vec{\bf z}(t)=\vec{z}_0+\I^{\vec X}(\vec{\bf z})(t).
\end{align}

\begin{proposition}
\label{PROP:main}
Given $\alpha\in (0,4)\setminus\{1\}$,  $\rho \ge\frac 12$,  $\frac12< \g < 1$, $c$ as in \eqref{roots1}, $\nu_c$ as in Definition \ref{inf_index}, $0<\eps<1-\nu_{c}$, and $T> 0$, let
$w$ be $(\rho,\g)$-irregular on $[0, T]$ in the sense of Definition~\ref{DEF:ir}, and $\vec{z_0}$ as in \eqref{z_vex_def}. Then, the following hold:

\noi \textup{(i) (local well-posedness)}
if $s$ satisfies \eqref{condX1} then 
 the system of nonlinear Young differential equations \eqref{sYDE} with driver $\vec{X}$ is locally well-posed in $\H^s(\T)$.
More precisely, given  $\vec{z_0}\in \H^s(\T)$, 
there exist $C_0 >0$ and   $\ta>0$, independent of $\vec{z}_0$ and $\vec{X}$, 
and 
 a unique solution $\vec{\bf z} \in \cC^\g([0, \tau]; \H^s(\T))$
   to 
 \eqref{sYDE}  
with $\vec{\bf z}|_{t= 0} = \vec{z}_0$, where the local existence time $\tau = 
\tau\big(\|\vec{z}_0\|_{\H^s(\T)}, \|\vec{X}\|_{\cX^{s, \g}_2(T)\times \cX^{s, \g}_2(T)}\big) \in (0, 1]$
satisfies 
\begin{align}
\tau \ge  C_0 
\Big(\|\vec{X}\|_{\cX^{s, \g}_2(T)\times \cX^{s, \g}_2(T)} (1+ \|\vec{z}_0\|_{\H^s})\Big)^{-\ta}.
\label{YD1}
\end{align}

\noi
Moreover, for every $\s>0, \gamma+\s>1$
\begin{align}
\|\!\, \vec{\bf z} \|_{\CC^\s_\tau \H^s_x}\le C_\s\|\vec{z}_0\|_{\H^s}, 
\label{YD1a}
\end{align}

\noi
while we have
\begin{align}
\| \vec{\bf z} \|_{\CC^\g_\tau \H^s_x}\les\|\vec{X}\|_{\cX^{s, \g}_2(T)\times \cX^{s, \g}_2(T)}(1+\|\vec{z}_0\|_{\H^s})^2.
\label{YD1b}
\end{align}

\smallskip

\noi
\textup{(ii) (persistence of regularity).}
In addition to the hypotheses of Part \textup{(i)},  suppose that $\vec{z}_0\in~\H^{s_0}(\T)$ for some $s_0 > s$.
Then, $\vec{X}\in \cY^{s, s_0, \g}_2([0, T]\times \T)\times \cY^{s, s_0, \g}_2([0, T]\times \T)$. Moreover, 
by possibly making $\tau > 0$ smaller by a multiplicative constant 
\textup{(}still satisfying \eqref{YD1}\textup{;} in particular $\tau > 0$
depends on the $\H^s$-norm of the initial data $\vec{z}$
but not on its $\H^{s_0}$-norm\textup{)}, 
we have $\vec{\bf z} \in \cC^\g([0, \tau]; \H^{s_0}(\T))$.

\smallskip

\noi
\textup{(iii) (nonlinear smoothing).}
In addition to the hypotheses of Part \textup{(i)}, assume that $r>0$ satisfies \eqref{nsmooth}. Then, $\vec{X} \in \cX^{s, s + r, \g}_2 ([0, T]\times \T)  $, and
 $\I^{\vec X}(\vec{\bf z}) \in \cC^\g([0, \tau]; \H^{s+r}(\T))$, 
 where $\vec{\bf z}$ is the solution to \eqref{sYDE}
 constructed in Part \textup{(i)}.

\smallskip

\noi
\textup{(iv) (convergence).}
In addition to the hypotheses of Part \textup{(i)}, 
suppose  that \eqref{galapprox_cond}  hold. Then,
% $\{X^{j,N}\}_{N \in \N}\subset 
% \cX^{s, \g}_2([0, T]\times \T)$
% converges to $X^j$ in 
% $ \cX^{s, \g}_2([0, T]\times \T)$ as $N \to \infty$.
%Then, 
by possibly making $\tau > 0$ smaller by a multiplicative constant 
\textup{(}still satisfying~\eqref{YD1}\textup{)},  the solution 
$\vec{\bf z}^N$  %\in \CC^\g([0, \tau]; H^s(\M))$
to the following Young differential equation\textup{:}
\begin{equation}
\label{YDE2}
\vec{\bf z}^N(t) = \vec{z}_0 + \I^{\vec{X^{N}}}(\vec{\bf z}^N)(t)
\end{equation}
\noi
converges as $N$ tends to infinity in $\CC^\g([0, \tau]; \H^s(\T))$ to the solution $\vec{\bf z}$ to \eqref{YDE1}
constructed in Part \textup{(i)}.
Moreover, 
given any $r > 0$, 
the rate of convergence of $\vec{\bf z}^N$ to $\vec{\bf z}$ is
uniform in $\vec{z}_0 \in B_r$ 
where
$B_r$ denotes the ball 
in $\H^s(\T)$
of radius $r > 0$ centered at the origin.
\end{proposition}
\begin{proof}
Under the assumptions \eqref{condX1}, by combining Remark \ref{b11},  Lemma \ref{LEM:OBS1} with   Proposition \ref{PROP:young1}, we infer that the nonlinear Young integrals  with drivers $X^{j}$ in the right hand side of \eqref{YDE1} are well defined in $\cC^\g([0, \tau]; \H^s(\T))$, see \cite[Remark 3.6]{CGLLO}. 
In addition, from  Lemmas \ref{lem_nonlinsmooth}, \ref{pers}, \ref{lem_galapprox} and arguing as in \cite[Proposition 3.8]{CGLLO}  we get local well-posedness, nonlinear smoothing (by further assuming \eqref{nsmooth}),  persistency of regularity,  and convergence of Galerkin approximation (by further assuming \eqref{galapprox_cond}).
\end{proof}

\begin{remark}\label{theta}\rm
(i) Note that, for the proof of Proposition \ref{PROP:main} we fix any $0<\s<\gamma$ such that $\g+\s>1$. 
Next, 
see \cite[(3.54), (3.64)]{CGLLO},
the exponent $\theta$ in \eqref{YD1} satisfies $\theta=\frac{1}{\gamma-\s}$. 
\smallskip

\noi (ii) As pointed out in \cite[Remark 3.9 (i)]{CGLLO}, by applying Proposition \ref{PROP:main} to $\vec{X}\in \cX^{s, \g}_2(\R_{+}\times \T)\times \cX^{s, \g}_2(\R_{+}\times \T)$,  we obtain the following blowup alternative; let $T_{\ast} \in (0, \infty]$ be the maximal time of existence of the solution $(\uu,\vv)$ to  \eqref{sYDE} Then, we have either
\begin{align} 
\label{BA1}
    T_{\ast} = \infty \qquad \text{or} \qquad \lim_{t \nearrow T_{\ast}} \| (\uu,\vv)(t) \|_{\H^{s}(\T)} = \infty. 
\end{align}
\end{remark}

\section{Global well-posedness}\label{sec6}
The goal of this section is to prove global well-posedness for \eqref{cmkdv}, namely Theorem \ref{THM:11} (iii). To do so, we first establish  the conservation of the $L^2$-norm
of the solution to the coupled modulated KdV system \eqref{cmkdv} constructed in Proposition \ref{PROP:main}. In particular, by combining such conservation with persistency regularity we obtain global well-posedness for $s\ge 0$ (by further assuming \eqref{condX1}); see Corollary \ref{GWP_2}.

\begin{proposition}
\label{PROP:L2}

Given $\alpha\in (0,4)\setminus\{1\}$,  $s \ge 0$,  $\rho > 0$,  $\frac 12 < \g < 1$ and $T> 0$, 
let  $w$ be $(\rho,\g)$-irregular on $[0, T]$. 
Assume that $(u,v)$ solves the modulated KdV system \eqref{cmkdv} on $\T$,  with
modulated interaction representation $(\uu,\vv)$ defined in \eqref{YDE1}
belonging to  $\CC^\g([0, T];\H^s(\T)) $.  Then, we have
\begin{align}
\| (u,v)(t) \|_{L^2\times L^2} = \|( u_0 ,v_0)\|_{L^2\times L^2}.
\label{GK0b}
\end{align}

\noi
for any $0 \le t \le T$.
\end{proposition}
\begin{proof}
Let $j=1,2$,  $X^{j}$ be as in \eqref{driver1}. Then, from the unitarity of $\uw(t)$ and $\uwa(t)$ on $L^2(\T)$, we have
\begin{equation}
\label{zero_term}
\begin{split}
&\jb{f_1,X^{1}_{t,r}(f_2,f_2)}_{L^2}+\jb{f_2,X^{2}_{t,r}(f_1,f_2)}_{L^2}\\
&\phantom{XXXX}=\frac 12\int_{r}^{t}\int_{\T}f_1\cdot\uw(t')^{-1}\dx \big(\big(\uwa(t')f_2\big)^2\big)dxdt'\\
&\phantom{XXXXXXX}+ \int_{r}^{t}\int_{\T}f_2\cdot\uwa(t')^{-1}\dx \big(\uw(t')f_1\cdot \uwa(t')f_2\big)dxdt'\\
&\phantom{XXXX}=\frac 12\int_{r}^{t}\int_{\T}\uw(t') f_1\cdot \dx \big(\big(\uwa(t')f_2\big)^2\big)dxdt'\\
&\phantom{XXXXXXX}+ \int_{r}^{t}\int_{\T}\uwa(t') f_2\cdot \dx \big(\uw(t')f_1\cdot \uwa(t')f_2\big)dxdt'\\
&\phantom{XXXX}=\int_{r}^t \int_{\T}\dx \big(\uw(t')f_1\cdot \big(\uwa(t')f_2\big)^2\big)dxdt'=0.
\end{split}
\end{equation}
As a result, we obtain 
\begin{align}
\label{GK1}
\begin{aligned}
\| (\uu,\vv)(t) \|^2_{L^2}
&=\| (\uu,\vv)(r) \|^2_{L^2}
+2\big\langle \uu(r),X^{1}_{t,r}(\vv(r),\vv(r)) 
\big\rangle_{L^2}
\\
&\phantom{XXX}+2\big\langle \vv(r),X^{2}_{t,r}(\uu(r),\vv(r)) 
\big\rangle_{L^2}+\|\uu(t) -\uu(r)  \|^2_{L^2} \\
&\phantom{XXX}+
\|   \vv(t)-\vv(r)   \|^2_{L^2}
+R^{1}_{t,r}+R^{2}_{t,r},
\end{aligned}
\end{align}
where $R^{1}_{t,r}$, $R^{2}_{t,r}$ are defined by 
\begin{equation}\label{R_boh}
\begin{split}
R^{1}_{t,r}=2\big\langle 
\uu(r) ,  \uu(t)-  \uu(r) -X^{1}_{t,r}(\vv(r),\vv(r))\big\rangle_{L^2},\\
R^{2}_{t,r}=2\big\langle 
\vv(r) ,  \vv(t)-  \vv(r) -X^{2}_{t,r}(\uu(r),\vv(r))\big\rangle_{L^2}. 
\end{split}
\end{equation}

By Proposition \ref{PROP:young1}, 
we have 
 $\uu,\vv\in \CC^\g([0, T];  L^2(\T))$
with $\g>\frac 12$. Then, 
from \eqref{R_boh}, \eqref{YDE1} and~\eqref{Y2} in Proposition \ref{PROP:young1}
(see also \eqref{Ja3x}), we infer
\begin{align}
\begin{split}
|R^{1}_{t,r}| 
&\leq 
\|\uu(r)\|_{L^2}
\|\uu(t)- \uu(r) -X^{1}_{t,r}(\vv(r),\vv(r))\|_{L^2}
\\
&= 
\|\uu(r)\|_{L^2}
\|(\updl\I^{X^{1}}(\vv,\vv))_{t, r}  -X^{1}_{t,r}(\vv(r),\vv(r))\|_{L^2}
\\
&\le
C\big( \|(\uu, \vv)\|_{\CC^\g_{T}L^2\times L^2} \big)  \|X^{1}\|_{\cX^{s, \g}_2(T)} |t-r|^{2\g}.
\end{split}
\label{GK2}
\end{align}
Similarly, 
\begin{align}
\label{GK2bis}
|R^{2}_{t,r}|\le C\big(\|(\uu,\vv)\|_{\CC^\g_{T}L^2\times L^2}\big)  \|X^{2}\|_{\cX^{s, \g}_2(T)} |t-r|^{2\g}.
\end{align}
Furthermore, 
\begin{align}
\|\uu(t)-\uu(r)\|_{L^2}^2+\|\vv(t)-\vv(r)\|^2_{L^{2}}
\leq 
(\|\uu\|_{C^\g_T L^2}^2+\|\vv\|^2_{C^\g_T L^{2}})
|t-r|^{2\g}
\label{GK3}
\end{align}

\noi
for any  $0\leq r<t\leq T$.
Thus, by combining \eqref{GK1}, \eqref{GK2}, \eqref{GK2bis} with \eqref{GK3}, we derive

\begin{equation}
\begin{split}
&\big|
 \|  (\uu,\vv)(t)  \|^2_{L^2}-\|(\uu,\vv)(r)\|^{2}_{L^{2}}
 \big|\\
 &\phantom{XXXXXXXX}\le
 C\big(\|(\uu,\vv)\|_{\CC^\g_{T} L^2\times L^2},   \|\vec{X}\|_{\cX^{s, \g}_2(T)\times \cX^{s, \g}_2(T)}\big)
 |t-r |^{2\g}.
 \end{split}
 \label{GK4}
\end{equation}
Hence, \eqref{GK0b} follows by \eqref{GK4} and the fact that $2\g>1$. 
\end{proof}

\begin{corollary}\label{GWP_2}
 Given $\alpha\in (0,4)\setminus\{1\}$, $\rho>0$, $\frac 12 < \g < 1$ and $w$ be $(\rho,\g)$-irregular on $\R_+$. Assume that $\rho$ and $s\ge 0$ satisfy \eqref{condX1} and $\vec{z_0}$ is as in \eqref{z_vex_def}. Then, the Young differential system of equations \eqref{sYDE} is globally well posed in $\H^s(\T)$.
\end{corollary}
\begin{proof}
As a consequence of Proposition \ref{PROP:L2}, 
by combining  the persistence of regularity result in Proposition~\ref{PROP:main}\,(ii) with a standard argument, we deduce global well-posedness for \eqref{cmkdv}.  
\end{proof}

\subsection{Scaled modulated \texorpdfstring{$I$}{I}-KdV system}
\label{Subsec4}
In view of Corollary \ref{GWP_2}, in order to complete the proof of Theorem \ref{THM:11}-(ii), it remains to consider the global well-posedness for $s<0$. To do so, we employ a one-parameter family of scaling transforms introduced by the second author and collaborators in \cite{GLO}.

Let $b \in \R$ and $\ld\ge 1$. Given $u$ defined on $[0,T]\times \T$  and a modulation $w$ on $[0,T]$, we define $w^{\ld}$ as in \eqref{scalw} such that $\partial_t w^{\ld}(t)=\ld^{3-b}\partial_t w(\ld^{-3}t)$, and $u^\ld$ as
\begin{align} 
\label{scaling1}
u^{\lambda}(t,x)& := \lambda^{- b + 1}u(\lambda^{-b} t, \lambda^{-1} x).
%v^{\lambda} (t, x)& = \lambda^{- b + 1}v (\lambda^{-b} t, \lambda^{-1} x)
%\label{scaling2}
%w^{\lambda} (t) &:= \lambda^{3} w (\lambda^{-b} t),
\end{align} 
 As a result, we formally see that  a pair $(u,v)$ solves \eqref{cmkdv} with initial data $(u_0,v_0)$ if and only if the corresponding scaled pair $(u^{\ld},v^{\ld})$ defined via \eqref{scaling1} solves
\begin{equation} 
\label{scaleKdV}
\begin{cases}	
\dt u^{\lambda} + \dx^3 u^{\lambda} \cdot \dt w^{\lambda} = \frac{1}{2}\dx (v^{\lambda})^2,\\
\dt v^{\lambda} + \al \dx^3 v^{\lambda} \cdot \dt w^{\lambda} = \dx (u^{\lambda}v^{\ld}),
\end{cases}
\quad (t,x)\in [0, \lambda^b T] \times \T_{\lambda},
\end{equation} 

\noi 
with the scaled initial data 
\begin{equation} 
\label{scaleini}
\begin{split}
&(u_0^{\lambda},v_0^{\lambda}) (x) = \lambda^{-b + 1} (u_0,v_0) (\lambda^{-1} x).\\
%&v_0^{\lambda} (x) = \lambda^{-b + 1} v_0 (\lambda^{-1} x)
\end{split}
\end{equation} 
%\noi where the scaled modulation $w^{\lambda}$ is given by \eqref{scaling2}. 
Note that, when $b=3$, \eqref{scaling1}, \eqref{scalw} reduces to the standard KdV scaling used in \cite{CGLLO}. On the other hand, in \cite{GLO} the authors  observed that the presence of the modulation leads to an additional degree of freedom. The idea is then to combine the $I$-method and the sewing lemma with the scalings \eqref{scaling1}, \eqref{scalw}, for a suitable $b$ (see \eqref{b_c}). 

Note that, from \eqref{scaleini}, we infer  $\F_{\T^{\ld}}(u_0^{\ld}) (n) = \ld^{-b+2} \F_{\T} (u_0) (\ld n)$, $n \in \Z_{\ld}$, which yields
\begin{align} 
\label{scaling3}
\| u_0^{\lambda} \|_{\dot{H}^s (\T_{\lambda})} = \lambda^{-b + \frac 32 - s} \| u_0 \|_{\dot{H}^s (\T)} 
\end{align} 
\noi 
for any $s \in \R$, while, for the non-homogeneous Sobolev spaces, we have 
\begin{align}\label{scaling3bis}
\|u_0^\ld\|_{ H^s(\T_\ld)}
\le  \ld^{-b+\frac 32 - s} \|u_0\|_{ H^s(\T)}
\end{align}
for $s\le 0$. In particular, if $b>\frac{3}{2}-s$, we can make
the $\dot H^s(\T_\ld)$-norm (respectively $H^s(\T_\ld)$-norm) of the scaled initial data $u_0^\ld$
small by choosing $\ld\gg 1$, depending either on $\|u_0\|_{H^s(\T)}$ or $\|u_0\|_{\dot H^s(\T)}$.

In order to give a proper meaning to the system \eqref{scaleKdV} we can argue as in Section \ref{sec4}. Namely, we first recast \eqref{scaleKdV} in terms of the Duhamel formulation 
for the modulated interaction representation 
\begin{align}
	(\uu^\ld,\vv^\ld)(t)  =( U^{w^\ld}(t)^{-1} u^\ld(t), U^{w_\al^\ld}(t)^{-1} v^\ld(t)),
	\label{ME0}
\end{align}
\noi
%is given by 
%\begin{equation}
%\begin{cases}
%\uu^{\ld} (t)=u^\ld_0
%+\frac{1}{2}\int_0^t   U^{w^\ld}(t')^{-1}  \dx\big(  ( U^{w_\al^\ld} (t') \vv^\ld(t'))^2   \big)dt',\\
%\vv^{\ld} (t)=v^\ld_0
%+\int_0^t   U^{w_\al^\ld}(t')^{-1}  \dx\big(  ( U^{w^\ld} (t') \uu^\ld(t'))\cdot ( U^{w_\al^\ld} (t') \vv^\ld(t'))  \big)dt'.
%\end{cases}
%\label{Dul3}
%\end{equation}

\noi
with the associated drivers $X^{1,\ld}, X^{2,\ld}$ 
defined as follows:
\begin{equation}
\begin{split}
X_{t,r}^{1,\ld} (f_1, f_2)
&=\frac{1}{2}\int_{r}^{t}
U^{w^\ld}(t')^{-1}
\dx \big( (U^{w_\al^\ld}(t') f_1  )(
U^{w_\al^\ld}(t') f_2)\big) dt',\\
X_{t,r}^{2,\ld} (f_1, f_2)
&=\int_{r}^{t}
U^{w_\al^\ld}(t')^{-1}
\dx \big( (U^{w^\ld}(t') f_1  )(
U^{w_\al^\ld}(t') f_2)\big) dt',
\end{split}
\label{Xld1}
\end{equation}

\noi
for functions $f_1$ and $  f_2$ on $\T_\ld$. %See also Subsection \ref{SUBSEC:I5} for the real line case (Theorem \ref{THM:2}\,(ii)).
%In view of \eqref{I1}, 
%it then suffices to control the so-called
%modified energy $\| Iu(t)\|_{L^2}^2$
%.\footnote{As we see below, given a target time $T> 0$, 
%	we in fact apply a scaling first and study the scaled modified energy in \eqref{ME1}.}
%
%In Subsection \ref{SUBSEC:I1}, 
%we go over basic scaling properties
%of the periodic modulated KdV.
%In Subsection~\ref{SUBSEC:I2}, 
%we then study  local well-posedness
%of the so-called (scaled) modulated $I$-KdV equation~\eqref{kdv4}.
%After establishing a crucial commutator estimate
%(Proposition \ref{PROP:com}) in Subsection~\ref{SUBSEC:I4}, 
%we establish global well-posedness
%by combining the commutator estimate (Proposition \ref{PROP:com})
%with the sewing lemma (Lemma \ref{LEM:sew}).
%It is worthwhile to note that we need to make use
%of the sewing lemma even in the globalization argument via the $I$-method
%which is a new feature in the current modulated setting.
%We also discuss the real line case in Subsection~\ref{SUBSEC:I5}.
Note that, from
\eqref{Xld1}
with 
\eqref{FT3}, 
\eqref{rho2}, 
\eqref{scalw}, 
and a  change of variables,   we have
\begin{equation}\label{driversF}
\begin{split}
 2\F_{\T_\ld}\big( X^{1,\ld}_{t,r} ( f _1,  f _2)\big)(n)
= \frac{in}{\ld} \sum_{ \sub {n_1,n_2 \in\Z_{\ld} \\ n=n_1+n_2}}
\Phi^{w^\ld}_{t,r}(\Xi^{(1)} (\bar n))
\ft  f _1(n_1) \ft  f _2(n_2),\\
 \F_{\T_\ld}\big( X^{2,\ld}_{t,r} ( f _1,  f _2)\big)(n)
= \frac{in}{\ld} \sum_{ \sub {n_1,n_2 \in\Z_{\ld} \\ n=n_1+n_2}}
\Phi^{w^\ld}_{t,r}(\Xi^{(2)} (\bar n))
\ft  f _1(n_1) \ft  f _2(n_2).
\end{split}
\end{equation}
Moreover, by arguing as in \cite[Lemma 7.3]{CGLLO}, we have  the equivalence between the scaled system \eqref{scaleKdV} and the original one in \eqref{cmkdv}; see Lemma \ref{boooh}.

\begin{lemma}\label{boooh}
Let $\ld\ge 1$ and $b\in \R$. Given  $\alpha\in (0,4)\setminus\{1\}$,  $\rho \ge\frac 12$,  $\frac12< \g < 1$, $T> 0$, 
and $w$ be $(\rho,\g)$-irregular on $[0, T]$. Fix $s\in \R$ satisfying \eqref{condsintro} and  $(u_0,v_0)\in H_0^s(\T)\times H^s(\T)$. Then, $(u,v)$ solves the modulated KdV system \eqref{cmkdv} on $[0,\tau]\times \T$ for some $0\le \tau\le T$ if and only if the scaled pair $(u^\ld,v^\ld)$ defined in \eqref{scaling1}  solves \eqref{scaleKdV} on $[0,\ld^b\tau]\times \T_\ld$ with scaled initial data defined in \eqref{scaleini}.
\end{lemma}
% We remind the reader to Remark \ref{sol} for the notion of solution. 

%In this section, we prove global well-posedness of 
%the coupled modulated   KdV system \eqref{cmkdv}  in negative Sobolev spaces.
 
In what follows, we first introduce the scaled modulated $I$-KdV system, and then adapt the so-called $I$-method ($=$ the method of almost conservation laws), 
introduced by Colliander, Keel, Staffilani, Takaoka, and Tao
\cite{CKSTT03}, to the  modulated setting.

Fix   $s<0$.
Given $N \geq 1$,
we define a smooth, even, 
non-increasing (in $|\xi|$) 
function 
$m_{s, N} \in C^\infty(\R; [0, 1])$ 
by setting
\begin{align}
	m_{s, N}(\xi)=
	\begin{cases}
		1, 
		&\text{for }
		|\xi|\leq N,   \\
		\frac{ |\xi|^s }{ N^s }.
		&\text{for }
		|\xi|\geq  2N.
	\end{cases}
	\label{Iop1}
\end{align}

\noi
Then, we define the so-called $I$-operator $I = I_{s, N}$
to be the Fourier multiplier operator with the multiplier $m_{s, N}$:
\begin{align}
	\ft{I_{s, N}f}(\xi)=m_{s, N}(\xi)\ft{f}(\xi),
	\label{Ix2}
\end{align}

\noi satisfying  the following bounds
\begin{align}
	\| f \|_{\dot H^s (\T_\ld) } \les \ld ^{-s}\| I f \|_{L^2(\T_\ld)}
	\les N^{-s} \| f \|_{H^s(\T_\ld)}.
	\label{I1}
\end{align}

By applying the $I$-operator defined in \eqref{Ix2} to the scaled modulated KdV system \eqref{scaleKdV}, we obtain that the pair $(Iu^{\ld}, Iv^{\ld})$ solves the following system on the dilated torus $\T_\ld$
\begin{equation}
\begin{cases}
\label{SMIKdV} 
\dt I u^{\lambda} + \partial_x^3 I u^{\lambda} \cdot \dt w^{\lambda} = \frac{1}{2}\dx I \big((I^{-1}Iv^{\lambda})(I^{-1}Iv^{\lambda}) \big),\\
\dt I v^{\lambda} + \al\partial_x^3 I v^{\lambda} \cdot \dt w^{\lambda} = \dx I \big((I^{-1}Iu^{\lambda})(I^{-1}Iv^{\lambda}) \big),
\end{cases}
\end{equation} 
with initial data $(I u_0^\ld, Iv^\ld_0)$.
Then, for  $j=1,2$,  similarly to \eqref{F_X}, \eqref{F_X2}, we define the associated drivers $Y^{j,\lambda}$,  by
\begin{equation} 
\begin{split}
\label{X1d2}
Y^{j,\lambda} (f_1, f_2) = I X^{j,\lambda} (I^{-1} f_1, I^{-1} f_2)\\
%Y^{2,\lambda} (f_1, f_2) = I X^{2,\lambda} (I^{-1} f_1, I^{-1} f_2) 
\end{split}
\end{equation} 
\noi 
for functions $f_1$ and $f_2$ on $\T_{\lambda}$, where $X^{j,\lambda}$ are defined as in \eqref{Xld1}.

Next, from Lemma \ref{LEM:tri2} (and in particular Corollary \ref{COR:I2}) we deduce local well-posedness in $L^2(\T_\ld)\times L^2(\T_\ld)$ of the scaled modulated $I$-KdV system \eqref{SMIKdV}.
%{\Bl{As it is customary,  in order to prove that the $H^s\times H^s$ norm of a solution does not blow-up in a finite time, we will employ a scaling argument and prove the scaled energy of a solution remains bounded (after choosing in a suitable way the size of the scaling parameter, see \eqref{ME}).}} 

%Note that $\dx I \big( u^{\lambda}v^{\ld} \big) = \dx I \big( (I^{-1} I u^{\lambda}) (I^{-1} I v^{\lambda}) \big)$. 

%Note that, for the convenience of the reader, in the following statements we use the same conventions already highlighted in Remark \ref{rem_0}.

%\begin{proof}
%By Lemma \ref{lem_Phi} and the H\"older inequality, \eqref{eq_lem_bi} is bounded by
%\begin{equation*}
%\begin{split}
%& \lambda^{2b(1-\gamma)-6\rho+1}\|\Phi^{w}\|^2_{\W^{\rho,\g}_{T}}|t-r|^{2\gamma}\sum_{n \in \Z_\lambda}
%\sum_{ \substack{n_1, n_2 \in \Z_\lambda\\n = n_1+n_2}} |F(\bar n)|^2 |a(\bar n) |^{-2\rho}
%| \ft f(n_1) |^2 \|g\|_{L^2(\T_\lambda)}^2\\
%& \le \lambda^{2b(1-\gamma)-6\rho+2}C_0\|\Phi^{w}\|^2_{\W^{\rho,\g}_{T}}|t-r|^{2\gamma}
%\|f\|_{H^s(\T_\lambda)}^2
%\|g\|_{H^s(\T_\lambda)}^2.
%\end{split}
%\end{equation*}
%\end{proof}

\remark \rm For the rest of the section, given $0<\eps<1-\nu_c$ and $\rho\ge \frac{1}{2}$,  our basic assumption on $s$ to guarantee local well-posedness, will be 
\begin{align}\label{condsneg}
 1-\rho(1-\nu_{c}-\eps)\le s<0,
\end{align}
see again \eqref{condsintro}.
In particular, for the restriction \eqref{condsneg} to be possible we need to further assume
\begin{align}\label{condrhog}
\rho>\frac{1}{1-\nu_c-\eps}.
\end{align}

\begin{lemma} 
\label{LEM:tri2}
Let $\ld\ge 1$ and $b \in \R$. Given $\alpha\in (0,4)\setminus\{1\}$, $0<\eps<1-\nu_c$, $\rho$ satisfying \eqref{condrhog}, $\frac 12 < \gamma < 1$, and $T > 0$, let $w$ $(\rho, \gamma)$-irregular on $[0, T]$. Suppose that $s$ satisfies \eqref{condsneg}. Then, the driver $Y^{j,\lambda}$ defined in \eqref{X1d2} belongs to $\mathcal{X}_2^{0, \gamma} ([0, \lambda^b T] \times \T_{\lambda})$ defined in \eqref{X1} satisfying  
\begin{align} 
\label{MIF1}
\| Y_{t, r}^{j,\lambda} \|_{\mathcal{L}_2 (L^2 (\T_{\lambda}))} \les \lambda^{-\frac 32 + b(1- \gamma)} \| \Phi^w \|_{\mathcal{W}^{\rho, \gamma}_{T}} |t - r|^{\gamma} 
\end{align} 

\noi 
for any $0 \leq r < t \leq \lambda^{b} T$.  In particular, 
\begin{align}\label{MIF11}
 \|Y_{t, r}^{j,\lambda}\|_{\mathcal{X}_2^{0, \gamma} ([0, \lambda^b T] \times \T_{\lambda})}\les \lambda^{-\frac 32 + b(1- \gamma)} \| \Phi^w \|_{\mathcal{W}^{\rho, \gamma}_{T}}.
\end{align}
\end{lemma} 
\begin{proof}
We follow the argument in the proof of \cite[Lemma 7.4]{CGLLO}, which is in turn a modification of the proof of Lemma \ref{lem_nonlinsmooth}. Furthermore,  in view of Lemma \ref{LEM:OBS1}, it suffices to prove \eqref{MIF1}. 
By the interpolation lemma (see \cite [Lemma 12.1]{CKSTT04}), it is enough to prove \eqref{MIF1} for $N=1$. 
As in the proof of Lemma \ref{lem_nonlinsmooth}, we start by assuming  $\ft f_1(0)=\ft f_2(0)= 0$. 
By \eqref{driversF},  \eqref{Ix2}, we have 
\begin{align} 
\begin{aligned}
&\|  Y^{j,\ld}_{t,r}( f _1, f _2)  \|^2_{L^2(\T_\ld)} \\
&\quad=
\ld^{-3} \sum_{ n\in\Z_\ld^*}
|n|^2 \bigg|  \sum_{ \substack{n_1, n_2 \in \Z_\ld^*\\n = n_1+n_2}}
 M (\bar n)
\Phi^{w^\ld}_{t,r}(\Xi^{(j)}(\bar n) ) 
\ft f _1(n_1)
\ft  f _2 (n_2)
\bigg|^2 \\
&\quad=
\ld^{-1} \sum_{ n\in\Z_\ld^*}
|n|^2  \bigg( \sum_{ \substack{n_1, n_2 \in \Z_\ld^*\\n = n_1+n_2}}\ld^{-1}
\big| M (\bar n)\big| \big|
\Phi^{w^\ld}_{t,r}(\Xi^{(j)}(\bar n) ) \big|
\big|\ft f _1(n_1)\big|
\big|\ft  f _2 (n_2)\big|\bigg)
^2, \\
%\les\,\, & \ld^{-1}
%\bigg(\sup_{n_1 \in \Z_\ld^*}
%\sum_{\substack{n, n_2 \in \Z_\ld^*\\n_1 = n - n_2}} 
%\ind_{\{|n_1| \ge |n_2|\}} |n|^2 | M (\bar n)|^2
%|\Phi^{w^\ld}_{t,r} (\Xi^{(1)}(\bar n)) |^2 \bigg)
%\|f_1\|_{L^2(\T_\ld)}^2
%\|f_2\|_{L^2(\T_\ld)}^2, 
\end{aligned} 
\label{IF2}
\end{align}

\noi
where $ M (\bar n) =  M(n,n_1, n_2)$ is defined by 
\begin{align}
M (\bar n) =  M(n,n_1, n_2)=
\frac{m_{s, 1}(n) }
{m_{s, 1}(n_1) m_{s, 1}(n_2)}, 
\label{IF3}
\end{align}
and $m_{s,1}$ is defined in \eqref{Iop1}.
Note that
\begin{align}
\label{bound_M}
 M (\bar n) \les \ld^{-s} \frac{|n|^{s}}{|n_1|^{s}|n_2|^{s}}
\end{align}

\noi
for any $n, n_1, n_2 \in \Z_\ld^*$
with $n = n_1 + n_2$. 
%Moreover, from \eqref{rho2}, 
%\eqref{scaling2}, 
%\eqref{rho1}, and~\eqref{K3a}, 
%we have
%\begin{align}
%\begin{aligned}
%\Phi^{w^\ld}_{t,r} (\Xi_\KDV(\bar n)) 
%& = \bigg| 
%\int_{r}^{t}  e^{i  \Xi_\KDV (\bar n) w^{\ld}(t') } d t'
%\bigg|
%= \ld^b   \bigg|  \int_{\ld^{-b}r}^{\ld^{-b}t} 
%e^{i  \ld^3  \Xi_\KDV(\bar n)  w(t') } dt'  \bigg|  \\
%& \les 
%\ld^{b(1-\gamma)-3\rho} 
% \|\Phi^{w}  \|_{\W_{T}^{\rho, \g}}
% |t-r|^{\g} 
% |n n_1 n_2|^{-\rho}
%\end{aligned}
%\label{D3}
%\end{align} 
Thus, by combining \eqref{IF2}, \eqref{IF3}, \eqref{bound_M}, Lemma \ref{lem_Phi}, Lemma \ref{lem_bi} and arguing as in the proof of Lemma \ref{lem_nonlinsmooth} with $s=s_0$ (see also Remark \ref{b11}) we obtain
\begin{equation}\label{Y22}
\begin{split}
\| Y^{j,\ld}_{t,r}( f _1, f _2)  \|^2_{L^2(\T_\ld)} &\les  \ld ^{-3+2b(1-\gamma)} A^{(j,k)}\| \Phi^w \|^2_{\mathcal{W}_{T}^{\rho, \gamma}} |t - r|^{2\gamma} \|f_1 \|^2_{L^2(\T_\ld)} \|f_1 \|^2_{L^2(\T_\ld)},
\end{split}
\end{equation}
where $A^{(j,k)}$ are defined as in \eqref{A19}, \eqref{A219} and \eqref{A319}. Thus, by combining  \eqref{Asmall}
with Lemma \ref{lem_nonlinsmooth} we conclude the proof provided \eqref{condX1} holds. 
Let us move to the general case, namely by only assuming $\ft{f}_1(0)=0$ if $j=2$. Then,  by arguing as in \eqref{IF2}, we have
\begin{align*}
&\|  Y^{j,\ld}_{t,r}( f _1, f _2)  \|^2_{L^2(\T_\ld)} \\
&\quad \les 
\ind_{\{j=1\}}\cdot \ld^{-1} \sum_{ n\in\Z_\ld^*}
|n|^2  \bigg( \sum_{ \substack{n_1\in \Z_\ld, n_2 \in \Z_\ld^*\\n = n_1+n_2}} \ind_{\{n_1=0\}}\cdot \ld^{-1}
\big| M (\bar n)\big| \big|
\Phi^{w^\ld}_{t,r}(\Xi^{(j)}(\bar n) ) \big|
\big|\ft f _1(n_1)\big|
\big|\ft  f _2 (n_2)\big|\bigg)
^2\\
&\phantom{XX}+\ld^{-1} \sum_{ n\in\Z_\ld^*}
|n|^2  \bigg( \sum_{ \substack{n_1 \in \Z_\ld^*, n_2\in \Z_\ld\\n = n_1+n_2}}\ind_{\{n_2=0\}}\cdot\ld^{-1}
\big| M (\bar n)\big| \big|
\Phi^{w^\ld}_{t,r}(\Xi^{(j)}(\bar n) ) \big|
\big|\ft f _1(n_1)\big|
\big|\ft  f _2 (n_2)\big|\bigg)
^2\\
&\phantom{XX} +\ld^{-1} \sum_{ n\in\Z_\ld^*}
|n|^2  \bigg( \sum_{ \substack{n_1, n_2 \in \Z_\ld^*\\n = n_1+n_2}}\ld^{-1}
\big| M (\bar n)\big| \big|
\Phi^{w^\ld}_{t,r}(\Xi^{(j)}(\bar n) ) \big|
\big|\ft f _1(n_1)\big|
\big|\ft  f _2 (n_2)\big|\bigg)
^2\\
&\quad=:S^{(j)}_1+S^{(j)}_2+S^{(j)}_3.
\end{align*}
The term $S^{(j)}_3$ is controlled as in \eqref{Y22} under the assumption \eqref{condX1}. It remains to consider the terms $S^{(j)}_1,\, S^{(j)}_2$.
As far as concern $S^{(j)}_2$, by Lemma \ref{lem_Phi}, noting that $M(\bar{n})=1$ if $n_2=0$, arguing as in the proof of Lemma \ref{lem_nonlinsmooth}, and \eqref{res2intro},
\begin{align*}
& S^{(j)}_2\les \ld^{-3+2b(1-\g)}\| \Phi^w \|^2_{\mathcal{W}^{\rho, \gamma}_{T}} |t - r|^{2\gamma} \| f_1 \|^2_{L^2 (\T_\ld)} \| f_2 \|^2_{L^2(\T_\ld)}\\
&\phantom{XXXX}\times \bigg(\sup_{m \in \Z^{\ast}} \sum_{\substack{m_1\in \Z^*, m_2 \in\Z\\m = m_1+m_2}}\ind_{\{m_2=0\}}\cdot  \frac {|m|^{2}}{|\Xi^{(j)} (\bar{m})|^{2\rho}}\bigg)\\
%& \phantom{XX}\sim \bigg(\sup_{m \in \Z^{\ast}}|m|^{2-6\rho}\bigg) \ld^{-3+2b(1-\g)}\| \Phi^w \|^2_{\mathcal{W}^{\rho, \gamma}_{T}} |t - r|^{2\gamma} \| f_1 \|^2_{L^2 (\T_\ld)} \| f_2 \|^2_{L^2(\T_\ld)} \\
& \phantom{XX} \les \ld^{-3+2b(1-\g)} \| \Phi^w \|^2_{\mathcal{W}^{\rho, \gamma}_{T}} |t - r|^{2\gamma} \| f_1 \|^2_{L^2 (\T_\ld)} \| f_2 \|^2_{L^2(\T_\ld)}
\end{align*}
since $\rho\ge \frac{1}{2}$.  If $j=1$ the same estimate holds for $S^{(j)}_1$, while if  $j=2$ then $S^{(j)}_1=0$ by definition so we don't get additional conditions. This concludes the proof of \eqref{MIF1}. Furthermore, by an adaptation of  Lemma \ref{LEM:OBS1} to the scaled setting (cf., Remark \ref{samedriver}) and Lemma \ref{LEM:tri2}, we have that ${Y}^{j,\ld}\in \cX^{0,\g}_2([0, \ld^b T]\times \T_\ld)$ satisfying \eqref{MIF11}. 
\end{proof}

Similarly to the previous section, for the convenience of the reader, we denote $\vec{Y}^{\ld}:=~(Y^{1,\ld},Y^{2,\ld})$. As a corollary to Lemma \ref{LEM:tri2}, we get
\begin{corollary}
\label{COR:I2}
Let $\ld \ge 1$ and $b \in \R$.
 Given $\alpha\in (0,4)\setminus\{1\}$, $0<\eps<1-\nu_c$, $\rho$ satisfying \eqref{condrhog}, $\frac 12 < \gamma < 1$, $T > 0$ and $w$ $(\rho, \gamma)$-irregular on $[0, T]$.
Fix  $s$ satisfying \eqref{condsneg}.
Then, given any 
$(u_0,v_0) \in L^2(\T_\ld)\times L^2(\T_\ld)$ and  $t_0 \in [0, \ld^b T]$, 
there exist $C_0 > 0$, $\ta > 0$, both independent of $(u_0,v_0)$ and $t_0$, 
and a unique solution $(Iu^{\ld},Iv^{\ld})$
to 
the scaled modulated $I$-KdV system~\eqref{SMIKdV} 
with $(I u^\ld,I v^\ld)(t_0)= (u_0,v_0)$
on the time interval $[t_0, t_0 + \tau]\cap [0, \ld^b T]$,  
%belonging to the class\textup{:}
%\begin{align*}
% \cD_{w}^0\big( ([t_0, t_0 + \tau] \cap [0, \ld^b T]) \times \T_\ld\big), 
%\end{align*}
\noi
where the local existence time 
$\tau\in (0,1]$ satisfies
\begin{align}
\tau  \ge C_0 \big(\|\vec{Y}^{\ld}\|_{\cX^{0,\g}_2(\ld^b T)\times \cX^{0,\g}_2(\ld^b T)})^{-\ta} \big( 1 + \|(u_0,v_0) \|_{L^2(\T_\ld)\times L^2(\T_\ld) }\big)^{-\ta}.
\label{ME0a}
\end{align}

\noi
Moreover, given any $0 < \s < \gamma$  such that $\s + \gamma > 1$, 
\begin{align*}
 \| (Iu^\ld,Iv^{\ld}) \|_{\CC^\s ([t_0, t_0 + \tau] \cap [0, \ld^b T]; L^2(\T_\ld)\times L^2(\T_\ld))}\le C_\s
 \|(u_0,v_0)\|_{L^2(\T_\ld)\times L^2(\T_\ld)}.
%\label{ME0b}
\end{align*}
\end{corollary} 
\begin{proof}
The proof is analogous to the one of Proposition \ref{PROP:main}, hence we omit the details.
\end{proof}
\begin{remark}\rm \label{rem28} We observe that, by combining \eqref{ME0a} with  \eqref{MIF11}, given $\ld\gg 1$ (and $\|(u_0,v_0)\|_{L^2(\T_\ld)\times L^2(\T_\ld)}=O(1)$), Corollary \ref{COR:I2} guaranties the existence of a unique solution to the $I$-scaled system \eqref{scaleKdV} on $[0,\tau]\times \T_\ld$, where $\tau>0$ satisfies
\begin{align}\label{MIFE}
\tau \sim 1\wedge \ld^{\big(\frac{3}{2}-b(1-\g)\big)\theta}.
\end{align}
In particular, the main goal of this section is to exhibit a regularization phenomenon for global well-posedness analogous to the one of the local theory; see Proposition \ref{PROP:main}. Namely, given $s<0$ and $\frac{1}{2}<\g<1$, we want to show that there exists $\rho\gg 1$ such that the modulated KdV system \eqref{cmkdv} is globally well-posed in $\H^s(\T)$ provided the modulation $w$ is $(\rho,\g)$-irregular. Then, two situations occur:

\noi (i) If $b\le \frac{3}{2(1-\g)}$, from \eqref{MIFE} we deduce that $\tau\sim 1$, as in \cite{CGLLO}. However, in the application of the $I$-method we need to take $b$ sufficiently large in order to have small initial data; see \eqref{bs}. As a result, in this regime we obtain the restriction
\begin{align}
s>-\frac{3\g}{2(1-\g)},
\end{align}
no matter how large $\rho$ is. Hence, we do not porsue this direction.

\noi (ii) If $b> \frac{3}{2(1-\g)}$, we don not have any a priori threshold on $s$. On the other hand, from \eqref{MIFE}, the local existence $\tau$  satisfies
\begin{align}\label{tau}
\tau\sim \ld^{\big(\frac{3}{2}-b(1-\g)\big)\theta}\to 0\quad \text{as}\ \ld\to \infty,
\end{align}
requiring more iteration steps in the $I$-method procedure. We finally recall that, in analogy to Remark \ref{theta}, we can take $\theta=\frac{1}{\g-\s}$, for $1-\g<\s<\g$. 
%For fixed $\s > 0$ such that $1 - \gamma < \s < \gamma$, we run the contraction argument in $\mathcal{C}^{\s}([t_0,t_0+\tau];L^{2}(\T_\ld))$. In particular, as we have already observed in Remark \ref{theta}, we have that $\theta = \frac {1}{\gamma - \s}$. Moreover, 
%suppose that $v_0 \in L^2 (\T_{\lambda})$ and $t_0 \in [0, \lambda^b T]$. Then, 
%from \eqref{ME0a} and \eqref{MIF1}, 
%there exist $C_0 > 0$ independent of $v_0$ and $t_0$, and a unique solution $I u$ to the scaled modulated $I$-KdV equation \eqref{SMIKdV} with $I u^{\lambda} (t_0) = v_0$ on the time interval $[t_0, t_0 + \tau] \cap [0, \lambda^b T]$, where 
%the local existence time $\tau$ in Corollary \ref{COR:I2} satisfies 
%\begin{align} 
   % 1\ge \tau \ges \lambda^{\left(\frac 32 - b (1 - \gamma) %\right) \theta} \big(1 + \|(u_0,v_0) \|_{L^2 (\T_{\lambda})\times L^2(\T_\ld)} \big)^{-\theta},
%\end{align}  
%provided $b\ge \frac{3}{2(1-\gamma)}$. Thus, if $b$ is large enough we do not have a uniform lower bound for $\lambda^{\frac 32 - b (1 - \gamma)}$ with $\lambda \geq 1$ and the local existence time $\tau$ will depend on $\lambda$.
\end{remark}

\subsection{Commutator estimates}
From the local well-posedness
of 
the modulated KdV system \eqref{cmkdv} (Proposition \ref{PROP:main}), 
 the blowup alternative \eqref{BA1}
ensures that 
global well-posedness holds
if we show that the $\H^s(\T)$-norm 
of a solution $(u,v)$ remains finite 
for each finite time $t > 0$. Actually, in view of the conservation of the spatial means, it will be enough to control the $\dot \H^s(\T)$-(semi)norm only. Thus, by \eqref{scaling1}, \eqref{scaling3} and \eqref{I1}
\begin{align}\label{hom_imp}
\|(u,v)(t)\|^2_{\dot \H^s(\T)}&=\ld^{2b-3+s} \|(u^\ld,v^{\ld})(
\ld^bt)\|^2_{\dot \H^s(\T_\ld)}\\ &\les \ld^{2b-3} \|(Iu^\ld,Iv^{\ld})(\ld^b t)\|^2_{L^2(\T_\ld)\times L^2(\T_\ld)}.
\end{align}
Next, if we fix  a target time $T\gg 1$, from \eqref{hom_imp}  and
\begin{align}
\label{mEnergy}
\|(I u^{\lambda}, Iv^{\lambda})(t) \|^2_{L^{2}(\T)\times L^2(\T)} = \|(I\uu^{\ld}, I\vv^{\ld})(t)\|^2_{L^2(\T_{\ld})\times L^2(\T_{\ld})},
\end{align}
it suffices to show that 
%\begin{align} 
%\label{ME2}
%    \sup_{0 \leq t \leq T} \| u (t) \|^2_{H^{s}(\T)} \leq C (T) < \infty 
%\end{align} 
\begin{equation}
\begin{split}
\label{ME} 
\sup_{0 \leq t \leq \lambda^b T} \|(I\uu^{\ld}, I\vv^{\ld})(t)\|^2_{L^2(\T_{\ld})\times L^2(\T_{\ld})} \leq C_{s,b} (T, \lambda) < \infty.
\end{split}
\end{equation} 

In particular, by arguing as in \eqref{zero_term} 
%{\Bl{state how we use \eqref{zero_term} more explicitly here.}},
we have for any $t > r \geq 0$
\begin{equation}
\begin{split}
\label{IDiff1}
\|(I\uu^{\ld},I\vv^{\ld})(t)\|^2_{L^2(\T_{\ld})\times L^2(\T_{\ld})}&=\|(I\uu^{\ld},I\vv^{\ld})(r)\|^2_{L^2(\T_{\ld})\times L^2(\T_{\ld})}\\
&\quad+
 2 \big\langle I \mathbf{u}^{\lambda} (r), \com^{1}_{t, r} (\mathbf{v}^{\lambda} (r), \mathbf{v}^{\lambda} (r)) \big\rangle_{L^2 (\T_{\lambda})}\\
 &\quad + 2 \big\langle I \mathbf{v}^{\lambda} (r), \com^{2}_{t, r} (\mathbf{u}^{\lambda} (r), \mathbf{v}^{\lambda} (r)) \big\rangle_{L^2 (\T_{\lambda})}\\
& \quad +R_{t, r}^{1,\lambda}+R_{t, r}^{2,\lambda}, 
\end{split}
\end{equation}
where the commutators $\com^{j}_{t, r}$ and the remainders $R_{t, r}^{j,\lambda}$ are given by 
\begin{align}
\label{IDiff2}   
\begin{split}
\com^{j}_{t, r} (f_1, f_2) & = I X^{j,\lambda}_{t, r} (f_1, f_2) - X^{j,\lambda}_{t, r} (I f_1, I f_2),\quad j=1,2, \\ 
R^{1,\lambda}_{t, r} & = \| I \mathbf{u}^{\lambda} (t) - I \mathbf{u}^{\lambda} (r) \|^2_{L^2 (\T_{\lambda})} \\ 
& \quad + 2 \big\langle I \mathbf{u}^{\lambda} (r), I \mathbf{u}^{\lambda} (t) - I \mathbf{u}^{\lambda} (r) - I X^{1,\lambda}_{t, r} (\mathbf{v}^{\lambda} (r), \mathbf{v}^{\lambda} (r)) \big\rangle_{L^2 (\T_{\lambda})},\\
R^{2,\lambda}_{t, r} & =\| I \mathbf{v}^{\lambda} (t) - I \mathbf{v}^{\lambda} (r) \|^2_{L^2 (\T_{\lambda})} \\
& \quad  + 2 \big\langle I \mathbf{v}^{\lambda} (r), I \mathbf{v}^{\lambda} (t) - I \mathbf{v}^{\lambda} (r) - I X^{2,\lambda}_{t, r} (\mathbf{u}^{\lambda} (r), \mathbf{v}^{\lambda} (r)) \big\rangle_{L^2 (\T_{\lambda})}.
\end{split}
\end{align}
The following proposition establishes control on the third and fourth terms on the right-hand side of \eqref{IDiff1}.
%%%%%%%%%%%%%%%%%%%%%%%%%%%%%%%%%%%%%%%%%%%%%%%%%%%%%%%%%%%%%%%%%%%%%%%%%%%%%%%%%%%
%%%%%%%%%%%%%%%%%%%%%%%%% Proposition 4.7 %%%%%%%%%%%%%%%%%%%%%%%%%%%%%%%%%%%%%%%%%
%%%%%%%%%%%%%%%%%%%%%%%%%%%%%%%%%%%%%%%%%%%%%%%%%%%%%%%%%%%%%%%%%%%%%%%%%%%%%%%%%%%
\begin{proposition} 
	\label{PROP:com1}
	Let  $\lambda \geq 1$ and $b\in \R$. Given $\alpha\in (0,4)\setminus\{1\}$, $0<\eps<1-\nu_c$, {{$\rho $}} satisfying \eqref{condrhog}, $\frac 12 < \gamma < 1$ and $T > 0$, let $w$ be $(\rho, \gamma)$-irregular on $[0, T]$, and let $X^{j,\lambda}$ be as in \eqref{Xld1}. If $s$ satisfies \eqref{condsneg} then  
	\begin{align} 
		\label{com1}
		\| \com^{j}_{t, r} (f_1, f_2) \|_{L^2 (\T_{\lambda})} \les {K}_b( \lambda, N) \| \Phi^w \|_{\mathcal{W}^{\rho, \gamma}_{T}} |t - r|^{\gamma} \| I f_1 \|_{L^2 (\T_{\lambda})} \| I f_2 \|_{L^2 (\T_{\lambda})}
	\end{align}  
	
	\noi 
	for any $0 \leq r <t \leq \lambda^{b} T$, $\lambda \geq 1$, and $N \in \N$, where ${K}_b(\lambda, N)$ is given by
	\begin{align}
    \label{Ktilde29}
	{K}_b(\lambda, N):=\ld^{-\frac{1}{2}+b(1-\gamma)-\rho(1-\nu_{c}-\eps)}N^{1-\rho(1-\nu_{c}-\eps)}.
\end{align}
\end{proposition}

\begin{proof} By \eqref{IDiff2} \eqref{driversF}, \eqref{Ix2} and \eqref{FT3}, we have
\begin{align}
		\|  & I X^{j,\ld}_{t,r}( f _1, f _2) - X^{j,\ld}_{t,r} (I  f_1, I  f_2) \|^2_{L^2(\T_\ld)}\notag\\
		&\les
		\ld^{-3} \sum_{ n\in\Z_{\ld}^*}
		\bigg( \sum_{\substack{n_1, n_2 \in\Z_{\ld}\\n = n_1+n_2}}
		\ind_{|n_1|\ge |n_2|}|n||M_N(\bar n)|
		|\Phi^{w^\ld}_{t,r}(\Xi^{(j)}(\bar n))|
		|\ft{If}_1(n_1)| |\ft{If}_2 (n_2)|
		\bigg)^2\label{eq_sum_1}\\
		&\phantom{X}+
		\ld^{-3} \sum_{ n\in\Z_{\ld}^*}
		\bigg( \sum_{\substack{n_1, n_2 \in\Z_{\ld}\\n = n_1+n_2}}
		\ind_{|n_1|\le |n_2|}|n||M_N(\bar n)|
		|\Phi^{w^\ld}_{t,r}(\Xi^{(j)}(\bar n))|
		|\ft{If}_1(n_1)| |\ft{If}_2 (n_2)|
		\bigg)^2, \label{eq_sum_2}
\end{align}
\noi 
where  $M_{N} (\bar {n})$ is defined by setting 
\begin{align} 
	\label{com4}
	M_{N} (\bar{n}) = M_{N} (n, n_1, n_2) = \frac{m_{s, N} (n)}{m_{s, N} (n_1) m_{s, N} (n_2)} - 1. 
\end{align} 
From \eqref{Iop1} and \eqref{com4},  we see that $M_N(\bar{n})=0$ if $n_1=0$ or $n_2=0$. Hence, it is enough to bound the terms \eqref{eq_sum_1}-\eqref{eq_sum_2} when the summations are over the set 
\begin{align}\label{D_0}
D:=\{(n,n_1,n_2)\in(\Z^{*}_\ld)^3\,:\, n=n_1+n_2\}.
\end{align}
Before proceeding, we also note the following. If $j=1$, 
due to the symmetry of $\Xi^{(1)} (\bar {n})$ in $n_1$ and $n_2$, it is clear that any bound for \eqref{eq_sum_2} follows from the one for \eqref{eq_sum_1}. 
If $j=2$, although $\Xi^{(2)} (\bar {n})$ is not symmetric in $n_1$ and $n_2$,  Corollary \ref{cor_res} provides symmetric lower bounds for  $|\Xi^{(2)} (\bar {n})|$. In particular, since such lower bounds are the same for $j=1,2$, in order to obtain \eqref{com1} we can reduce to the bound \eqref{eq_sum_1} with internal summation on the set in \eqref{D_0}, for $j=1$.

For this purpose, we divide $\sum=\{(n, n_1, n_2) \in (\Z_{\lambda}^{\ast})^3: n=n_1+n_2, |n_1| \ge |n_2| \} = \cup_{j=1}^5 D_j$, where
\begin{align}
	\label{com5} 
	\begin{split} 
		D_{1} &:= \Big\{ (n, n_1, n_2) \in \sum : \frac {N}{2}> |n_1| \ge |n_2| \Big\}, \\ 
		D_{2}  &:= \Big\{ (n, n_1, n_2) \in \sum : |n_1| \geq \frac {N}{2} >  |n_2|,\, |n| < \frac {N}{4} \Big\}, \\ 
		D_{3} &:= \Big\{ (n, n_1, n_2) \in \sum : |n_1| \geq \frac {N}{2} > |n_2|,\, |n| \geq \frac {N}{4} \Big\}, \\ 
		D_{4} &:= \Big\{ (n, n_1, n_2) \in \sum : |n_1| \ge |n_2| \geq \frac {N}{2},\, |n| \ll N \Big\}, \\
		D_{5} &:= \Big\{ (n, n_1, n_2) \in \sum : |n_1| \ge |n_2| \geq \frac {N}{2},\, |n| \ges N \Big\}.
	\end{split}
\end{align} 

\noi 
First, we provide some bounds on $M_{N} (\bar{n})$ defined in \eqref{com4} on each set defined in \eqref{com5}.
\underline{On $D_{1}$ and $D_{2}$:} Since $m_{s, N} (n) = m_{s, N} (n_1) = m_{s, N} (n_2) = 1$ we conclude that 
\begin{align}\label{0_29}
M_{N} (\bar{n}) = 0.
\end{align}
 \underline{On $D_3$:} Since $m_{s, N} (n) \sim |n|^sN^{-s}$, $m_{s, N} (n_1) \sim |n_1|^sN^{-s}$ and $m_{s, N} (n_2) = 1$, we have
\begin{align}\label{1_29}
|M_{N} (\bar{n})| \les |n|^s|n_1|^{-s}+1 \les |n|^s|n_1|^{-s}.
\end{align}
\underline{On $D_4$:} Since $m_{s, N}(n)=1$, $m_{s, N} (n_1) \sim |n_1|^sN^{-s}$ and $m_{s, N} (n_2) \sim |n_2|^sN^{-s}$ we infer 
\begin{equation}\label{2_29}
\begin{split}
|M_{N} (\bar{n})| &\les (m_{s, N}(n_1)m_{s, N} (n_2))^{-1}+1\\
& \les |n_1|^{-s}|n_2|^{-s}N^{2s}\\
&  \les |n_1|^{-2s}N^{2s}.
\end{split}
\end{equation}
\underline{On $D_5$:} Since $m_{s, N} (n) \sim |n|^sN^{-s}$, $m_{s, N} (n_1) \sim |n_1|^sN^{-s}$ and $m_{s, N} (n_2) = |n_2|^sN^{-s}$ we get 
\begin{align}\label{3_29}
|M_{N} (\bar{n})| \les |n|^{s}|n_1|^{-s}|n_2|^{-s}N^{s}.
\end{align}

Next, by \eqref{0_29}, \eqref{1_29}, \eqref{2_29}, \eqref{3_29}, and Lemma \ref{lem_bi},  we infer 
\begin{align}\label{eq30}
\eqref{eq_sum_1}\les   \lambda^{-1+2b(1-\gamma)-6\rho} \Big(\sum_{k=3}^{5}A_{\ld, N}^{(k)}\Big)\|\Phi^w_{t,r}\|^2_{\W^{\rho,\g}_T} |t-r|^{2\g} \|{{I}} f_1\|^2_{L^2(\T_\lambda)}\|{{I}} f_2\|^2_{L^2(\T_\lambda)}, 
\end{align} 
where the quantities $A_{\ld, N}^{(k)}$, $k = 3,\,4,\,5$ are defined as 
\begin{align}
	A_{\ld, N}^{(3)} &:= {{ \sup_{n_1 \in \Z^{\ast}_{\ld}} \sum_{\substack{n, n_2 \in\Z_{\ld}^*\\n = n_1+n_2}} }} \ind_{D_{3} } \frac {|n|^{2s+2}|n_1|^{-2s}}{|\Xi^{(1)} (\bar{n})|^{2\rho}}, \label{com6}\\ 
	A_{\ld, N}^{(4)} &:= N^{4s}\sup_{n_1 \in \Z^{\ast}_{\ld}} \sum_{\substack{n, n_2 \in\Z_{\ld}^*\\n = n_1+n_2}} \ind_{D_{4} } \frac { {{ |n|^{2}|n_1|^{-4s} }} }{|\Xi^{(1)} (\bar{n})|^{2\rho}},  \label{com7}\\
	A_{\ld, N}^{(5)} &:= N^{2s}{{ \sup_{n \in \Z^{\ast}_{\ld}} \sum_{\substack{n_1, n_2 \in\Z_{\ld}^*\\n = n_1+n_2}} }} \ind_{D_{5} } \frac {|n|^{2s+2}|n_1|^{-2s}|n_2|^{-2s} }{|\Xi^{(1)} (\bar{n})|^{2\rho}},  \label{com8}
\end{align} 
provided
\begin{equation}\label{eq_result}
A_{\ld, N}^{(k)}\les \ld^{2\rho(2+\nu_{c}+\eps)} N^{2-2\rho(1-\nu_c-\eps)},\quad \text{for}\ k=3,4,5.
\end{equation}

%%%%%%%%%%%%%%%%%%%%%%%%%  Case 1 %%%%%%%%%%%%%%%%%%%%%%%%%%%%%%%%%%%
\noi $\bul$ \textbf{Case 1:}  on $D_3$. 
Let us split our analysis into three cases, corresponding to those of Corollary \ref{cor_res}. In particular, under the condition $|n_1|\ge |n_2|$,
we define the following sets:
\begin{align}
B_1&=\left\{(n,n_1,n_2)\in (\Z^*_\ld)^3: |n| \ll |n_1| \right\},\label{L1}\\  B_2&=\left\{(n,n_1,n_2)\in (\Z^*_\ld)^3: |n|\ges |n_1|, \min\left\{|n_1 - cn|, |n_1 - c_*n|\right\} \ges \ld^{-1} \right\},\label{L2}\\
B_3&=\left\{(n,n_1,n_2)\in (\Z^*_\ld)^3: |n|\ges  |n_1|, \min\left\{|n_1 - cn|, |n_1 - c_*n|\right\} \ll  \ld^{-1}\right\}\label{L3}. 
\end{align}
{{
Note that $D_3\cap B_1$ is empty because $|n|+|n_2|\ge |n-n_2|=|n_1| \gg |n| \ge \frac{|n|}{3}+\frac{N}{6}  > \frac{|n|+|n_2|}{3} $ is a contradiction.
On $D_{3} \cap B_2$, by (i) of Corollary \ref{cor_res} and Proposition \ref{sum28}, we have
\begin{align*}
\sup_{n_1 \in \Z^{\ast}_{\ld}} \sum_{\substack{n, n_2 \in\Z_{\ld}^*\\n = n_1+n_2}} \ind_{D_{3} \cap B_2} \frac {|n|^{2s+2}|n_1|^{-2s}}{|\Xi^{(1)} (\bar{n})|^{2\rho}} &\les \ld^{2\rho} \sup_{n_1 \in \Z^{\ast}_{\ld}} \sum_{\substack{n, n_2 \in\Z_{\ld}^*\\n = n_1+n_2}} \ind_{D_{3} \cap B_2} {|n|^{-4\rho+2}}\notag\\
&\les  \ld ^{2\rho} N^{-4\rho+2}  \sum_{|n_2|<\frac{N}{2}} 1\\
&\les  \ld ^{2\rho+1} N^{-4\rho+3}.
\end{align*}
On $D_{3} \cap B_3$, by (ii) of Corollary \ref{cor_res} and $0>s \ge 1-\rho(1-\nu_c-\eps)$, we infer 
\begin{align*}
\sup_{n_1 \in \Z^{\ast}_{\ld}} \sum_{\substack{n, n_2 \in\Z_{\ld}^*\\n = n_1+n_2}} \ind_{D_{3}\cap B_3} \frac {|n|^{2s+2}|n_1|^{-2s} } {|\Xi^{(1)} (\bar{n})|^{2\rho}} &\les \ld^{2\rho(2+\nu_{c}+\eps)} \sup_{n_1 \in \Z^{\ast}_{\ld}} \sum_{\substack{n, n_2 \in\Z_{\ld}^*\\n = n_1+n_2}} \ind_{D_{3} \cap B_3}|n|^{2-2\rho(1-\nu_c-\eps)}\\
&\les \ld^{2\rho(2+\nu_{c}+\eps)} N^{2-2\rho(1-\nu_c-\eps)}
\end{align*}
where we used $\#B_3\les 1$ for any $n_1 \in \Z^{*}_\ld$.
Since $\rho > \frac{1}{1-\nu_c-\eps}> 1$ and $\nu_c \ge 0$, we have $$\ld ^{2\rho+1} N^{-4\rho+3} \le \ld^{2\rho(2+\nu_{c}+\eps)} N^{2-2\rho(1-\nu_c-\eps)}.$$
Therefore, we obtain \eqref{eq_result} for $k=3$.

%%%%%%%%%%%%%%%%%%%%%%%% Case 2 %%%%%%%%%%%%%%%%%%%%%%%%%%%%%%%%
\noi $\bullet$ \textbf{Case 2: }on $D_4$.
Obviously, $D_4 \cap B_2$ and $D_4 \cap B_3$ are empty. 
Thus, we only need to consider the summation in $D_4 \cap B_1$.
By (iii) of Corollary \ref{cor_res}, $4\rho+4s\ge 0$ and Proposition \ref{sum28}, we have (for $\rho$ satisfying \eqref{condrhog})
\begin{align*}
A^{(4)}_{\ld,N} &\les  N^{4s}\sup_{n_1\in \Z^{*}_\ld}\sum_{\substack{n,n_2\in \Z^*_\ld\\ n=n_1+n_2}}\ind_{D_{4}\cap B_1}\frac{|n|^{2-2\rho}}{|n_1|^{4\rho+4s}}\\
& \les  N^{-4\rho} \sum_{\substack{n\in \Z^*_\ld\\ |n|\ll N}} |n|^{2-2\rho}\\
&\les \begin{cases}
	\lambda^{}N^{3-6\rho}, & \textup{if } \rho < \frac 32, \\ 
	\lambda^{} N^{-4\rho} (\log N + \log \lambda), & \textup{if } \rho = \frac 32,  \\ 
	\lambda^{2\rho-2} N^{-4 \rho}, & \textup{if } \rho > \frac 32,
\end{cases}
\end{align*}
which implies \eqref{eq_result} for $k=4$.

%%%%%%%%%%%%%%%%%%%%%%%  Case 3 %%%%%%%%%%%%%%%%%%%%%%%%%%%%%%%%%
\noi 
 $\bullet$ \textbf{Case 3: }on $D_5$.

In this case, we need to separately consider the summations on $D_5\cap B_i$, $i=1,2,3$.

On $D_5\cap B_1$, $|n_1|\sim |n_2| \gg |n| \ges N$ holds. Thus, by (iii) of Corollary \ref{cor_res} and Proposition \ref{sum28},
\begin{equation}\label{D51}
\begin{split}
  N^{2s}&\sup_{n \in \Z^{\ast}_{\ld}} \sum_{\substack{n_1, n_2 \in\Z_{\ld}^*\\n = n_1+n_2}} \ind_{D_{5} \cap B_1} \frac {|n|^{2s+2}|n_1|^{-2s}|n_2|^{-2s}}{|\Xi^{(1)} (\bar{n})|^{2\rho}}\\
  & \les    N^{2s}\sup_{|n| \ges N} |n|^{2s-2\rho+2}\sum_{n_1 \in\Z_{\ld}^*} |n_1|^{-4s-4\rho}\\
  & \les \lambda^{-4s-4\rho} N^{4s-2\rho+2},
\end{split}
\end{equation}
since $s>1-\rho$ and $2s-2\rho+2<4s<0$.

On $D_5\cap B_2$, $|n|\sim |n_1| \ge |n_2| \ges N$ holds. Thus, by (i) of Corollary \ref{cor_res} and Proposition \ref{sum28},
\begin{equation}\label{D52}
\begin{split}
N^{2s}&\sup_{n \in \Z^{\ast}_{\ld}} \sum_{\substack{n_1, n_2 \in\Z_{\ld}^*\\n = n_1+n_2}} \ind_{D_{5} \cap B_2} \frac {|n|^{2s+2}|n_1|^{-2s}|n_2|^{-2s}}{|\Xi^{(1)} (\bar{n})|^{2\rho}}\\
& \les  \ld^{2\rho}N^{2s}\sup_{\substack{n \in \Z^{*}_\ld\\ |n|\ges N}}|n|^{2-4\rho}  \sum_{\substack{n_2\in \Z^{*}_\ld\\ |n_2|\le |n|}} |n_2|^{-2s}\\
& \les  \ld^{2\rho+1}N^{2s}\sup_{\substack{n\in \Z^{*}_\ld\\ |n|\ges N}}|n|^{3-4\rho-2s} \\
&\les \ld^{2\rho+1}  N^{3-4\rho},
\end{split}
\end{equation}
since $3-4\rho-2s <1-2\rho <0$.

It remains to consider the case on $D_{5}\cap B_3$. Here, $|n|\sim |n_1| \ge |n_2| \ges N$ holds.
Thus, by (ii) of Corollary \ref{cor_res} 
\begin{equation}\label{D53}
\begin{split}
 N^{2s}&\sup_{n \in \Z^{\ast}_{\ld}} \sum_{\substack{n_1, n_2 \in\Z_{\ld}^*\\n = n_1+n_2}} \ind_{D_{5} \cap B_3} \frac {|n|^{2s+2}|n_1|^{-2s}|n_2|^{-2s}}{|\Xi^{(1)} (\bar{n})|^{2\rho}}\\
 & \les\ld^{2\rho(2+\nu_{c}+\eps)}N^{2s}\sup_{n \in \Z^{\ast}_{\ld}} \sum_{\substack{n_1, n_2 \in\Z_{\ld}^*\\n = n_1+n_2}} \ind_{D_{5} \cap B_3} |n_1|^{2-2\rho(1-\nu_c-\eps)-2s}\\
 & \les \ld^{2\rho(2+\nu_{c_1}+\eps)}N^{2-2\rho(1-\nu_{c}-\eps)},
 \end{split}
\end{equation}
since $s\ge 1-\rho(1-\nu_{c_1}-\eps)$ and $\# B_3 \les 1$ for any $n \in \Z^{\ast}_{\ld}$.
By combining \eqref{D51}, \eqref{D52} with \eqref{D53} and noting that the worst bound comes from \eqref{D53} we obtain \eqref{eq_result} for $k=5$.
}}
Putting together the above estimates with \eqref{eq30} yields \eqref{com1}.
\end{proof}

\remark \rm
Usually, in the $I$-method, a key argument in the proof of the commutator estimates is to exploit the mean value theorem in the high–low interaction in order to obtain favorable bounds. By using this argument, the estimate in $\eqref{1_29}$ can be improved to $$|M_N(\bar{n})| \les |n|^s |n_1|^{-s-1} |n_2|,$$
where $M_N(\bar {n})$ is defined as in \eqref{com4}. However, note that this argument is not needed in the proof of Proposition \ref{PROP:com1}.

%%%%%%%%%%%%%%%%%%%%%%%%%%%%%%%%%%%%%%%%%%%%%%%%%%%%%%%%%%%%%%%%%%%%%%%%%%%%%%%%%%%%%%%
%%%%%%%%%%%%%%%%%%%%%%%%%%%%% proof of theorem 1.6 %%%%%%%%%%%%%%%%%%%%%%%%%%%%%%%%%%%%
%%%%%%%%%%%%%%%%%%%%%%%%%%%%%%%%%%%%%%%%%%%%%%%%%%%%%%%%%%%%%%%%%%%%%%%%%%%%%%%%%%%%%%%

\subsection{Proof of Theorem \ref{THM:11} (iii)} 
%In this subsection, we aim to show the global well-posedness of the modulated KdV system \eqref{cmkdv} in $H^{s}$ with negative regularity parameter $s$. 
Fix $(u_0,v_0) \in H_0^s ({\T})\times H^s(\T)$ for some $s < 0$ satisfying \eqref{condX1} (to be determined later). It suffices to show \eqref{ME} for a suitable choice of $b=b(s,\g)$, $N=N(T)$, $\ld=\ld(T)$ sufficiently large.
%\begin{align} 
%\label{ME2}
%    \sup_{0 \leq t \leq T} \| u (t) \|^2_{H^{s}(\T)} \leq C (T) < \infty 
%\end{align} 

From \eqref{I1} and \eqref{scaling1}, we have 
\begin{align} 
\label{IT1}
    \| (I u_0^{\lambda},I v_0^{\lambda}) \|_{L^2 (\T_{\lambda})\times L^2(\T_\ld)} \les \lambda^{-b + \frac 32 - s} N^{-s} \| (u_0,v_0) \|_{H^s (\T)\times H^s(\T)}. 
\end{align} 
Then, given small $\eps_0 > 0$, we can choose 
\begin{align} 
\label{IT2}
\lambda \sim N^{- \frac {2s}{2b + 2s - 3}}, 
\end{align} 

\noi 
such that 
\begin{align} 
\label{IT3}
\| (I u_0^{\lambda}, I v_0^{\lambda}) \|_{L^2 (\T_{\lambda})\times L^2(\T_\ld)}^2 \leq \eps_0,
\end{align} 

\noi 
provided that 
\begin{align}\label{bs}
b>\tfrac{3}{2}-s.
\end{align} Here, the implicit constant in \eqref{IT2} depends only on $\| (u_0,v_0) \|_{H^s (\T)\times H^s(\T)}$.

As specified in Remark \ref{rem28}-(ii), we further assume $b> \frac{3}{2(1-\gamma)}$. Then, given $(u_0, v_0)$ with $$\|(u_0, v_0) \|^2_{L^2 (\T_{\lambda})\times L^2(\T_\ld)} = 2 \eps_0,$$ Corollary \ref{COR:I2}  guaranties the existence of a solution $(Iu^\ld,Iv^\ld)$ to the scaled modulated $I$-KdV \eqref{SMIKdV} until the time $\tau$ given by \eqref{tau}.  More precisely, if we have 
\begin{align} 
\label{Ia} 
    \| (I \mathbf{u}^{\lambda},I \mathbf{v}^{\lambda}) (t_0) \|^2_{L^2 (\T_{\lambda})} \leq 2 \eps_0
\end{align} 

\noi 
for some $0 \leq t_0 \leq \lambda^b T$, then Corollary \ref{COR:I2} guaranties that the solution $(I u^{\lambda},I v^{\lambda})$ to the scaled modulated $I$-KdV \eqref{SMIKdV} exists in the time interval $[t_0, t_0 + \tau] \cap [0, \lambda^b T]$, satisfying the bound: 
\begin{align}
\label{Ib} 
    \| (I \mathbf{u}^{\lambda}, I \mathbf{v}^{\lambda}) \|^2_{\mathcal{C}^{\s} ([t_0, t_0 + \tau] \cap [0, \lambda^b T]; L^2 (\T_{\lambda})\times L^2 (\T_{\lambda}))}
    %+\| I \mathbf{v}^{\lambda} \|^2_{\mathcal{C}^{\s} ([t_0, t_0 + \tau] \cap [0, \lambda^b T]; L^2 (\T_{\lambda}))}
    \leq C_\s\eps_0. 
\end{align}  

First, we make some preliminary estimates on the remainders $R^{\ld}_{t,r}=R^{1,\ld}_{t,r}+R^{2,\ld}_{t,r}$ defined in \eqref{IDiff1}. By re-writing \eqref{IDiff1} and arguing as in \cite[eq. (7.50)-(7.51)]{CGLLO} we obtain that
\begin{equation}
\label{R1}
\begin{split}
(\dl R^{\ld})_{t_1,t_2,t_3} &= {R^{\ld}_{t_1,t_3}-R^{\ld}_{t_1,t_2}-R^{\ld}_{t_2,t_3}}\\
&=2\langle I\uu^{\ld}(t_2)-I\uu^{\ld}(t_3),\textup{com}^{1}_{t_1,t_2}(\vv^{\ld}(t_2), \vv^{\ld}(t_2))\rangle_{L^2(\T_\ld)}\\
&\phantom{X}+2\langle I\uu^{\ld}(t_3),\textup{com}^{1}_{t_1,t_2}(\vv^{\ld}(t_2)-\vv^{\ld}(t_3), \vv^{\ld}(t_2))\rangle_{L^2(\T_\ld)}\\
&\phantom{X}+2\langle I\uu^{\ld}(t_3),\textup{com}^{1}_{t_1,t_2}(\vv^{\ld}(t_3), \vv^{\ld}(t_2)-\vv^{\ld}(t_3))\rangle_{L^2(\T_\ld)}\\
&\phantom{X}+2\langle I\vv^{\ld}(t_2)-I\vv^{\ld}(t_3),\textup{com}^{2}_{t_1,t_2}(\uu^{\ld}(t_2), \vv^{\ld}(t_2))\rangle_{L^2(\T_\ld)}\\
&\phantom{X}+2\langle I\vv^{\ld}(t_3),\textup{com}^{2}_{t_1,t_2}(\uu^{\ld}(t_2)-\uu^{\ld}(t_3), \vv^{\ld}(t_2))\rangle_{L^2(\T_\ld)}\\
&\phantom{X}+2\langle I\vv^{\ld}(t_3),\textup{com}^{2}_{t_1,t_2}(\uu^{\ld}(t_3), \vv^{\ld}(t_2)-\vv^{\ld}(t_3))\rangle_{L^2(\T_\ld)},
\end{split}
\end{equation}
for any $t_1>t_2>t_3\ge 0$. 

Next, let's fix $J\subset [0,\ld^bT]$ and $0<\s<\gamma$ such that $\gamma+\s>1$. Then, by \eqref{R1}, H\"older inequality and Proposition \ref{PROP:com1} we infer 
\begin{align*}
&|(\dl R^{\ld})_{t_1,t_2,t_3}|\\
&\phantom{XXX}\les {K}_b(\ld,N)\| \Phi^w \|_{\mathcal{W}^{\rho, \gamma}_{T}}|t_1-t_2|^{\gamma}|t_2 - t_3|^{\s}\Big(\| I \mathbf{v}^{\lambda} \|^2_{\mathcal{C}^{\s} (J; L^2 (\T_{\lambda}))}\| I \mathbf{u}^{\lambda} \|_{\mathcal{C}^{\s} (J; L^2 (\T_{\lambda}))}\\
&\phantom{XXXXXXXX}+\| I \mathbf{u}^{\lambda} \|^2_{\mathcal{C}^{\s} (J; L^2 (\T_{\lambda}))}\| I \mathbf{v}^{\lambda} \|_{\mathcal{C}^{\s} (J; L^2 (\T_{\lambda}))}\Big), 
\end{align*}
for any $t_1>t_2>t_3$ belonging to $J$.
Then, from, Lemma \ref{LEM:sew} we have 
\begin{align}
\label{R2}
&|R^{\ld}_{t_1,t_2}|\les {K}_b(\ld,N)\| \Phi^w \|_{\mathcal{W}^{\rho, \gamma}_{T}} |t_1-t_2|^{\s+\gamma}\Big(\| I \mathbf{v}^{\lambda} \|^2_{\mathcal{C}^{\s} (J; L^2 (\T_{\lambda}))}\| I \mathbf{u}^{\lambda} \|_{\mathcal{C}^{\s} (J; L^2 (\T_{\lambda}))}\notag\\
&\phantom{XXXXXXXX}+\| I \mathbf{u}^{\lambda} \|^2_{\mathcal{C}^{\s} (J; L^2 (\T_{\lambda}))}\| I \mathbf{v}^{\lambda} \|_{\mathcal{C}^{\s} (J; L^2 (\T_{\lambda}))}\Big),
\end{align}
for any $t_1>t_2$ belonging to $J$.

Suppose that by choosing suitable $\rho$, $\gamma$, $s$, $b$, and $N \gg 1$, we can make $K_b( \lambda,  N)$ in \eqref{Ktilde29} as small as we need. Then, by \eqref{IT3}, Proposition \ref{com1}, \eqref{R2}, \eqref{Ib},  \eqref{tau} and  \eqref{IDiff1} with $t_0 = 0$ and $t\in [0,\tau]$, we have  
\begin{align*}
&\| (I \mathbf{u}^{\lambda},I \mathbf{v}^{\lambda}) (t) \|^2_{L^2 (\T_{\lambda})\times L^2 (\T_{\lambda})}\\
&\phantom{XXX}\le \| (I u_0^{\lambda},I v_0^{\lambda}) \|_{L^2 (\T_{\lambda})\times L^2 (\T_{\lambda})}^2+|R^{\ld}_{t,0}|+2 {K}_b(\ld,N) \,t^{\gamma}\|I\vv_0^{\ld}\|^2_{L^{2}(\T_\ld)}\|I\uu_0^{\ld}\|_{L^{2}(\T_\ld)} \| \Phi^{w} \|_{\mathcal{W}_{T}^{\rho, \gamma}}\\
&\phantom{XXXXX}+2 {K}_b(\ld,N) \,t^{\gamma}\|I\vv_0^{\ld}\|^2_{L^{2}(\T_\ld)}\|I\uu_0^{\ld}\|_{L^{2}(\T_\ld)} \| \Phi^{w} \|_{\mathcal{W}_{T}^{\rho, \gamma}}\\ 
&\phantom{XXX}\le \eps_0+C {K}_b(\lambda, N)\, (t^{\s + \gamma}+t^{\g})\| \Phi^{w} \|_{\mathcal{W}_{T}^{\rho, \gamma}}\\
& \phantom{XXX} \le \eps_0 +C_1 \tau^\g  {K} _b(\lambda, N) \| \Phi^{w} \|_{\mathcal{W}_{T}^{\rho, \gamma}}.
\end{align*}
As a consequence, we infer
\begin{equation} 
\begin{split}
\label{Ic}
    \sup_{0 \leq t \leq \tau} \| (I \mathbf{u}&^{\lambda},I \mathbf{v}^{\lambda}) (t) \|^2_{L^2 (\T_{\lambda})\times L^2(\T_\ld)}\\
    &   \leq \eps_0 + C_1 \tau^\g{K}_b(\lambda, N) \| \Phi^{w} \|_{\mathcal{W}_{T}^{\rho, \gamma}}\\
    &  \leq 2 \eps_0
   \end{split}
\end{equation} 

\noi 
for some constant $C_1 > 0$. Hence, the solution $(I u^{\ld}, Iv^{\ld})$ exists on $[\tau, 2 \tau]$ satisfying the bounds \eqref{Ia} and \eqref{Ib} with $t_0 = \tau$, which allows us to iterate this process. After $j$ steps, this iterative process shows that the solution $I u^{\lambda}$ exists on $[0, j \tau]$, satisfying \eqref{Ia} and \eqref{Ib} with $t_0 = (j-1) \tau$, and moreover, we have 
\begin{equation}
\begin{split}
\label{Id} 
    \sup_{0 \leq t \leq j \tau} \| (I \mathbf{u}&^{\lambda},I \mathbf{v}^{\lambda}) (t) \|^2_{L^2 (\T_{\lambda})\times L^2(\T_\ld)}\\
     &\leq  \eps_0 + j C_1 \tau^\g{K}_b(\lambda, N) \| \Phi^w \|_{\mathcal{W}^{\rho, \gamma}_{T}}\\
     & \le 2\eps_0,
     \end{split}
\end{equation} 
where the last inequality holds for $j=1,\dots, [\frac{\ld^bT}{\tau}]+1$, if 
\begin{align}\label{cond35}
j\tau^\g{K}_b(\lambda, N)\les \frac{\ld^bT}{\tau}\tau^{\g}K_b(\ld,N)\ll 1.
\end{align}
Now, let us consider  $\theta=\theta(\s)$ defined as
\begin{align}
\label{dr35}
\theta(\s)=\frac{1}{\gamma-\s},\quad 1-\gamma<\s<\g,
\end{align} 
and $b$ satisfying
\begin{align}\label{b_c}
b>b_*(s,\g):=\max\bigg(\frac{3}{2}-s, \frac{3}{2(1-\gamma)}\bigg).
\end{align}
In view of \eqref{tau}, the inequality \eqref{cond35} is equivalent to the following assertion:
\vspace{2mm}

\noi (A) There exist $N=N(T)\gg 1$ and $\ld=\ld(T)\gg 1$ satisfying \eqref{IT2} such that 
\begin{align} 
\label{CON1}
   \lambda^{b - (1-\g)\big( \frac 32 - b (1 - \gamma) \big) \theta}{K}_b( \lambda, N)=\eps_1,
\end{align}
for some small $\eps_1>0$.
%\noi 
%Then, proceeding as in \eqref{Ic} with \eqref{Id}, we obtain that 
%\begin{equation}
%\begin{split}
%\label{Ie} 
%&\sup_{0 \leq t \leq (j + 1) \tau} \Big(\| I \mathbf{u}^{\lambda} (t) \|^2_{L^2 (\T_{\lambda})}+ \| I \mathbf{v}^{\lambda} (t) \|^2_{L^2 (\T_{\lambda})}\Big)\\
%&\phantom{XXXXXXXX}\leq  \eps_0 + (j + 1) C_1 \lambda^{\left(\frac 32 - b (1 - \gamma) \right) \theta \gamma} {K}_b( \lambda, N) \| \Phi^w \|_{\mathcal{W}^{\rho, \gamma}_{T}} \\
%&\phantom{XXXXXXXX}\leq 2\eps_0, 
%\end{split}
%\end{equation}   

%For the convenience of the reader,  we call (A) such an assertion. 
Given $\rho\ge \frac{1}{2}$, $\frac{1}{2}<\g<1$, the task of the remaining part of this section is to find minimal restrictions on $s<0$ satisfying \eqref{condsintro} so that this assertion (A) holds, by making a suitable choice of $\theta$ and $b$ satisfying \eqref{dr35} and \eqref{b_c}.

Before stating our final lemma, for the convenience of the reader, we also set the following quantities:
\begin{align}
\rho_*&:=\frac{5-2\g}{2(1-\g)(1-\nu_c -\eps)}\label{rho_last};\\
\g_*&:=-\frac{3\gamma}{2(1-\gamma)}\label{gamma_last};\\
s_1&:=\g \left( 1-\rho(1-\nu_c-\eps)\right)\label{s1_last};\\
s_2&:= \frac{
2(2\gamma-1)\bigl(1 - \rho(1-\nu_c-\varepsilon)\bigr)
- 3(1-\gamma)^2
}{
2(-\gamma^2+3\gamma-1)
}.\label{s2_last}
\end{align}
\begin{lemma}\label{lem_region2}
Let $\frac{1}{2}<\g<1$, $0<\eps<1-\nu_c$, $\rho_*, s_1, s_2$  as in \eqref{rho_last}-\eqref{s1_last}-\eqref{s2_last}, $\rho$ satisfying \eqref{condrhog} and  $s<0$. Then, the following hold:

\begin{itemize}
\item[(i)]  Suppose that $\rho \le \rho_*$. Then, there exist $\theta$ and $b$ as in \eqref{dr35} and \eqref{b_c} such that
the assertion (A) holds if and only if
$s>s_1$.
\item[(ii)]  Suppose that $\rho>\rho_*$. Then, there exist $\theta$ and $b$ as in \eqref{dr35} and \eqref{b_c} such that
the assertion (A) holds if and only if
$s>%\tfrac{2-2\rho(1-\nu_c-\eps) +3(1-\g)\left(1-\g\left( \tfrac{1}{2\g-1}\right) \right)}{2 \left( 2-\g + \left( \tfrac{1}{2\g-1}\right)(1-\g)^2 \right)} 
s_2.$
\end{itemize}
In particular, under the above assumptions, $\rho$,$s$ satisfy \eqref{condX1} and, if $w$ is $(\rho,\g)$-irregular on $\R_+$ and $u_0$ has mean zero, then the system of Young differential equations \eqref{sYDE} is globally well-posed in $\H^s(\T).$
\end{lemma}
\begin{remark}\label{rem_sstar}
\rm  Note that, if $\frac{1}{2}<\g<1$, $0<\eps<1-\nu_c$, and $\rho_*, \g_*, s_1, s_2$ are as in \eqref{rho_last}-\eqref{gamma_last}-\eqref{s1_last}-\eqref{s2_last}, we have 
\begin{align}\label{eqqqstar}
\sgn ( \rho_*-\rho )=\sgn ( s_1-\g_* )=\sgn ( s_2-\g_* ).
\end{align}
Indeed, by a direct calculation one can see that 
\begin{align*}
s_1-\g_*
%&=\gamma\left(1-\rho(1-\nu_c-\varepsilon) +\frac{3}{2(1-\gamma)}\right)\\
%&=\gamma(1-\nu_c-\varepsilon) \left(\frac{5-2\g}{2(1-\g)(1-\nu_c -\eps)}-\rho\right)
=\gamma(1-\nu_c-\varepsilon)(\rho_*-\rho),
\end{align*}
and
\begin{align*}
s_2-\g_*
%=&\frac{2-2\rho(1-\nu_c-\varepsilon)+3(1-\gamma)\left(1-\gamma \left( \frac{1}{2\gamma-1}\right) \right)+\frac{3\gamma}{1-\gamma}\left(2-\gamma+\left(\frac{1}{2\gamma-1}\right)(1-\gamma)^2\right)}{2\left(2-\gamma+\left(\frac{1}{2\gamma-1}\right)(1-\gamma)^2\right)}\\
%= &\frac{\frac{5-2\gamma}{2(1-\gamma)}-\rho(1-\nu_c-\varepsilon)}{2-\gamma+\left(\frac{1}{2\gamma-1}\right)(1-\gamma)^2}
%=\frac{(1-\nu_c-\varepsilon)}{2-\gamma+\left(\frac{1}{2\gamma-1}\right)(1-\gamma)^2}(\rho_*-\rho)
=\frac{(2\g-1)(1-\nu_c-\varepsilon)}{-\g^2+3\g-1}(\rho_*-\rho).
\end{align*}
\end{remark}

\begin{proof}[Proof of Lemma \ref{lem_region2}]
By \eqref{IT2} and \eqref{Ktilde29} we see that \eqref{CON1} is equivalent to
\[
N^{\frac{1}{2b+2s-3}F_{\gamma,\rho,\nu_c}(b,\theta)} \les 1
\]
where $\theta$ is defined by \eqref{dr35}, while $F_{\gamma,\rho,\nu_c}(b,\theta)$ by 
\begin{align}
F_{\gamma,\rho,\nu_c}(b,\theta)=&-2s\left( b-(1-\gamma)\left( \frac{3}{2}-b(1-\gamma)\right)\theta-\frac{1}{2}+b(1-\gamma)-\rho(1-\nu_c-\varepsilon)  \right)\notag\\
&\quad +(1-\rho(1-\nu_c-\varepsilon))(2b+2s-3)\notag\\
=&C_0 \left(b-\frac{3}{2(1-\gamma)}\right)+C_1(-s+s_1)\label{eq_F1}\\
=&C_0 \left(b-\left(\frac{3}{2}-s\right) \right)+C_2(-s+\g_*)\left(\theta-\frac{1}{2\gamma-1}\right)+C_3(-s+s_2),\label{eq_F2}
\end{align}
with 
\begin{align*}
&C_0:=2(2-\gamma+(1-\gamma)^2\theta)\left( -s+\frac{1-\rho(1-\nu_c-\varepsilon)}{2-\gamma+(1-\gamma)^2\theta}\right),\\
&C_1:=\frac{3}{1-\gamma},\\
&C_2:=-2s(1-\gamma),\\
& C_3:=-2s\left(\frac{-\g^2+3\g-1}{2\g-1}\right).
\end{align*}
Note that $C_1,C_2,C_3$ are positive constants. 
Moreover, since $\theta>\frac{1}{2\gamma-1}$, and $2b+2s-3>0$, the assertion (A) holds if and only if
there exist $\theta >\frac{1}{2\gamma-1}$ and  $b>b_*(s,\g)$ such that $F_{\gamma,\rho,\nu_c}(b,\theta)<0$.

First, we assume that (i) holds. By Remark \ref{rem_sstar} and a direct computation we obtain that
\begin{align}\label{sbar}
\g_* \le s_1 \le \frac{(2\g-1)(1-\rho(1-\nu_c-\varepsilon))}{-\g^2+3\g-1}=: \bar{s}.
\end{align} 
Then we consider two subcases.
\smallskip

\noi {\bf Case (i.a)}: $s > \bar{s}$.
By taking $\theta>\frac{1}{2\gamma-1}$ sufficiently close to $\frac{1}{2\gamma-1}$
then $C_0<0$.
Moreover, in view of \eqref{eq_F1} and positivity of $C_1$, for any $b>b_*(s,\g)$ we conclude that $F_{\gamma,\rho,\nu_c}(b,\theta)<0$.
\smallskip

\noi {\bf Case (i.b)}: $s_1 < s \le \bar{s}$.
Note that in this case $C_0 > 0$ for any $\theta>\frac{1}{2\gamma-1}$. On the other hand, 
since $\g_* \le s_1 < s$, we have $b_*(s,\g) =\frac{3}{2(1-\gamma)}$.
Therefore, taking $b=\frac{3}{2(1-\gamma)}+\delta$ for $\delta>0$ sufficiently small in \eqref{eq_F1},  from positivity of $C_1$, yields $F_{\gamma,\rho,\nu_c}(b,\theta)<0$.

We have therefore proved that (A) holds for any $s >s_1$ by assuming (i). 

Next, we consider (ii). Note that in this case, by Remark \ref{rem_sstar} and a direct computation,  we have 
\begin{align}
\text{max}(s_1, s_2)  < \g_* \le \bar{s}.
\end{align}
We consider  three subcases:

\noi {\bf Case (ii.a)}: $s > \bar{s}$.
Similarly to (i.a),  $F_{\gamma,\rho,\nu_c}(b,\theta)<0$ holds if we take
$\theta>\frac{1}{2\gamma-1}$ sufficiently close to $\frac{1}{2\gamma-1}$ and $b$ sufficiently large.
\smallskip

\noi {\bf Case (ii.b)}: $\g_* \le  s \le\bar{s}$.
Since $b_*(s,\g)=\frac{3}{2(1-\gamma)}$ and $s_1<s_*\le s$,
taking $b= \frac{3}{2(1-\gamma)}+\delta$ for $\delta>0$ sufficiently small and arguing as in case (i.b)  we obtain $F_{\gamma,\rho,\nu_c}(b,\theta)<0$.
\smallskip

\noi {\bf Case (ii.c)}:  $s_2 <s < \g_*$.
Since $b_*(s,\g) =\frac{3}{2}-s$, the positivity of $C_2,C_3$ guarantees that
by taking $b =\frac{3}{2}-s+\delta_1$ and $\theta=\frac{1}{2\gamma-1}+\delta_2$ for some sufficiently small $\delta_1,\delta_2>0$, then $F_{\gamma,\rho,\nu}(b,\theta)<0$ holds true.

We have therefore proved that (A) holds for any $s >s_2$ by assuming (ii). 

In order to conclude the proof of Lemma \ref{lem_region2}, it remains to prove the ``only if" part of the statement. We argue as follows.

If  $\rho\le  \rho_*$ and $s \le s_1$, then $C_0>0$. Moreover, by \eqref{eq_F1} and  \eqref{eqqqstar}, we have $F_{\gamma,\rho,\nu_c}(b,\theta)>0$ for any $b>\frac{3}{2(1-\gamma)}$. Hence, the assertion (A) does not hold.
Similarly, if $\rho> \rho_*$ and   $s \le s_2$, it follows that $C_0>0$. Furthermore, by \eqref{eq_F2} and \eqref{eqqqstar}, we have $F_{\gamma,\rho,\nu}(b,\theta)>0$ for any $b>\frac{3}{2}-s$ and $\theta >\frac{1}{2\gamma-1}$. Therefore, the assertion (A) does not hold.
\end{proof}

\begin{ackno}\rm

%G.L. was supported by the NSFC (grant no. 12501181). 
D.G. was supported by the European Research Council (grant no.~864138 ``SingStochDispDyn'').
%and by the EPSRC Mathematical Sciences Small Grant
%(grant no. EP/Y033507/1).
%S.L. acknowledged support from the Deutsche Forschungsgemeinschaft (DFG, German Research Foundation) under Germany's Excellence Strategy -- EXC-2047/1 -- 390685813. 
 The authors are very thankful to Prof. Tadahiro Oh for suggesting the problem as well as providing useful feedbacks and comments. K.T. would like to thank the School of Mathematics at the University of Edinburgh for its hospitality, where this
manuscript was prepared during his visit.
%N.T. was partially supported by the ANR
%project Smooth ANR-22-CE40-0017. 

\end{ackno}

%%%%%%%%%%%%%%%%%%%%%%%%%%%%%%%%%%%%%%%%%%%%%%%%%%%%%%%%%%%%%%
%%%%%%%%%%%%%%%%%%%%%%%%%%%%%%%%%%%%%%%%%%%%%%%%%%%%%%%%%%%%%%
%%%%%%%%%%%%%%%%%%%%%%%%%%%%%%%%%%%%%%%%%%%%%%%%%%%%%%%%%%%%%%

%\section{Preliminary lemmas}

%%%%%%%%%%%%%%%%%%%%%%%%%%%%%%%%%%%%%%%%%%%%%%%%%%%%%
%%%%%%%%%%%%%%%% lem_res %%%%%%%%%%%%%%%%%%%%%%%%%%%%%%
%%%%%%%%%%%%%%%%%%%%%%%%%%%%%%%%%%%%%%%%%%%%%%%%%%%%%

\end{document}